\documentclass[a4paper,11pt,twoside]{article}

\RequirePackage[colorlinks,citecolor=blue,urlcolor=blue]{hyperref}
\usepackage[utf8]{inputenc}
\usepackage[UKenglish]{babel}
\usepackage[T1]{fontenc}
\usepackage{charter}
\usepackage{enumerate}
\usepackage{bm}
\usepackage{indentfirst}
\usepackage{xcolor}
\usepackage[top=5.cm, bottom=5.0cm, left=2cm, right=2cm]{geometry}
\usepackage{authblk}
\usepackage{mathtools}
\allowdisplaybreaks
\usepackage{framed}
\usepackage{amsmath}
\usepackage{amsthm}	
\usepackage{amsfonts}
\usepackage{mathrsfs}	
\usepackage{amssymb}
\usepackage{dsfont}
\usepackage{bbm}

\usepackage{enumerate}
\usepackage{cases}
\usepackage{tcolorbox}
\usepackage[shortlabels]{enumitem}
\theoremstyle{plain}
\newtheorem{theorem}{Theorem}[section]
\newtheorem{lemma}[theorem]{Lemma}

\newtheorem{corollary}[theorem]{Corollary}
\theoremstyle{definition}
\newtheorem{definition}[theorem]{Definition}

\theoremstyle{definition}
\newtheorem{remark}[theorem]{Remark}

\numberwithin{equation}{section}
\numberwithin{figure}{section}
 \usepackage[nodayofweek]{datetime}

\newcommand{\Phia}{U}

\newcommand{\pmu}{\partial_{\mu}}
\newcommand{\ptwomu}{\partial^2_{\mu}}

\newcommand{\pnmu}{\partial^n_{\mu}}

\newcommand{\law}[1][] {\mathscr L({#1})}
\newcommand{\nlaw}[1][]{\mathfrak{X}^N_{#1}}
\newcommand{\claw}[1]{\mathscr L{#1}}

\newcommand{\bE}{\mathbb{E}}

\newcommand{\bN}{\mathbb{N}}
\newcommand{\bP}{\mathbb{P}}

\newcommand{\bR}{\mathbb{R}}

\makeatletter
\newcommand{\mylabel}[2]{#2\def\@currentlabel{#2}\label{#1}}
\makeatother

\newenvironment{magetext}{\color{magenta}}{\ignorespacesafterend}

\newcommand{\cC}{\mathcal{C}}
\newcommand{\cD}{\mathcal{D}}

\newcommand{\cF}{\mathcal{F}}
\newcommand{\cG}{\mathcal{G}}
\newcommand{\cH}{\mathcal{H}}
\newcommand{\cI}{\mathcal{I}}

\newcommand{\cK}{\mathcal{K}}
\newcommand{\cL}{\mathcal{L}}
\newcommand{\cM}{\mathcal{M}}

\newcommand{\cO}{\mathcal{O}}
\newcommand{\cP}{\mathcal{P}}

\newcommand{\cV}{\mathcal{V}}
\newcommand{\cW}{\mathcal{W}}
\newcommand{\cX}{\mathcal{X}}

\newcommand{\R}{\mathbb{R}}

\def \eproof{\hbox{ }\hfill$\Box$}

\newcommand{\ud}{\mathrm{d}}

\newcommand{\set}[1]
    {\ensuremath{\{ #1 \}}}
    
\newcommand{\HP}[1] 
    {\ensuremath{\mathscr{H}^{#1}}}

\newcommand{\esp}[1]{\ensuremath{\mathbb{E} \!\! \left[#1\right] }}

\renewcommand{\Xi}[1]{X_{i #1}}

\newcommand{\Tr}[1]{\ensuremath{ \mathrm{Tr}\!\left[ #1 \right]} }

\title{Weak quantitative propagation of chaos via differential calculus on the space of measures }
\author{Jean-Fran\c{c}ois Chassagneux}
\author[2]{Lukasz Szpruch}
\author[2]{Alvin Tse}

\affil[1]{LSPM, Universit\'{e} Paris Diderot}
\affil[2]{School of Mathematics, University of Edinburgh}
\date{}

\begin{document}	
\maketitle

\begin{abstract}
Consider the metric space $(\mathcal{P}_2(\bR^d),W_2)$ of square integrable laws on $\bR^d$ with the topology induced by the 2-Wasserstein distance $W_2$. Let $\Phi: \mathcal{P}_2( \bR^d) \to \bR$ be a function and $\mu_N$ be the empirical measure of a sample of $N$ random variables distributed as $\mu$. The main result of this paper is to show that under suitable regularity conditions, 
we have
\[
|\Phi(\mu) -  \bE\Phi(\mu_N)|= \sum_{j=1}^{k-1}\frac{C_j}{N^j} + O(\frac{1}{N^k}),
\]
for some positive constants $C_1, \ldots, C_{k-1}$ that do not depend on $N$, where $k$ corresponds to the degree of smoothness. We distinguish two cases: a) $\mu_N$ is the empirical measure of $N$-samples from $\mu$; b) $\mu$ is a marginal law of McKean-Vlasov stochastic differential equation in which case $\mu_N$ is an empirical law of marginal laws of the corresponding particle system. The first case is studied using functional derivatives on the space of measures. The second case relies on an It\^{o}-type formula for the flow of probability measures and is intimately connected to PDEs on the space of measures, called the master equation in the literature of mean-field games. We state the general regularity conditions required for each case and analyse the regularity in the case of functionals of the laws of McKean-Vlasov SDEs. Ultimately, this work reveals quantitative estimates of propagation of chaos  for interacting particle systems. Furthermore, we are able to provide weak propagation of chaos estimates for ensembles of interacting particles and show that these may have some remarkable properties.

%

\end{abstract}

\section{Introduction}
The aim of this work is to provide an exact weak error expansion between a (nonlinear) functional $\Phi: \mathcal{P}_2( \bR^d) \to \bR$ of the empirical measure $\mu_N \in \mathcal P_2(\bR^d)$  and its deterministic limit $\Phi(\mu)$, $\mu \in  \mathcal{P}_2( \bR^d)$.  We distinguish two cases:  a) $\mu_N$ is the empirical measure of $N$-samples from $\mu$; b) $\mu$ is the marginal law of a process described by a McKean-Vlasov stochastic differential equation (McKV-SDE), in which case $\mu_N$ is the empirical measure of the marginal laws of the corresponding particle system. 

In the first case where $\mu_N$ is the empirical measure of $N$-samples from $\mu$,  the only interesting case is when the functional $\Phi$ is  non-linear. To provide some context to our results,  one may, for example,  assume that $\Phi$ is Lipschitz continuous with respect to the Wasserstein distance, i.e, there exists a constant $C>0$ such that
\[
|\Phi(\mu) - \Phi(\nu)| \leq \, C W_2(\mu,\nu), \, \qquad \forall  \mu,\nu \in \mathcal P_2(\bR^d)\,,
\]
one could bound $|\Phi(\mu) -  \bE\Phi(\mu_N)|$ by $\bE W_2(\mu,\mu_N)$. Consequently, following \cite{fournier2015rate} or \cite{dereich2013constructive}, the rate of convergence in the number of samples $N$  deteriorates as the dimension $d$ increases. On the other hand, recently,  authors \cite[Lem. 5.10]{delarue2018master} made a remarkable observation that if the functional $\Phi$ is twice-differentiable with respect to the functional derivative (see Section \ref{sec lfd}), then one can obtain a dimension-independent bound for the strong error $\bE|\Phi(\mu) -  \Phi(\mu_N)|^p$, $p\leq 4$, which is of order $O(N^{-1/2})$ (as expected by  CLT). Here, we study a weak error and show that, (see Theorem  \ref{eq:main result with regularity: final}) if $\Phi$ is $(2k+1)$-times differentiable with respect to the functional derivative, then indeed we have 
\[
|\Phi(\mu) -  \bE\Phi(\mu_N)|= \sum_{j=1}^{k-1}\frac{C_j}{N^j} + O(\frac{1}{N^k}),
\]
for some positive and explicit constants $C_1, \ldots, C_{k-1}$ that do not depend on $N$. The result is of independent interest, but is also needed to obtain a complete expansion for the error in particle  approximations of McKV-SDEs that we discuss next. 

The second situation we treat in this work concerns estimates of propagation-of-chaos type \footnote{We would like to remark that our results also cover a situation where the law $\mu$ is induced by a system of stochastic differential equations with random initial conditions that are not of McKean-Vlasov type in which case the samples are i.i.d.}.
Consider a probability space $(\Omega, \mathcal{F}, \bP)$ with a $d$-dimensional Brownian motion $W$. We are interested in the McKean-Vlasov process $ \{X^{s, \xi}_t \}_{t \in [s,T]}$ with interacting kernels $b$ and $\sigma$, starting from a random variable $\xi$, defined by the SDE\footnote{{We assume without loss of generality that the dimensions of $X$ and $W$ are the same because we will not make any non-degeneracy assumption on the diffusion coefficient $\sigma$ in our work. In particular, one dimension of $X$ could be time itself.}} 
\begin{equation} \label{eq:McKeanflow1}
    X^{0, \xi}_t = \xi + \int_0^t  b(X^{0, \xi}_r,   \law[{X^{0,  \xi}_r}]  )  \,dr + \int_0^t \sigma (X^{s, \xi}_r,   \law[{X^{0,  \xi}_r}] )  \,dW_r, \quad t \in [s,T],
\end{equation}
where $\law[X^{s, \xi}_r]$ denotes the law of $X^{s, \xi}_r$ and functions $b=(b_i)_{1 \leq i \leq d}: \bR^d \times \mathcal{P}_2( \bR^d) \to \bR^d$ and $\sigma=(\sigma_{i,j})_{1 \leq i,j \leq  d}: \bR^d \times \mathcal{P}_2( \bR^d)  \to \bR^d \otimes \bR^{{d}}$ satisfy suitable conditions so that there exists a unique weak solution (see e.g. \cite{sznitman1991topics} or, for more up-to-date panorama on research on existence and uniqueness, see \cite{mishura2016existence,hammersley2018mckean,barbu2018nonlinear,de2018well}). McKean-Vlasov SDE \eqref{eq:McKeanflow1} can be derived as a limit of interacting diffusions. Indeed, one can approximate the law $\law[{X^{0,\xi}_{\cdot}}]$ by the empirical measure $  \nlaw[\cdot]:=\frac{1}{N} \sum_{j=1}^N \delta_{\cX^{j,N}_{\cdot}}$  generated by $N$ particles $\{ \cX^{i,N} \}_{1 \leq i \leq N} $ defined as 
\begin{equation} \label{eq particles}
    \cX^{i,N}_t= \xi_i + \int_0^t  b \bigg( \cX^{i,N}_r, \mathfrak{X}^N_r \bigg)  \,dr + \int_0^t  \sigma \bigg( \cX^{i,N}_r,\nlaw[r] \bigg)    \,  dW^i_r, \quad 1 \leq i \leq N, \quad t \in [0,T],
\end{equation}
where $W^{i},$ $1 \leq i \leq N$,  are independent $d$-dimensional Brownian motions and $\xi_i,$ $1 \leq i \leq N$, are i.i.d. random variables with the same distribution as $\xi$. It is well known, \cite[Prop. 2.1]{sznitman1991topics}, that the property of propagation of chaos  is equivalent to weak convergence of measure-valued random variables $\nlaw[t]$ to $\claw(X_{t})$. A common strategy is to establish tightness of $\pi^N = \claw( \nlaw[t]) \in \mathcal P (\mathcal P (\bR^d))$ and to identify the limit by showing that $\pi^N$ converges weakly to $\delta_{\claw(X_{t})}$. This approach does not reveal quantitative bounds we seek in this paper, but is a very active area of research. We refer the reader to \cite{gartner1988mckean,sznitman1991topics,meleard1996asymptotic} for the classical results in this direction and to \cite{jourdain2007nonlinear,bossy2011conditional,fournier2016propagation,mischler2013kac,lacker2018strong} for an  account (non-exhaustive) of recent results. On the other hand, the results on quantitative propagation of chaos  are few and far in between. In the case when coefficients of \eqref{eq:McKeanflow1} depend on the measure component linearly, i.e., are of the form
\[
b(x, \mu) = \int_{\bR^d} B(x,y) \, \mu(dy) , \quad \sigma(x, \mu) =  \int_{\bR^d} \Sigma(x,y) \mu(dy)\,, 
\]
with $B,\Sigma$ being Lipschitz continuous in both variables, it  follows from a simple calculation 
\cite{sznitman1991topics} to see that $W_2(\claw(\cX^{i,N}_t), \claw(X^{0, \xi}_t))=O(N^{-1/2})$. We refer to Sznitman's result as strong propagation of chaos. Note that in this work we treat the case of McKean-Vlasov SDEs with coefficients with general measure dependence. In that case, as explicitly demonstrated in \cite[Ch. 1]{carmona2016lectures}, the rate of strong propagation of chaos deteriorates with the dimension $d$. This is due to the fact that one needs to estimate the difference between the empirical law of i.i.d. samples from $\mu$ and $\mu$ itself using results such as \cite{fournier2015rate} or \cite{dereich2013constructive}. In the special case when the diffusion coefficient is constant, with linear measure dependence on the drift (which lies in some negative Sobolev space), the rate of convergence in the total variation norm has been shown to be $O(1/\sqrt{N})$ in \cite{jabin2018quantitative}. Of course, in a strong setting, $O(1/\sqrt{N})$ is widely considered to be optimal as it corresponds to the size of stochastic fluctuations as predicted by the CLT. In this work, we are interested in weak quantitative estimates of propagation of chaos. Indeed, this new direction of research has been put forward very recently by two independent works \cite[Ch. 9]{kolokoltsov2010nonlinear} and \cite[Th. 2.1]{mischler2015new}. The authors presented novel weak estimates of propagation of chaos  for linear functions in measure, i.e. $\Phi(\mu):=\int_{\bR^d}F(x)\mu(dx)$ with $F:\bR^d \rightarrow \bR$ being smooth. This gives the rate of convergence $O(1/N)$, plus the error due to approximation of the functional of the initial law (see \cite[Lem. 4.6]{mischler2015new} for a discussion of a dimensional-dependent case). While the aim of \cite{mischler2015new} is to establish quantitative propagation of chaos for the Boltzmann's equation, in a spirit of Kac's programme \cite{kac1956foundations,mckean1967exponential}, Theorem 6.1 in \cite[Th. 6.1]{mischler2015new} specialises their result to McKV-SDEs studied here, but only for elliptic diffusion coefficients that do not depend on measure and symmetric Lipschitz drifts with linear measure dependence. The key idea behind both results is to work with the semigroup that acts on the space of functions of measure, sometimes called the lifted semigroup, which can be viewed a dual to the space of probability measures on  $\mathcal P(\bR^d)$ as presented in \cite{mischler2013kac}. A similar research programme, but in the context of mean-field games with a common noise, has been successfully undertaken in \cite{cardaliaguet2015master}. In this work, the authors study the master equation, which is a PDE driven by the Markov generator of a lifted semi-group. They show that existence of classical solution to that PDE is the key to obtain quantitative bounds between an $n$-player Nash system and its mean field limit. Indeed, perturbation analysis of the PDE on the space of measures leads to the weak error being of the order $O(1/N)$. 
 
 In this work we build on these observations, and identify minimal assumptions for the expansion in number of particles $N$ to hold. Next, we verify these assumptions for McKV-SDEs with a general drift and general (and possibly non-elliptic) diffusion coefficients. We also consider non-linear functionals of measure. The main theorem in this paper, Theorem \ref{eq:main result with regularity: final}, states that given sufficient regularity we have
\begin{equation*}
\esp{\Phi(\nlaw[T])} - \Phi (\law[X^{0, \xi}_T])
= \sum_{j=1}^{k-1} \frac{C_j}{N^j} + O(\frac1{N^k}), \label{eq: first expansion eq} 
\end{equation*}
where $C_1, \ldots, C_{k-1}$ are constants that do not depend on $N$.

As  mentioned above, the method of expansion relies heavily on the calculus on 
$(\mathcal{P}_2(\bR^d),W_2)$  and we follow the approach presented by P$.$ Lions in his course at Coll\`{e}ge de France \cite{lions2014cours} (redacted by Cardaliaguet \cite{cardaliaguet2010notes}). To obtain such an expansion, one needs to rely on some smoothness property for the solution of \eqref{eq:McKeanflow1} as it is always the case when one wants to gets error expansion for some approximating procedure (see \cite{talay1990expansion} for a similar expansion on the weak error expansion of SDE approximation with time-discretisation). The important object in our study, similarly to \cite{cardaliaguet2015master}, is the PDE written on the space $[0,T]\times  \cP_2(\R^d)$, which corresponds to the  lifted semigroup and comes from the It\^{o}'s formula of functionals of measures established in \cite{buckdahn2017mean} and \cite{chassagneux2014probabilistic}. Smoothness properties on the functions $\cV^{(m)}$ (see Definition \ref{de class D}) required for expansion \eqref{eq: first expansion eq} to hold are formulated in Theorem \ref{eq:mainresultwithoutregularity}. A natural question is then to identify some sufficient conditions on the SDE coefficients to guarantee the smoothness property of the functions $\cV^{(m)}$. We give one possible answer to this question in Theorem \ref{eq:main result with regularity: final}. In that theorem, we show that if the coefficients of the SDE are smooth, then the 
functions $\cV^{(m)}$ are also smooth enough provided that $\Phi$ is itself smooth. This result, which is expected, comes from an extension of Theorem 7.2 in \cite{buckdahn2017mean} (see Theorem \ref{eq:generalisationmainresult}). 
 
While in the current paper, we assume high order of smoothness of the coefficients of McKV-SDEs and $\Phi$, we anticipate this general approach to be valid under a less regular setting. Indeed, when working with a strictly elliptic setting {with some structural conditions}, the lifted semigroup may be smooth even in the case when drift and diffusion coefficients are irregular. This has been demonstrated in \cite{de2015strong,de2018well}. Similarly, $\Phi$ does not need to be smooth for the lifted semigroup to be differentiable in the measure direction. This has been shown using techniques of Malliavin calculus in \cite{crisan2017smoothing}. Finally, when the underlying equation has some special structure, the more classical approach can be deployed to study weak propagation of chaos property \cite{bencheikh2018bias}. The analysis of irregular cases goes beyond the scope of this paper. 

To sum up, there are three main contributions in this paper. Firstly, the main result (Theorem \ref{eq:main result with regularity: final}) allows us to use Romberg extrapolation to obtain an estimator of $X$ with weak error being in the order of $O(\frac{1}{N^k})$, for each $k \in \bN$. (See Section \ref{eq:rombergremark} for details.) Thus, effectively, a higher-order particle system (in terms of the weak error) can be constructed up to a desired order of approximation. Secondly, the analysis in this paper makes use of the notions of measure derivatives and linear functional derivatives by generalising them to an arbitrary order of differentiation. This is in line with the approach in \cite{crisan2017smoothing}. Some properties (e.g. Lemma \ref{le polynomial growth}) relate the regularity of the two notions of derivatives in measure and might be of an independent interest. In particular, the generalisation of  Theorem 7.2 in \cite{buckdahn2017mean} from second order derivatives in measure to higher order derivatives is proven to be useful in the analysis of McKean-Vlasov SDEs in general. Finally, as a by-product of the weak error expansion, a version of the law of large numbers in terms of functionals of measures is developed in Theorem  \ref{ theorem second order expansion}.

\subsection{Romberg extrapolation and ensembles of particles} \label{eq:rombergremark}
In this section we construct an ensemble particle system in the spirit of  Richardson's extrapolation method \cite{richardson1911ix} that has been studied in the context of time-discretisation of SDEs in \cite{talay1990expansion} and in the context of discretisation of SPDEs in \cite{gyongy2005accelerated}. 

 Let $F:\bR^d \rightarrow \bR$ be a Borel-measurable function and define $\Phi(\mu):=\int_{\bR^d}F(x)\mu(dx)$. Observe that
\[
\bE[\Phi(\nlaw[T])]=\bE \bigg[\frac{1}{N}\sum_{i=1}^{N}F(\cX_T^{i,N}) \bigg] = \bE[F(\cX_T^{1,N})]. 
\]
Hence, the weak error reads  $|\bE[F(X^{0, \xi}_T)] - \bE[F(\cX^{1,N}_T)]|$.  By the result of Theorem \ref{eq:main result with regularity: final}, we can apply the technique of Romberg extrapolation to construct an estimator which approximates $\bE[F(X^{0,\xi}_T)]$ such that the weak error is of the order of $O(1/N^k)$. More precisely, for $k=2$, since $C_1$ is independent of $N$,
$$ \bE F(\cX^{i,N}_T) -\bE[F(X^{0,\xi}_T)]   = \frac{C_1}{N} + O\bigg( \frac{1}{N^{2}} \bigg)$$ 
and
$$ \bE F(\cX^{i,2N}_T) -\bE[F(X^{0,\xi}_T)]   =\frac{C_1}{2N} + O\bigg( \frac{1}{N^{2}} \bigg).$$
Hence,
$$ \Big| \Big( 2 \bE F(\cX^{i,2N}_T) - \bE F(\cX^{i,N}_T) \Big) -  \bE[F(X^{0,\xi}_T)] \Big| = O \Big( \frac{1}{N^2} \Big). $$
For general $k$, we can use a similar method to show that 
$$ \bigg|  \sum_{m=1}^k \alpha_m \bE F(\cX^{i,mN}_T)   - \bE[F(X^{0,\xi}_T)] \bigg| = O \Big( \frac{1}{N^k} \Big), $$ 
where
$$ \alpha_m = (-1)^{k-m} \frac{m^k}{m! (k-m)!}, \quad 1 \leq m \leq k. $$ 

To motivate the study of the weak error expansion we will analyse an estimator that uses  $M$ ensembles of particles. Fix $M\geq 1$. The ensembles are indexed by $j$. For $j\in \{1,\ldots, M \}$, consider 
\begin{equation} \label{eq m particles}
    \cX^{(i,j),N}_t= \xi_{(i,j)} + \int_0^t  b \bigg( \cX^{(i,j),N}_r, \mathfrak{X}^{(j,N)}_r \bigg)  \,dr + \int_0^t  \sigma \bigg( \cX^{(i,j),N}_r,\mathfrak{X}^{(j,N)}_r  \bigg)    \,  dW^{(i,j)}_r, \quad 1 \leq i \leq N, 
\end{equation}
where $\{W^{i}:$ $1 \leq i \leq N\}_{1\leq j \leq M}$ are $M$ independent ensembles each consisting of $N$ $d$-dimensional Brownian motions; and $\{\xi_{(i,j)}: 1 \leq i \leq N\}_{1\leq j \leq M}$  are $M$ independent ensembles each consisting of $N$ i.i.d. random variables with the same distribution as $\xi$. We consider the following estimator 
\[
\frac{1}{M}\sum_{j=1}^M  \sum_{m=1}^k \alpha_m \frac{1}{mN} \sum_{i=1}^{mN}   F(\cX^{(i,j),mN}_T). 
\]
Next we analyse mean-square error\footnote{We look at the mean-square error for simplicity, but a similar computation could be done to verify the Lindeberg  condition and produce CLT with an appropriate scaling.} of this estimator 
\begin{align*}
& \bE \bigg[ \bigg( \bE[F(X_T)] - \frac{1}{M}\sum_{j=1}^M  \sum_{m=1}^k \alpha_m \frac{1}{mN} \sum_{i=1}^{mN}   F(\cX^{(i,j),mN}_T)  \bigg)^2 \bigg] \\  \leq 
&	2\bigg[ \bigg( \bE[F(X_T)]  - \sum_{m=1}^k \alpha_m \bE F(\cX^{1,mN}_T) \bigg)^2 \bigg]  \\
& +   2 \bE \bigg[ \bigg( \bE \bigg[ \sum_{m=1}^k \alpha_m \frac{1}{mN} \sum_{i=1}^{mN}F(\cX^{i,mN}_T) \bigg] - \frac{1}{M}\sum_{j=1}^M  \sum_{m=1}^k \alpha_m \frac{1}{mN} \sum_{i=1}^{mN}   F(\cX^{(i,j),mN}_T)    \bigg)^2 \bigg].     
\end{align*}
The first term on the right-hand side is studied in Theorem \ref{eq:main result with regularity: final} and, provided that the coefficients of \eqref{eq:McKeanflow1} are sufficiently smooth, it converges with order $\cO(N^{-2k})$. Control of the second term follows from the qualitative strong propagation of chaos. Indeed, we write 
\begin{align*}
 \mathbb{V}ar  \bigg[  \frac{1}{M}\sum_{j=1}^M  \sum_{m=1}^k \alpha_m \frac{1}{mN} \sum_{i=1}^{mN}   F(\cX^{(i,j),mN}_T)     \bigg]    
 \\ \leq 2  \mathbb{V}ar  \bigg[  \frac{1}{M}\sum_{j=1}^M  \sum_{m=1}^k \alpha_m \frac{1}{mN} \sum_{i=1}^{mN}   F(X^{(i,j)}_T)   \bigg]    
 &+ 2  \mathbb{V}ar  \bigg[  \frac{1}{M}\sum_{j=1}^M  \sum_{m=1}^k \alpha_m \frac{1}{mN} \sum_{i=1}^{mN}    \bigg( F(\cX^{(i,j),mN}_T) - F(X^{(i,j)}_T)    \bigg) \bigg],
\end{align*}
where $X^{(i,j)}$ denotes the solution of \eqref{eq:McKeanflow1} driven by $W^{i,j}$ with initial data $\xi^{i,j}$. Hence, independence implies that 
\begin{align*}
 \mathbb{V}ar  \bigg[  \frac{1}{M}\sum_{j=1}^M  \sum_{m=1}^k \alpha_m \frac{1}{mN} \sum_{i=1}^{mN}   F(X^{(i,j)}_T)   \bigg]    
\leq \frac{1}{M} \sum_{m=1}^k \alpha_m^2 \frac{1}{mN}  \mathbb{V}ar [ F(X^{(1,1)}_T)]. 
\end{align*}
On the other hand, 
\begin{eqnarray*}
&& \mathbb{V}ar  \bigg[  \frac{1}{M}\sum_{j=1}^M  \sum_{m=1}^k \alpha_m \frac{1}{mN} \sum_{i=1}^{mN}   F(\cX^{(i,j),mN}_T) - F(X^{(i,j)}_T)    \bigg] \nonumber \\
& \leq & \frac{k^2}{M} \sum_{m=1}^k \alpha_m^2 \bE \bigg[ \bigg|\frac{1}{mN} \sum_{i=1}^{mN}   F(\cX^{(i,j),mN}_T) - F(X^{(i,j)}_T)  \bigg|^2\bigg] \nonumber \\
&\leq  & \frac{k^2}{M} \sum_{m=1}^k \alpha_m^2 \frac{1}{mN} \sum_{i=1}^{mN} \bE \big[ \big|  F(\cX^{(i,j),mN}_T) - F(X^{(i,j)}_T)  \big|^2\big],
\end{eqnarray*}
where Jensen's inequality is used.  Using the fact $F$ is Lipschitz continuous and the result on a dimension-free bound for strong propagation of chaos, established in \cite{tse2018antithetic}, there exists a constant $C>0$ with no dependence on $N$ such that  
\[
\bE[ |  F(\cX^{(i,j),mN}_T) - F(X^{(i,j)}_T)  |^2] \leq  \frac{C}{mN}.
\]
Consequently, we have 
\begin{align*}
\bE \bigg[ \bigg( \bE[F(X_T)] - \frac{1}{M}\sum_{j=1}^M  \sum_{m=1}^k \alpha_m \frac{1}{mN} \sum_{i=1}^{mN}   F(\cX^{(i,j),mN}_T)  \bigg)^2 \bigg] \leq C( N^{-2k} + \frac{1}{M} \sum_{m=1}^k \alpha_m^2 \frac{1}{mN}).
\end{align*}
Since there are $M$ ensembles corresponding to the estimator and each ensemble has $k$ sub-particle systems with $mN$ particles each, $m \in \{1, \ldots, k \}$, the total number of interactions 
is $\cC= M \sum_{m=1}^k{(m N)}^{2}$. When we take $N=\epsilon^{-1/k}$ and $M=\epsilon^{-2+1/k}$ the mean-square error is of the order $O(\epsilon^2)$ (since   $\sum_{m=1}^{k}\alpha_m^2 m^{-1}$ is a constant). The corresponding number of interactions $\cC$ is of the order $O(\epsilon^{-2-1/k})$. The message here is that as the smoothness increases, less interactions among particles are needed when approximating the law of McKean-Vlasov SDE \eqref{eq:McKeanflow1}. We would like to stress out again that the dimension of the system does not deteriorate the rate of convergence, in contrast to results presented in the literature \cite{carmona2017probabilistic,fournier2015rate,mischler2013kac}. It is instructive to compare the above computation with a usual mean-square analysis of a single particle system
\begin{eqnarray*}
&& \bE \bigg[ \bigg( \bE[F(X_T)] -  \frac{1}{N} \sum_{i=1}^{N}   F(\cX^{i,N}_T)  \bigg)^2 \bigg] \nonumber \\
& = & \bigg( \bE[F(X_T)] - \bE[F(\cX^{1,N}_T)]   \bigg)^2 
+ \bE \bigg[ \bigg( \bE[F(\cX^{1,N}_T)]  -  \frac{1}{N} \sum_{i=1}^{N}   F(\cX^{i,N}_T)  \bigg)^2 \bigg]. \nonumber 
\end{eqnarray*}
As above, invoking strong propagation of chaos, one can show that the second term is of order $O \Big( N^{-1} \Big)$. That means that there would be no gain to go beyond what we can obtain from the strong propagation of chaos analysis to control the first term. Taking $N=\epsilon^{-2}$ results in mean-square error being of the order $O(\epsilon^2)$ and number of interactions $\cC= N^2 = \epsilon^{-4}$. That clearly demonstrates that working with ensembles of particles leads to an improvement in quantitative properties of  propagation of chaos,  which is interesting on its own but can also be explored when simulating particle systems on the computer. 

\paragraph*{{Notations.}}
\begin{itemize}
\item The $2-$Wasserstein metric is defined by 
\begin{equation*}
W_2(\mu,\nu) := \left( \inf_{\pi \in \Pi(\mu,\nu)}  \int_{\bR^d\times \bR^d} |x-y|^2 \, \pi(dx,dy)	\right)^{\frac{1}{2}}\,,
\end{equation*}
where $\Pi(\mu,\nu)$ denotes the set of {\em couplings} between $\mu$ and $\nu$ i.e. all measures on $\mathscr{B}(\bR^d\times \bR^d)$ 
such that $\pi(B,\bR^d) = \mu(B)$ and $\pi(\bR^d, B) = \nu(B)$ 
for every $B \in \mathscr B(\bR^d)$.
\item Uniqueness in law of  \eqref{eq:McKeanflow1} implies that for any random variables  $\xi,\xi'$ such that $\law[\xi]=\law[\xi']=\mu$, we have
$\law[X^{s, \xi}_t]=\law[X^{s, \xi'}_t]$. Therefore,  we adopt the notation $X^{s, \mu}_t:=X^{s, \xi}_t$ if only the law of the process is concerned.
\item  When the total number $N$ of particles is clear from context, we will often simply write $\cX^i$ for $\cX^{i,N}$. 
\item For any $x, y \in \bR^d$, we denote their inner product by $xy$. Since different measure derivatives lie in different tensor product spaces, we use $| \cdot|$ to denote the Euclidean norm for \emph{any} tensor product space in the form $\bR^{d_1} \otimes \ldots \otimes \bR^{d_{\ell}}$. 
\item  The law of any random variable $Z$ is denoted by $\law[Z]$. For any function $f: \mathcal{P}_2(\bR^d) \to \bR$, its lift $\tilde{f}: L^2 ( \Omega,\mathcal{F}, \bP; \bR^d) \to \bR$ is defined by $\tilde{f}(\xi)= f(\law[{\xi}])$.
\item Also, $(\hat{\Omega},\hat{\cF},\hat{\bP})$ stands for a copy of $({\Omega},{\cF},{\bP})$, which is useful to represent the Lions' derivative of a function of a probability measure. Any random variable $\eta$ defined on $(\Omega,\cF,\bP)$ is represented by $\hat{\eta}$ as a pointwise copy on $(\hat{\Omega},\hat{\cF},\hat{\bP})$.
In the section on regularity, we shall introduce a sequence of copies of $(\Omega,\cF,\bP)$, denoted by $\{ (\Omega^{(n)},\cF^{(n)},\bP^{(n)}) \}_{n}$. As before, any random variable $\eta$ defined on $(\Omega,\cF,\bP)$ is represented by $\eta^{(n)}$ as a pointwise copy on $(\Omega^{(n)},\cF^{(n)},\bP^{(n)})$.
\item For $T>0$, we define the following subsets of $[0,T]^m$, $m\ge 1$:
\begin{equation*}
\Delta^{m}_T := \set{ (t_1,\dots,t_m) \in [0,T]^m \,|\, 0 < t_m <  t_{m-1} < \dots <  t_1 <   T}
\end{equation*}
and
\begin{equation*}
\overline{\Delta^{m}_T} := \set{ (t_1,\dots,t_m) \in [0,T]^m \,|\, 0 \leq  t_m \leq  t_{m-1} < \dots <  t_1 < T}.
\end{equation*}
We often denote $(t_1,\dots,t_m) = \mathbf{t}$ and $(t_1,\dots,t_{m-1}) = \tau$.  We shall also sometimes use the convention $\Delta^0_T :=: \overline{\Delta^0_T} := \emptyset$ for simplicity of notation.  
\item With the above definition, we  denote
\begin{align*}
\int_{\Delta^m_T} f(\mathbf{t}) \ud \mathbf{t} := \int_{0 < t_m < \dots <t_1 < T} f(t_1,\dots,t_m) \ud t_1 \dots \ud t_m.
\end{align*}
\item For any function $f:{\Delta}^{m}_T \rightarrow \R$, we always denote by 
$\partial_t f(t_1,\dots,t_m)$ the partial derivatives of $f$ in the variable $t_m$ at $(t_1,\dots,t_m)$ whenever they exist.
\item { $L^2$ denotes the set of square integrable random variables, $\cH^2$ the set of square-integrable progressively measurable processes $\theta$ such that 
$\left( \int_0^T |\theta_s|^2\right)^\frac12 \in L^2$.
}
\end{itemize}

 \section{Method of weak error expansion}


\subsection{Calculus on the space of measures}
Our method of proof is based on expansion of an auxiliary map satisfying a PDE on  the Wasserstein space. One of the most important tools of the paper is thus the theory of differentiation in measure. 

We make an intensive use of the so-called ``L-derivatives'' and ``linear functional derivatives'' that we recall now, following essentially \cite{cardaliaguet2015master}. We also introduce a higher-order version of this derivative as this is needed in the proofs of our expansion.

\subsubsection{Linear functional derivatives}\label{sec lfd}

A continuous function $\frac{\delta U}{\delta m}: \cP_2(\bR^d) \times \bR^d \to \bR$ is said to be the \emph{linear functional derivative} of $U: \cP_2(\bR^d) \to \bR$, if 
\begin{itemize}
\item for any bounded set $\cK \subset \cP_2(\R^d)$, $y \mapsto \frac{\delta U}{\delta m}(m,y)$ has at most quadratic growth in $y$ uniformly in $m \in \cK$,
\item for any $m, m' \in \cP_2(\bR^d)$,
 \begin{align}\label{eq de first order deriv}
 U(m')- U(m) = \int_0^1 \int_{\bR^d} \frac{\delta U}{\delta m}( (1-s)m + sm',y) \, (m'-m)(dy) \, ds. 
 \end{align}
\end{itemize}

\noindent For the purpose of our work, we need to introduce derivatives at any order $p\ge 1$. 


\begin{definition}
For any $p \ge 1$, the $p$-th order linear functional of the function $U$ is a continuous
function from $\frac{\delta^p U}{\delta m^p}: \cP_2(\R^d)\times (\R^d)^{p-1}\times \R^d \rightarrow \R$ satisfying
\begin{itemize}
\item for any bounded set $\cK \subset \cP_2(\R^d)$, $(y,y') \mapsto \frac{\delta^p U}{\delta m^p}(m,y,y')$ has at most quadratic growth in $(y,y')$ uniformly in $m \in \cK$,
\item for any $m, m' \in \cP_2(\bR^d)$,
\begin{align*}
 \frac{\delta^{p-1} U}{\delta m^{p-1}}(m',y) - \frac{\delta^{p-1} U}{\delta m^{p-1}}(m,y) = \int_0^1 \int_{\bR^{d}} \frac{\delta^p U}{\delta m^p}((1-s)m +sm',y,y') \, (m'-m)( \ud y') \,\ud s,
\end{align*}
provided that the $(p-1)$-th order derivative is well defined.
\end{itemize}
\end{definition}
{
\noindent The above derivatives are defined up to an additive constant via \eqref{eq de first order deriv}.
They are normalised by
\begin{equation}
   \frac{\delta^p U}{\delta m^p}(m,0) = 0\;.
    \label{eq: normalisation linear functional deriatives}
\end{equation} }
We make the following easy observation, which will be useful in the latest parts.
\begin{lemma}\label{le easy expansion}
If $U$ admits linear functional derivatives up to order $q$, then the following expansion holds
\begin{align*}
U(m') - U(m) =
\sum_{p=1}^{q-1}\frac1{p!} & \int_{\bR^{pd}} \frac{\delta^p U}{\delta m^p}(m,\mathbf{y}) \, \set{m'-m}^{\otimes p}( \ud \mathbf{y}) 
\\
&+ \frac1{(q-1)!}\int_0^1 (1-t)^{q-1} \int_{\bR^{qd}} \frac{\delta^q U}{\delta m^q}((1-t)m +tm',\mathbf{y}) \, \set{m'-m}^{\otimes q}( \ud \mathbf{y}) \,\ud t.
\end{align*}
\end{lemma}

\begin{proof}
We define 
\begin{align}
[0,1] \ni t \mapsto f(t) = \Phia \big((1-t)m + t m' \big) =  \Phia \big(m + t (m'-m) \big) \in \R
\end{align}
and apply Taylor-Lagrange formula to $f$ up to order $q$, namely
\begin{align*}
f(1)-f(0) 
&= \sum_{p=1}^{q-1}\frac1{p!} f^{(p)}(0) + \frac1{(q-1)!} \int_0^1(1-t)^{(q-1)} f^{(q)}(t)\ud t.
\end{align*}
It remains to show that
\begin{equation}
  f^{(p)}(t)= \int_{\bR^{pd}} \frac{\delta^p U}{\delta m^p}(m + t (m'-m),\mathbf{y}) \, \set{m'-m}^{\otimes p}( \ud \mathbf{y}), \quad \quad \forall p \in \{0, \ldots, q\}.  \label{eq: pth order linear functional derivative}
\end{equation}
by induction. Since \eqref{eq: pth order linear functional derivative} holds trivially for $p=0$, we suppose that \eqref{eq: pth order linear functional derivative} holds for $p \in \{0, \ldots, q-1\}. $ Then
\begin{eqnarray}
 & & \frac{f^{(p)}(t+h)-f^{(p)}(t)}{h} \nonumber \\
  &= & \frac{1}{h} \bigg[ \int_{\bR^{pd}} \frac{\delta^p U}{\delta m^p}(m + (t+h) (m'-m),\mathbf{y}) \, \set{m'-m}^{\otimes p}( \ud \mathbf{y}) \nonumber \\
  && - \int_{\bR^{pd}} \frac{\delta^p U}{\delta m^p}(m + t (m'-m),\mathbf{y}) \, \set{m'-m}^{\otimes p}( \ud \mathbf{y}) \bigg] \nonumber \\
  & = & \int_{\bR^{pd}} \int_0^1 \int_{\bR^d} \frac{\delta^{p+1} U}{\delta m^{p+1}}(m + (t+sh) (m'-m),\mathbf{y},y') \, (m'-m)(dy') \,ds  \, \set{m'-m}^{\otimes p}( \ud \mathbf{y}). \nonumber  
\end{eqnarray}
Taking $h \to 0$ gives \eqref{eq: pth order linear functional derivative} for $p+1$. This completes the proof. 
\end{proof}

 
\subsubsection{L-derivatives}


The above notion of linear functional derivatives is not enough for our work. We shall need to consider further derivatives in the non-measure argument of the derivative function.

If the function $y \mapsto \frac{\delta U}{\delta m}(m,y)$ is of class $\cC^1$, we consider the \emph{intrinsic} derivative of $U$ that we denote
\begin{align*}
\partial_\mu U(m,y) := \partial_y \frac{\delta U}{\delta m}(m,y)\;.
\end{align*}
The notation is borrowed from the literature on mean field games and corresponds to the notion of ``L-derivative'' introduced by P.-L. Lions in his lectures at Coll?ge de France \cite{lions2014cours}. Traditionally, it is introduced by considering a lift on an $L^2$ space of the function $U$ and using the Fr?chet differentiability of this lift on this Hilbert space. The equivalence between the two notions is proved in \cite[Tome I, Chapter 5]{carmona2017probabilistic}, where the link with the notion of derivatives used in optimal transport theory is also made.

In this context, higher order derivatives are introduced by iterating the operator $\partial_\mu$ and the derivation in the non-measure arguments. Namely, at order $2$, one considers
\begin{align*}
\cP_2(\R^d)\times \R^d \ni   (m,y) \mapsto \partial_y \partial_\mu U(m,y) \text{ and } \
\cP_2(\R^d)\times \R^d\times \R^d \ni   (m,y,y') \mapsto  \partial^2_{\mu} U(m,y,y') \;.
\end{align*}

This leads in particular to the notion of a \emph{fully} $\cC^2$ function that will be of great interest for us (see \cite{chassagneux2014probabilistic}). 
\begin{definition}[Fully $\cC^2$]
A function $U : \cP_2(\R^d) \rightarrow \R$ is fully $\cC^2$ if the following mappings 
\begin{align*}
&\cP_2(\R^d)\times \R^d \ni (m,y) \mapsto \partial_\mu U(m,y)
 \\
 & \cP_2(\R^d)\times \R^d \ni (m,y) \mapsto \partial_y \partial_\mu U(m,y)
 \\
 &\cP_2(\R^d)\times \R^d\times \R^d \ni (m,y,y') \mapsto \partial^2_\mu U(m,y,y')
\end{align*}
are well-defined and continuous for the product topologies.
\end{definition}

Let us observe for later use that if the function $U$ is fully $\cC^2$  and moreover satisfies, for any compact subset $\cK \subset \cP_2(\R^d)$,
\begin{align*}
\sup_{m \in \cK} \int_{\R^d}\left\{|\partial_\mu U(m,y)|^2 +|\partial_y \partial_\mu U(m,y)|^2\right\}\ud m(y) < +\infty\,,
\end{align*}
then it follows from Theorem 3.3 in \cite{chassagneux2014probabilistic} that $U$ can be expanded along the flow of marginals of an It? process. Namely, let $\mu_t = \law[X_t]$ where
\begin{align*}
\ud X_t = b_t \ud t + \sigma_t \ud W_t, \quad X_0 \in L^2,
\end{align*}
with {$b \in \cH_2$ and  $a_t := \sigma_t \sigma_t' \in \cH_2 $,} then
\begin{align}\label{eq chain rule basic}
U(\mu_t) = U(\mu_0) + \int_0^t\esp{\partial_\mu U(\mu_s,X_s)b_s + \frac12 \mathrm{Tr}\set{ \partial_y \partial_\mu U(\mu_s,X_s)a_s } } \ud s.
\end{align}

In order to prove our expansion, we need to iterate the application of the previous chain rule and in order to proceed, we need to use higher order derivatives of the measure functional.

\noindent Inspired by the work \cite{crisan2017smoothing}, for any $k \in \bN$, we formally define  the  higher order derivatives in measures through the following iteration (provided that they actually exist): for any $k \geq 2$, $(i_1, \ldots, i_k) \in \{ 1, \ldots, d \}^k$ and $x_1, \ldots, x_k \in \bR^d$, the function $\partial^k_{\mu} f:\mathcal{P}_2(\bR^d) \times (\bR^d)^{ k} \to (\bR^d)^{\otimes k}$ is defined by
\begin{equation} \bigg( \partial^k_{\mu} f( \mu, x_1, \ldots, x_k) \bigg)_{(i_1, \ldots, i_k)} := \bigg( \pmu \bigg( \Big( \partial^{k-1}_{\mu} f( \cdot, x_{1}, \ldots, x_{k-1}) \Big)_{(i_1, \ldots, i_{k-1})} \bigg)(\mu,x_k) \bigg)_{i_k}, \label{eq:generalformulaintro} \end{equation} 
and its corresponding mixed derivatives in space $\partial^{\ell_k}_{v_k} \ldots \partial^{\ell_1}_{v_1} \partial^k_{\mu} f:\mathcal{P}_2(\bR^d) \times (\bR^d)^{ k} \to (\bR^d)^{\otimes (k+ \ell_1 +\ldots \ell_k)} $ are defined by
\begin{equation} \bigg( \partial^{\ell_k}_{v_k} \ldots \partial^{\ell_1}_{v_1} \partial^k_{\mu} f( \mu, x_1, \ldots, x_k) \bigg)_{(i_1, \ldots, i_k)} := \frac{\partial^{\ell_k}}{\partial x^{\ell_k}_k} \ldots \frac{\partial^{\ell_1}}{\partial x^{\ell_1}_1} \bigg[ \bigg( \partial^k_{\mu} f( \mu, x_1, \ldots, x_k) \bigg)_{(i_1, \ldots, i_k)} \bigg], \quad \ell_1 \ldots \ell_k \in \bN \cup \{0 \}. \label{eq:generalformulamixed} \end{equation}
Since this notation for higher order derivatives in measure is quite cumbersome, we introduce the following multi-index notation for brevity. This notation was first proposed in \cite{crisan2017smoothing}. 
\begin{definition}[Multi-index notation]
Let $ n, \ell$ be non-negative integers. Also, let $\bm{\beta}=(\beta_1, \ldots, \beta_n)$ be an $n$-dimensional vector of non-negative integers. Then we call any ordered tuple of the form $(n,\ell, \bm{\beta})$ or $(n,\bm{\beta})$ a \emph{multi-index}. For a function $f:\R^d \times\cP_2(\R^d) \mapsto \R$, the derivative $D^{(n,\ell, \bm{\beta})} f(x, \mu , v_1, \ldots, v_n)$  is defined as 
$$D^{(n,\ell, \bm{\beta})} f (x, \mu , v_1, \ldots, v_n) :=  \partial^{\beta_n}_{v_n} \ldots \partial^{\beta_1}_{v_1} \partial^{\ell}_x \pnmu f( x, \mu , v_1, \ldots, v_n)$$
if this derivative is well-defined.
For any function $\Phi: \mathcal{P}_2 ( \bR^d) \to \bR$, we define
$$D^{(n,\bm{\beta})} \Phi( \mu , v_1, \ldots, v_n) :=  \partial^{\beta_n}_{v_n} \ldots \partial^{\beta_1}_{v_1} \pnmu \Phi( \mu , v_1, \ldots, v_n), $$ 
if this derivative is well-defined. Finally, we also define the \emph{order} \footnote{ We do not consider `zeroth' order derivatives in our definition, i.e. at least one of $n$, $\beta_1, \ldots, \beta_n$ and $\ell$ must be non-zero, for every multi-index $\big(n, \ell, (\beta_1, \ldots, \beta_n) \big)$.} $ |(n,\ell, \bm{\beta})| $ (resp.  $|(n,\bm{\beta})|$ ) by
\begin{equation} |(n,\ell, \bm{\beta})|:= n+ \beta_1 + \ldots \beta_n + \ell , \quad \quad |(n,\bm{\beta})|:= n+ \beta_1 + \ldots \beta_n .  \label{eq:orderdef}
\end{equation}
\end{definition}

{
As for the first order case, we can establish the following relationship with linear functional derivatives, see e.g. \cite{cardaliaguet2015master} for the correspondence up to order 2,
\begin{equation}
   \partial^n_\mu U(\cdot) =  \partial_{y_n} \frac{\delta }{\delta m} \dots \partial_{y_1} \frac{\delta }{\delta m} U(\cdot) = \partial_{y_n}\dots \partial_{y_1}\frac{\delta^n }{\delta m^n} U(\cdot) \,, \label{eq: connection lions functional derivatives high order} 
\end{equation}
provided one of the two derivatives is well-defined.
}

Next, we deduce the following lemma that will be useful later on.
\begin{lemma} \label{le polynomial growth}
Let $p \ge 1$ and assume that $\partial^p_\mu U \in L^\infty$. Then
\begin{align*}
\bigg|\frac{\delta^p U}{\delta m^p} (m,y_1,\dots,y_p)\bigg| \le C_p(|y_1|^p + \dots + |y_p|^p)\;,
\end{align*} 
for some constant $C_p>0$.
\end{lemma}
\begin{proof} We sketch the proof by induction in dimension one, for ease of notation.
Let $p\ge 1$. 
\\
First, we compute that
\begin{align*}
\frac{\delta^p U}{\delta m^p}(m,y_1,\dots,y_p) = 
\frac{\delta^p U}{\delta m^p}(m,y_1,\dots,y_{p-1},0) + \int_0^1 
\partial_{t_p} \bigg[ \frac{\delta^p U}{\delta m^p}(m,y_1,\dots,t_p y_p) \bigg] \ud t_p \;.
\end{align*}
Let $\partial_{x_p}$ denote the derivative w.r.t. the $p$th component of the spatial variables.  From the convention of normalisation \eqref{eq: normalisation linear functional deriatives}, we simply obtain that
\begin{align*}
\frac{\delta^p U}{\delta m^p}(m,y_1,\dots,y_p) = y_p \int_0^1 
\partial_{x_p} \bigg[ \frac{\delta^p U}{\delta m^p}(m,y_1,\dots,t_p y_p) \bigg] \ud t_p \;.
\end{align*}
Let $k < p$ and assume that
\begin{align}
&\frac{\delta^p U}{\delta m^p}(m,y_1,\dots,y_p) =  \nonumber 
\\
&
y_{p-k} \dots  y_p\int_{[0,1]^{k+1}}
\partial_{x_{p-k}}\dots\partial_{x_p}\bigg[ \frac{\delta^p U}{\delta m^p}(m,y_1,\dots,y_{p-k-1},t_{p-k} y_{p-k},\dots,t_p y_p) \bigg] \; \ud t_{p-k}  \dots  \ud t_{p} . \label{eq: normalisation induction proof}
\end{align}
Then, observing that
\begin{align*}
\partial_{x_{p-k}}\dots\partial_{x_p}\bigg[ \frac{\delta^p U}{\delta m^p} (m,y_1,\dots,y_{p-k-2},0,t_{p-k} y_{p-k},\dots,t_p y_p) \bigg]= 0,
\end{align*}
 we recover 
\begin{align*}
&\frac{\delta^p U}{\delta m^p}(m,y_1,\dots,y_p) =  
\\
&
y_{p-k-1} \ldots  y_p\int_{[0,1]^{k+2}} 
\partial_{x_{p-k-1}}\dots\partial_{x_p} \bigg[ \frac{\delta^p U}{\delta m^p}(m,y_1,\dots,y_{p-k-2},t_{p-k-1} y_{p-k-1},\dots,t_p y_p) \bigg] \ud t_{p-k-1} \dots \ud t_p .
\end{align*}
Setting $k=p-1$ in \eqref{eq: normalisation induction proof}, we then obtain 
\[
\frac{\delta^p U}{\delta m^p}(m,y_1,\dots,y_p) =  
y_{1}  \dots  y_p\int_{[0,1]^{p}}
\partial^p_{\mu} U(m,t_1y_1,\dots,t_p y_p) \, \ud t_{1} \dots \ud t_p .
\]
The proof is concluded by invoking the boundedness assumption of $\partial^p_{\mu} U$ along with Young's inequality. 

\end{proof}

\subsection{Weak error expansion along dynamics}
To state our expansion for the dynamic case, we will need some notion of smoothness given in the following definition.
\begin{definition} \label{de class D} Let $m$ be a positive integer.
A function $U:\overline{\Delta^m_T} \times \cP_2(\R^d) \rightarrow \R$ is of class $\cD(\Delta^m_T)$ if the following conditions hold:
\begin{enumerate}[i)]
\item \textcolor{black}{$U$ is jointly continuous on $\Delta^m_T \times  \cP_2(\R^d)$}.
\item For all $\mathbf{t} \in \Delta^m_T$, $U(\mathbf{t},\cdot)$ is \emph{fully} $C^2$.
\item Let $L$ be a positive constant. For all $\mathbf{t} \in \Delta^m_T$ and $\xi \in L^2(\R^d)$,
\begin{align*}
\esp{|\partial_\mu U(\mathbf{t},\law[\xi])(\xi)|^2 +|\partial_\upsilon \partial_\mu U(\mathbf{t},\law[\xi])(\xi)|^2
+ |\partial^2_\mu U(\mathbf{t},\law[\xi])(\xi,\xi)|^2}
\le
L.
\end{align*}
\item \begin{itemize}
\item $m=1$: $s \mapsto U(s,\mu)$ is continuously differentiable on $(0,T)$.
\item $m>1$: for all $(\tau_1,\dots,\tau_{m-1}) \in \Delta^{m-1}_T$ and all $\mu \in \cP_2(\R^d)$, the function
$$ (0,\tau_{m-1}) \ni s \mapsto U((\tau_1,\dots,\tau_{m-1},s),\mu) \in \R$$
is continuously differentiable on $(0,\tau_{m-1})$.
\end{itemize}
\item The functions 
\begin{align*}
&\Delta^m_T  \times L^2(\R^d) \ni (\mathbf{t},\xi) \mapsto \partial_t U(\mathbf{t},\law[\xi])(\xi) \in L^2(\R^d)
\\
&\Delta^m_T  \times L^2(\R^d) \ni (\mathbf{t},\xi) \mapsto \partial_\mu U(\mathbf{t},\law[\xi])(\xi) \in L^2(\R^d)
\\
&\Delta^m_T  \times L^2(\R^d) \ni (\mathbf{t},\xi) \mapsto \partial_\upsilon \partial_\mu U(\mathbf{t},\law[\xi])(\xi) \in L^2(\R^{d\times d})
\\
&\Delta^m_T \times L^2(\R^d) \ni (\mathbf{t},\xi) \mapsto  \partial^2_\mu U(\mathbf{t},\law[\xi])(\xi,\xi) \in L^2(\R^{d\times d})
\end{align*}
are continuous.

\vspace{5mm}

\end{enumerate}
\end{definition}

We define recursively the functions $\Phi^{(m)}$, $\cV^{(m)}$, $1 \le m \le k$, that are used to prove the expansion.
\begin{definition}\label{de expansion function}
\begin{enumerate}[i)]
\item For $m=1$, we set $\Phi^{(0)}= \Phi$ 
and define
$\cV^{(1)}:[0,T]\times \cP_2(\R^d) \rightarrow \R$ by
$$\cV^{(1)}(t,\mu) := \cV (t,\mu) =  \Phi^{(0)}(\law[X^{t,\mu}_T]) = \Phi(\law[X^{t,\mu}_T])\;.$$ 
Assuming that $\cV^{(1)}$ belongs to the class $\cD(\Delta^1_T)$,
we set $\Phi^{(1)}:(0,T)\times \cP_2(\R^d) \rightarrow \R$ as
\begin{align}\label{eq de Phi 1}
\Phi^{(1)}(t,\mu) &:= \int_{\bR^d} \Tr{ \partial^2_{\mu} {\cV}^{(1)} (t, \mu) (x,x)  a(x, \mu)   } \, \mu(dx).
\end{align}
\item For $1 < m \le k$, we define
$\cV^{(m)}:\overline{\Delta^m_T} \times \cP_2(\R^d) \rightarrow \R$ by
\begin{align*}
\cV^{(m)}((\tau,t),\mu) := \Phi^{(m-1)}(\tau,\law[X^{t,\mu}_{\tau_{m-1}}]) , \quad \quad \tau \in {\Delta^{m-1}_T}.
\end{align*}
Assuming that $ {\cV}^{(m)}$ belongs to the class $\cD(\Delta^m_T)$,
we set $\Phi^{(m)}:\Delta^m_T \times \cP_2(\R^d) \rightarrow \R$ as
\begin{align*}
\Phi^{(m)}(\bm{t},\mu) &:= \int_{\bR^d} \Tr{ \partial^2_{\mu} {\cV}^{(m)} (\bm{t}, \mu) (x,x)  a(x, \mu)   } \, \mu(dx).
\end{align*}
\end{enumerate}
\end{definition}



A key point in our work is to show that the previous definition is licit under some assumptions on the coefficient functions $b,\sigma$ and $\Phi$ (Theorem \ref{eq:main result with regularity} and Theorem \ref{eq:main result with regularity: final}). 

Before we proceed we state the following assumptions
\begin{enumerate}[label=\textnormal{(\arabic*)}]
    \item[\mylabel{eq:Lip}{(Lip)}]   
      $b$ and  $\sigma$  are Lipschitz continuous with respect to the Euclidean norm  and the  $W_2$  norm.  
\end{enumerate} 
\begin{enumerate}[label=\textnormal{(\arabic*)}]  
      \item[\mylabel{eq:UB}{(UB)}]  There exists $L>0$ such that $| \sigma (x, \mu) | \leq L$, for every $x \in \bR^d $ and $\mu \in \cP_2(\bR^d)$.
\end{enumerate}
It will become apparent from the proofs that when working only with \ref{eq:Lip}, higher order integrability conditions would need to be stated in Definition 
 \eqref{de class D}. We refrain from this extension and assume \ref{eq:UB} to improve readability of the paper, but encourage a curious reader to perform this simple extension.    

We begin with the following technical lemma.
\begin{lemma} \label{le toolbox expansion} Assume \ref{eq:Lip} and  \ref{eq:UB}. Let $m$ be a positive integer and $f: \cP_2(\R^d) \rightarrow \R$ be a continuous function.
Consider $U:\overline{\Delta^m_T} \times \cP_2(\R^d) \rightarrow \R$ given by $U((\tau,t),\mu) := f([X^{t,\mu}_{\tau_{m-1}}])$  \color{black} and set $U_{\tau} (t, \mu) = U((\tau,t),\mu)$, where $ \tau \in {\Delta^{m-1}_T}$. If $U$ is of class $\cD(\Delta^m_T)$, then the following statements hold:
\begin{enumerate}[(i)]
\item $U_\tau$ satisfies on $(0,\tau_{m-1})\times \cP_2(\R^d)$ the following PDE 
\begin{equation} \label{eq pde measure}
      \partial_s U_{\tau}(s, \mu) + \int_{\bR^d}\big[ \partial_{\mu} U_{\tau}(s, \mu) (y) b( y, \mu)  + \frac{1}{2} \text{Tr} \big( \partial_v \partial_{\mu} U_{\tau}( s,\mu) ( y) a( y, \mu) \big) \big] \mu(dy)=0,  \\
\end{equation} with terminal condition $U_\tau(\tau_{m-1},\cdot)=f(\tau,\cdot)$, 
where $a=(a_{i,k})_{1 \leq i,k \leq  d }: \bR^d \times \mathcal{P}_2(\bR^d) \to \bR^d \otimes \bR^d$ denotes the diffusion operator
$$ a_{i,k} (x,\mu) := \sum_{j=1}^q  \sigma_{i,j} (x,\mu) \sigma_{k,j} (x,\mu), \quad \quad \forall x \in \bR^d, \quad \forall \mu \in \mathcal{P}_2(\bR^d). $$ 
\item $U_\tau$ can be expanded along the flow of random measure associated to the particle system \eqref{eq particles} as follows, for all $0\le t \le \tau_{m-1}$,
\begin{align}\label{eq empirical expansion}
U_\tau(t,\nlaw[t])= U_\tau(0,\nlaw[0])
+ 
\frac1{2N}\int_0^t  \int_{\R^d} 
\mathrm{Tr}\!\left[a(\upsilon,\nlaw[s] ) \partial^2_{\mu} {U_\tau}(s,\nlaw[s])(\upsilon,\upsilon) \right]
\, \nlaw[s]\!(d\upsilon) \, \ud s + M^N_t\;,
\end{align}
where $M^N$ is a square integrable martingale with $M^N_0 = 0$.
\end{enumerate}
\end{lemma}

\begin{proof} (i) By the flow property, we observe that the function $[0,\tau_{m-1}) \ni s \mapsto U_\tau(s,\law[X^{0,\xi}_s]) \in \R$ is constant. Indeed, 
$U_\tau(s,\law[X^{0,\xi}_s]) = f(\law[X^{s,X^{0,\xi}_s}_\tau]) = f(\law[X^{0,\xi}_\tau]).$ Applying the chain rule in both time and measure arguments between $t$ and $t+h$, we get
\begin{eqnarray}
0 & = &  \int_0^h  \partial_tU_\tau(s,\law[X^{0,\xi}_s]) + \esp{\partial_{\mu} {U}_\tau (s,\law[X^{0,\xi}_s])( X^{0,\xi}_s )b( X^{0,\xi}_s, \law[X^{0,\xi}_s])}  \nonumber \\
&& +\frac12 \Tr{a(X^{0,\xi}_s,\law[X^{0,\xi}_s] ) \partial_{\upsilon} \partial_{\mu} {U}_\tau(s,\law[X^{0,\xi}_s])(X^{0,\xi}_s)}   \, ds. \nonumber 
\end{eqnarray}
Dividing by $h$ and letting $h \rightarrow 0$ allows to recover the first claim.
\\
(ii) To recover the expansion, we use the known strategy of considering finite dimensional projection of $U$. Namely, for a fixed number of particles $N$, we define
\begin{align*}
u(t,x_1,\dots,x_N) := U_\tau(t, \frac1N \sum_{i=1}^N \delta_{x_i})\,.
\end{align*}
From Definition \ref{de class D}(ii), (iv) and (v), we have that $u$ is $C^{1,2}([0,\tau)\times(\R^d)^n)$ (see Proposition 3.1 in \cite{chassagneux2014probabilistic}. Recalling the link between the derivatives of $u$ and $U$ (again see Proposition 3.1 in \cite{chassagneux2014probabilistic}), we can apply the classical Ito's formula to 
 $t \mapsto U_\tau(t,\nlaw[t]) = u(t,\cX^1_t,\dots,\cX^N_t)$ to get
 \begin{align} 
&U_\tau(t,\nlaw[t])  = U(0,\nlaw[0]) +  
\frac1N  \sum_{i=0}^N \int_0^t \partial_{\mu} U_\tau(s,\nlaw[s])(\cX^i_s) \sigma(\cX^i_s,\nlaw[s]) \ud W^i_s 
=:M^N_t
\label{eq martingale term}
\\
&+ \!\! \int_0^t \!\! \left(\partial_t U(s, \nlaw[s]) + \int_{\R^d}  \!\! \left \{ \partial_{\mu} {U}_\tau (s, \nlaw[s] )( \upsilon ) b  (\upsilon, \nlaw[s]  ) 
+\frac12 \Tr{a(\upsilon,\nlaw[s] ) \partial_{\upsilon} \partial_{\mu} {U}_\tau(s,\nlaw[s])(\upsilon)}
\right \}
\, \nlaw[s]\!(d\upsilon) \right ) \ud s \label{eq pde term}
\\
& + \frac1{2N} \int_0^t  \int_{\R^d} 
\Tr{a(\upsilon,\nlaw[s] ) \partial^2_{\mu} {U}_\tau(s,\nlaw[s])(\upsilon,\upsilon) }
\, \nlaw[s]\!(d\upsilon) \, \ud s \,. \label{eq remaining term}
\end{align}
We first note that the term in \eqref{eq pde term} is precisely \eqref{eq pde measure} evaluated at $(s,\nlaw[s])$ and  is thus equal to zero. We now study the local martingale term $M^N$ in \eqref{eq martingale term}. We simply compute
\begin{align*}
\esp{|M^N_t|^2} &= \frac1{N^2}\sum_{i=1}^N \int_0^t \esp{|\partial_\mu U_\tau(s,\nlaw[s])(\cX^i_s) \sigma(\cX^i_s,\nlaw[s])|^2} \ud s \;
\\
& \le \frac{C}{N} \int_0^t \esp{\int_{\R^d} |\partial_{\mu} {U}_\tau (s, \nlaw[s] )( \upsilon )|^2 \, \nlaw[s] (d \upsilon) } \ud s
\;,
\end{align*}
where we have used \ref{eq:UB}.
Using Definition \ref{de class D}(iii), we have 
\begin{align*}
\int_{\R^d} |\partial_{\mu} {U}_\tau (s, \nlaw[s] )( \upsilon )|^2 \, \nlaw[s](d \upsilon) \le C,
\end{align*}
which concludes that $M^N$ is a square integrable martingale.

\end{proof}

  \begin{theorem}[Weak error expansion: dynamic case] \label{eq:mainresultwithoutregularity}
Assume \ref{eq:Lip} and  \ref{eq:UB}. Suppose that Definition \ref{de expansion function} is well-posed for $m \in \{1, \ldots, k \}$.
Then the weak error in the particle approximation can be expressed as

\begin{equation}
\esp{\Phi(\nlaw[T])} - \Phi (\law[X^{0, \xi}_T])
= \sum_{j=0}^{k-1} \frac1{N^j}\left(C_j +  \cI^N_{j+1} \right) + O(\frac1{N^k}), \label{eq: first expansion eq} 
\end{equation}
where $C_0 := 0 \, $  and 
\begin{align*}
C_m := \int_{\Delta^m_T}\Phi^{(m)}(\mathbf{t},\law[X^{0,\xi}_{t_m}]) \ud \mathbf{t}\;, \quad \quad m \in \{1, \ldots, k-1 \},
\end{align*}
and $\cI^N_1 :=  \esp{ {\cV}(0,\nlaw[0]) - {\cV}(0,\law[\xi])} \,$ and
\begin{align*}
\cI^N_{m} := \int_{\Delta^{m-1}_T}\left(\esp{ \cV^{(m)}((\tau,0),\nlaw[0])} - \cV^{(m)}((\tau,0),\law[\xi]) \right) \ud \tau, \; \quad \quad m \in \{2, \ldots, k \}.
\end{align*}

\end{theorem}

\begin{proof}
Part 1: We first check that the constants $(C_m,\cI^N_{m+1})_{0\le m \le k-1}$ are well defined.
\\
For $1 \le m \le k-1$, we first show that the function
\begin{align*}
\Delta^m_T\times \cP_2(\R^d)\ni (\mathbf{t},\mu)\mapsto \Phi^{(m)}(\mathbf{t},\mu) 
 \in \R
\end{align*}
is continuous. Indeed, let $(\mathbf{t}_n,\mu_n)_n$ be a sequence converging to $(\mathbf{t},\mu)$ in the product topology. Then  there exists a sequence $(\xi_n)$ of random variable such that $\law[\xi_n]=\mu_n$ converging to $\xi$ with law $\mu$ in $L^2$.
By continuity of $\sigma$, Definition \ref{de expansion function} and Definition \ref{de class D}(v),
$$ 
\Gamma_n := \Tr{ \partial^2_{\mu} {\cV}^{(m)} (\mathbf{t}_n, \law[\xi_n]) (\xi_n,\xi_n)  a(\xi_n, \law[\xi_n]) }
\rightarrow 
\Tr{ \partial^2_{\mu} {\cV}^{(m)} (\mathbf{t}, \law[\xi]) (\xi,\xi)  a(\xi, \law[\xi]) } =: \Gamma $$
in probability. Next, {since $\sigma$ is bounded},
\begin{equation*}
\esp{\Big|\Tr{ \partial^2_{\mu} {\cV}^{(m)} (\mathbf{t}_n, \law[\xi_n]) (\xi_n,\xi_n)  a(\xi_n, \law[\xi_n]) } \Big|^2 }
\le
C \esp{\Big|\partial^2_{\mu} {\cV}^{(m)} (\mathbf{t}_n, \law[\xi_n]) (\xi_n,\xi_n) \Big|^2}
 \leq C,
\end{equation*}
where the last inequality follows from the fact that ${\cV}^{(m)}$ is of class $\cD(\Delta^m_T)$, by Definition \ref{de class D}(iii).  By de La Vall\'{e}e Poussin Theorem, the previous computation shows that $(\Gamma_n)$ is uniformly integrable  and thus $\Phi^{(m)}(\mathbf{t}_n, \mu_n)=\esp{\Gamma_n} \rightarrow \esp{\Gamma}= \Phi^{(m)}(\mathbf{t}, \mu)$. Observing that $\Delta^m_T \ni \mathbf{t} \mapsto (\mathbf{t},\law[X^{0,\xi}_{t_m}]) \in \Delta^m_T\times \cP_2(\R^d)$ is continuous, we conclude that $ \Delta^{m}_T \ni \mathbf{t} \mapsto \Phi^{(m)}(\mathbf{t},\law[X^{0,\xi}_{t_m}]) \in \R$ is also continuous (hence measurable) and therefore $C_m$  is well-defined. 

Hence, by the definition of $\cV^{(m)}$, for each $\mu \in \cP_2(\R^d)$, the function
$$ \Delta^{m-1}_T \ni \tau \mapsto \cV^{(m)}((\tau,0),\mu) \in \R$$
is continuous. Also, by the previous argument along with Definition \ref{de class D}(iii), we can see that $\Phi^{(m)}$ is uniformly bounded. Therefore, the function
$$  (\tau,\mu) \mapsto \cV^{(m)}((\tau,0),\mu) \in \R$$
is also uniformly bounded. By the dominated convergence theorem, the function 
$$ \Delta^{m-1}_T \ni \tau \mapsto \esp{\cV^{(m)}((\tau,0),\nlaw[0])} \in \R$$
is continuous. This shows that $\cI^N_m$ is well-defined.

\vspace{4pt}
\noindent Part 2:  We now proceed with the proof of the expansion, which is done by induction on $m$.\\ 
Base step: We decompose the weak error  as
\begin{align}\label{eq error decomp proof}
\esp{\Phi (\nlaw[T])}  - \Phi (\law[X^{0, \xi}_T]) =  \esp{ {\cV}(T,\nlaw[T]) - {\cV}(0,\nlaw[0]) } 
+ 
\big( \esp{{\cV}(0,\nlaw[0])} - {\cV}(0,\law[\xi]) \big). 
\end{align}
Applying Lemma \ref{le toolbox expansion}(ii) for the first term in the right-hand side and taking expectation on both side, we obtain that
%
%
\begin{align*}
 \esp{ \Phi (\nlaw[T])  - \Phi (\law[X^{0, \xi}_T]) } 
=
 \esp{ {\cV}(0,\nlaw[0]) - {\cV}(0,\law[\xi]) + \frac1{2N} \int_0^T  \int_{\R^d} 
\mathrm{Tr}\!\left[a(\upsilon,\nlaw[s] ) \partial^2_{\mu} \cV(s,\nlaw[s])(\upsilon,\upsilon) \right]
\, \nlaw[s]\!(d\upsilon) \,  \ud s}.
\end{align*}
Recalling 
the definition of $\Phi^{(1)}$  in \eqref{eq de Phi 1}, we get
\begin{align*}
 \esp{ \Phi (\nlaw[T])  - \Phi (\law[X^{0, \xi}_T]) } 
=
 \esp{ {\cV}(0,\nlaw[0]) - {\cV}(0,\law[\xi]) } +  \frac1N\int_0^T\esp{ \Phi^{(1)}(t_1,\nlaw[t_1]) } \ud t_1 \;.
\end{align*}
From Part 1, we know that $\Phi^{(1)}$ is uniformly bounded and thus
$\int_0^T\esp{ \Phi^{(1)}(t_1,\nlaw[t_1]) } \ud t_1 < C$, where $C>0$ does not depend on $N$. This proves the induction for the base step.

\noindent Induction step: Assume that for $1< m < k$,
 \begin{align*}
 \esp{ \Phi (\nlaw[T])  - \Phi (\law[X^{0, \xi}_T]) } 
& = \sum_{j=0}^{m-1} \frac1{N^j}\left( \cI^N_{j+1} + C_j \right)
+ \frac1{N^m}\int_{\Delta^m_T}\esp{\Phi^{(m)}(\mathbf{t},\nlaw[t_m])} \ud \mathbf{t}.
\end{align*}
Then, we observe that
\begin{eqnarray*}
&& \esp{\Phi^{(m)}(\mathbf{t},\nlaw[t_m]) - \Phi^{(m)}(\mathbf{t},\law[X^{0,\xi}_{t_m}])} \\
&=&  \esp{ \cV^{(m+1)}((\mathbf{t},t_m),\nlaw[t_m]) -\cV^{(m+1)}((\mathbf{t},0),\nlaw[0]) + \cV^{(m+1)}((\mathbf{t},0),\nlaw[0]) - \cV^{(m+1)}((\mathbf{t},0),\law[\xi]) },
\end{eqnarray*}
which leads to
 \begin{eqnarray}
&&  \esp{ \Phi (\nlaw[T])  - \Phi (\law[X^{0, \xi}_T]) }  \nonumber \\
& = & \sum_{j=0}^{m} \frac1{N^j}\cI^N_j + \sum_{j=1}^{m-1} \frac1{N^j}C_j
+ \frac1{N^m}\int_{\Delta^m_T}\esp{ \cV^{(m+1)}((\mathbf{t},t_m),\nlaw[t_m]) -\cV^{(m+1)}((\mathbf{t},0),\nlaw[0]) } \ud \mathbf{t}. \label{eq temp exp gene}
\end{eqnarray}
Applying Lemma \ref{le toolbox expansion}(ii) to  $\cV^{(m+1)}(\mathbf{t},\cdot)$, we obtain that
\begin{eqnarray}
&& \esp{ \cV^{(m+1)}((\mathbf{t},t_m),\nlaw[t_m]) \!-\!\cV^{(m+1)}((\mathbf{t},0),\nlaw[0]) } \nonumber \\
& = &
\frac1{2N}\esp{\int_0^{t_m} \!\!\! \int_{\R^d} \!\!\Tr{\partial^2_\mu \cV^{(m+1)}(({\mathbf{t}},t_{m+1}),\nlaw[t_{m+1}])(\upsilon,\upsilon) a(v,\nlaw[t_{m+1}])} \, \nlaw[t_{m+1}](d\upsilon) \, \ud t_{m+1} } . \nonumber 
\end{eqnarray} 
Inserting this back into \eqref{eq temp exp gene}, we get
\begin{align*}
 \esp{ \Phi (\nlaw[T])  - \Phi (\law[X^{0, \xi}_T]) } 
& = \sum_{j=0}^{m} \frac1{N^j}\cI^N_j + \sum_{j=1}^{m} \frac1{N^j}C_j
+ \frac1{N^{m+1}}\int_{\Delta^{m+1}_T}\esp{ \Phi^{(m+1)} (\mathbf{t},\nlaw[t_{m+1}]) } \ud \mathbf{t}\,.
\end{align*}
The proof is concluded by observing that $\int_{\Delta^{m+1}_T}\esp{ \Phi^{(m+1)} (\mathbf{t},\nlaw[t_{m+1}]) } \ud \mathbf{t}<C$, due to the uniform boundedness of $\Phi^{(m+1)}$ given in Part 1. 

  \end{proof}

%
%

\subsection{Weak error expansion for the initial condition}

Assuming enough smoothness of the  functions $\cV^{(m)}$, we can take care of the terms $\cI^N_m$ appearing in the previous theorem, which are error made at time $0$. The following weak error analysis relies on the notion of linear functional derivatives. We first start by studying the weak error generated between the evaluation of the function at a measure and its empirical measure counterparts. We prove two results: one dealing mainly with low order expansion and the order one, available at any order.

\vspace{4pt}
\noindent The main assumption we work with relates to the couple $(U,m)$, where $U$ is a function with domain $\cP_2(\bR^d)$.
\begin{enumerate}[label=\textnormal{(\arabic*)}]
    \item[\mylabel{as linear derivative}{(p-LFD)}] The $p$th order linear functional derivative of $\Phia$ exists and is continuous and that 
 for any family $(\xi_i)_{1\le i \le p}$ of random variable identically distributed with law $m$ the following holds
\begin{align*}
\esp{\sup_{\nu \in \cP_2(\R^d)} \left|\frac{\delta^{p} \Phia}{\delta m^{p}}(\nu,\xi_1,\dots,\xi_p)\right|} \le L_{(U,m)}\;,
\end{align*}    
for some positive constant $L_{(U,m)}$.
\end{enumerate}

We first make the following observation regarding assumption \ref{as linear derivative}, that will be of later use.
\begin{remark} \label{re poly growth}
(i) Lemma \ref{le polynomial growth} states that
\begin{equation} \bigg| \frac{\delta^{p} \Phia}{\delta m^{p}}(m)(y_1, \ldots, y_p) \bigg| \leq C \big( |y_1|^p + \ldots +  |y_p|^p  \big),  \label{eq: poly growth second expansion} \end{equation}
for every $m \in \cP_2(\bR^d)$, for every $y_1, \ldots, y_p \in \bR^d$, and for some $C>0$.
This means that for any $\mu \in \cP_p(\R^d)$, the couple $(U,\mu)$ satisfies \ref{as linear derivative}.  
This polynomial growth condition is motivated by our example of application, stated in Section \ref{subse regular expansion}, that relies on the smoothness of the coefficients.
\\
(ii) The following simple example of measure functional shows that the above condition is reasonable to consider: For any \emph{bounded} smooth  function $b:\R \to \R$, we set
$
\Psi(m) := b\left(\int x \ud m(x)\right).
$
The linear derivative functional of order $p$ can then be computed by induction, using the normalisation convention \eqref{eq: normalisation linear functional deriatives}, to obtain that
$$\frac{\delta^p \Psi}{\delta m^p}(m,y_1,\dots,y_p) = y_1\dots  y_p \, \,  b^{(p)}\left(\int x \ud m(x)\right)\;,$$
which easily relates to \eqref{eq: poly growth second expansion}\;.
\end{remark}

\begin{theorem}\label{th expansion initial condition} Let $(\xi_i)_{1\le i \le N}$ be i.i.d. random variables with law $\mu \in \cP_2(\bR^d)$. The following statements hold:
\begin{enumerate}[(i)]
\item Let \ref{as linear derivative} hold with $p\in\{1,2\}$ for $(U,\mu)$. Then
\[
\esp{\Phia \bigg(\frac1N \sum_{i=1}^N \delta_{\xi_i} \bigg)} - \Phia(\mu) = O(\frac1N)\;.
\]
\item Let \ref{as linear derivative} hold with $p\in\{1,2,3,4\}$  for $(U,\mu)$. Suppose that $\mu \in \cP_4(\bR^d)$. Then
\begin{align*}
\esp{\Phia \bigg(\frac1N \sum_{i=1}^N \delta_{\xi_i} \bigg)} - \Phia(\mu) &= \frac1{2N}
\esp{\int_{\bR^d} \frac{\delta^2 \Phia}{\delta m^2}(\mu)(\tilde{\xi}_1,y)(\delta_{\tilde{\xi}_1}-\delta_{\xi_1})(dy)}
+ O(\frac1{N^2})\;,
\end{align*}
where $\tilde{\xi}_1\sim \mu$ and is independent of $(\xi_i)_{1\le i \le N}$. 
\end{enumerate}
\end{theorem}

\begin{proof} Let $\mu_N = \frac1N \sum_{i=1}^N {\delta_{\xi_i}}$
and $m^N_t = \mu + t (\mu_N-\mu)$, $t \in [0,1].$ We also consider i.i.d. random variables $(\tilde{\xi}_i)$  with law $\mu$ that are also independent of $(\xi_i)$.
\\
\begin{enumerate}[(i)]
\item  By the definition of linear functional derivatives, we have
\begin{align}
\esp{\Phia(\mu_N)} - \Phia(\mu) &= \esp{\int_0^1 \int_{\bR^d} \frac{\delta \Phia}{\delta m}(m^N_t)(v)\, (\mu_N-\mu)(dv) \ud t} \nonumber 
\\
&=\int_0^1  \frac1N\sum_{i=1}^N\left(\esp{\frac{\delta \Phia}{\delta m}(m^N_t)(\xi_i)} -\esp{\frac{\delta \Phia}{\delta m}(m^N_t)(\tilde{\xi}_1)} \right) \ud t \nonumber 
\\
&=\int_0^1\esp{\frac{\delta \Phia}{\delta m}(m^N_t)(\xi_1) -\frac{\delta \Phia}{\delta m}(m^N_t)(\tilde{\xi}_1)} \ud t. \nonumber 
\end{align}
We introduce measures
\begin{align*}
\tilde{m}^{N}_{t} :=  {m}^{N}_{t} + \frac{t}{N} (\delta_{\tilde{\xi}_1} - \delta_{{\xi}_1}) \quad \text{and} \quad {m}^{N}_{t,t_1}:= (\tilde{m}^{N}_{t}-m^{N}_t)t_1 + m^{N}_t, 
 \quad t,t_1 \in [0,1],
\end{align*}
and notice that
\begin{align*}
\esp{\frac{\delta \Phia}{\delta m}(\tilde{m}^{N}_{t})(\tilde{\xi}_1)} = \esp{\frac{\delta \Phia}{\delta m}({m}^{N}_{t})(\xi_1)}.
\end{align*}
Therefore, 
\begin{eqnarray}
  \esp{\Phia(\mu_N)} - \Phia(\mu) & = &  \int_0^1  \esp{\frac{\delta \Phia}{\delta m}(\tilde{m}^{N}_{t})(\tilde{\xi}_1)-\frac{\delta \Phia}{\delta m}(m^N_t)(\tilde{\xi}_1)}  \ud t \nonumber \\
  & = &  \int_0^1 \bE \bigg[ \int_0^1  \int_{\bR^d} \frac{\delta^2 \Phia}{\delta m^2}( {m}^{N}_{t,t_1})(\tilde{\xi}_1,y_1) (\tilde{m}^{N}_{t} - {m}^{N}_{t}) (dy_1) \, dt_1 \bigg]  \ud t \nonumber \\
  & = & \frac1{N}  \bE \bigg[ \int_0^1 \int_0^1  \int_{\bR^d} t \frac{\delta^2 \Phia}{\delta m^2}( {m}^{N}_{t,t_1})(\tilde{\xi}_1,y_1) (\delta_{\tilde{\xi}_1} - \delta_{{\xi}_1}) (dy_1) \, dt_1 \ud t \bigg].   \label{exp first step}
\end{eqnarray}
To conclude part (i), we observe that
\begin{align*}
\esp{\frac{\delta^2 \Phia}{\delta m^2}( {m}^{N}_{t,t_1})(\tilde{\xi}_1,y_1) (\delta_{\tilde{\xi}_1} - \delta_{{\xi}_1}) (dy_1)} & \le \esp{\sup_{\nu \in \cP_2(\R^d)} |\frac{\delta^2 \Phia}{\delta m^2}( \nu)(\tilde{\xi}_1,\tilde{\xi}_1)|+|\frac{\delta^2 \Phia}{\delta m^2}( \nu)(\tilde{\xi}_1,{\xi}_1)|}
\\
&\le 2L_{(U,\mu)}\,,
\end{align*}
by assumption \ref{as linear derivative} with $p=2$.

\item We continue the expansion of \eqref{exp first step}. To avoid a further interpolation in measure between ${m}^{N}_{t,t_1} $ and $\mu$, we proceed via integration by parts.  Let 
$$ g(t):= \int_0^1  \int_{\bR^d} \frac{\delta^2 \Phia}{\delta m^2}( {m}^{N}_{t,t_1})(\tilde{\xi}_1,y_1) (\delta_{\tilde{\xi}_1} - \delta_{{\xi}_1}) (dy_1) \, dt_1, \quad t \in [0,1], $$ 
and note that ${m}^{N}_{t,t_1}:= \frac{tt_1}{N}(\delta_{\tilde{\xi}_1} - \delta_{{\xi}_1})+ \mu + t(\mu_N- \mu)$. Then, by a similar method as the derivation of \eqref{eq: pth order linear functional derivative},
\[
  g'(t)  =  \int_0^1  \int_{\bR^d}  \int_{\bR^d} \frac{\delta^3 \Phia}{\delta m^3}( {m}^{N}_{t,t_1})(\tilde{\xi}_1,y_1,y_2) (\delta_{\tilde{\xi}_1} - \delta_{{\xi}_1}) (dy_1) \Big( \frac{t_1}{N}(\delta_{\tilde{\xi}_1} - \delta_{{\xi}_1})  + (\mu_N - \mu ) \Big) (dy_2 ) \, dt_1. \nonumber  
\]
Therefore, by integration by parts,
\begin{eqnarray}
 &&  \bE \bigg[ \int_0^1 \int_0^1  \int_{\bR^d} t \frac{\delta^2 \Phia}{\delta m^2}( {m}^{N}_{t,t_1})(\tilde{\xi}_1,y_1) (\delta_{\tilde{\xi}_1} - \delta_{{\xi}_1}) (dy_1) \, dt_1 \ud t \bigg] \nonumber \\
 & = & \bE \bigg[ \int_0^1 t g(t) \,dt \bigg] = \bE \bigg[ \int_0^1 (1-t) g(1-t) \,dt \bigg] = \bE \bigg[ \frac{1}{2} g(0) + \int_0^1 (t- \frac{t^2}{2})g'(1-t) \,dt \bigg] \nonumber \\
 & = & \bE \bigg[ \frac{1}{2} g(0) + \int_0^1 ( \frac{1-t^2}{2})g'(t) \,dt \bigg]  \nonumber \\
 & = & \frac{1}{2} \bE \bigg[   \int_{\bR^d} \frac{\delta^2 \Phia}{\delta m^2}( \mu)(\tilde{\xi}_1,y_1) (\delta_{\tilde{\xi}_1} - \delta_{{\xi}_1}) (dy_1)  \bigg] \nonumber  \\
&& + \frac{1}{2N} \bE \bigg[ \int_0^1 \int_0^1  \int_{\bR^d}   \int_{\bR^d} ( 1-t^2 ){t_1}  \frac{\delta^3 \Phia}{\delta m^3}( {m}^{N}_{t,t_1})(\tilde{\xi}_1,y_1,y_2) (\delta_{\tilde{\xi}_1} - \delta_{{\xi}_1}) (dy_1) (\delta_{\tilde{\xi}_1} - \delta_{{\xi}_1})   (dy_2 ) \, dt_1 \,dt \bigg] \nonumber \\
&& + \frac{1}{2} \bE \bigg[ \int_0^1 \int_0^1  \int_{\bR^d}   \int_{\bR^d} (1-t^2) \frac{\delta^3 \Phia}{\delta m^3}( {m}^{N}_{t,t_1})(\tilde{\xi}_1,y_1,y_2) (\delta_{\tilde{\xi}_1} - \delta_{{\xi}_1}) (dy_1)   (\mu_N - \mu )  (dy_2 ) \, dt_1 \,dt \bigg]. \nonumber  \\
&& \label{eq:exp second step} \end{eqnarray}
For the final term in \eqref{eq:exp second step}, by exchangeability, we rewrite 
\begin{eqnarray}
&& \bE \bigg[ \int_{\bR^d}   \int_{\bR^d}  \frac{\delta^3 \Phia}{\delta m^3}( {m}^{N}_{t,t_1})(\tilde{\xi}_1,y_1,y_2) (\delta_{\tilde{\xi}_1} - \delta_{{\xi}_1}) (dy_1)   (\mu_N - \mu )  (dy_2 ) \bigg] \nonumber \\
& = & \frac{1}{N} \bE \bigg[ \int_{\bR^d}   \int_{\bR^d}  \frac{\delta^3 \Phia}{\delta m^3}( {m}^{N}_{t,t_1})(\tilde{\xi}_1,y_1,y_2) (\delta_{\tilde{\xi}_1} - \delta_{{\xi}_1}) (dy_1)   (\delta_{{\xi}_1}-\delta_{\tilde{\xi}_2})  (dy_2 ) \bigg] \nonumber \\
&& + \frac{1}{N} \sum_{i=2}^N \bE \bigg[ \int_{\bR^d}   \int_{\bR^d}  \frac{\delta^3 \Phia}{\delta m^3}( {m}^{N}_{t,t_1})(\tilde{\xi}_1,y_1,y_2) (\delta_{\tilde{\xi}_1} - \delta_{{\xi}_1}) (dy_1)   (\delta_{{\xi}_i}-\delta_{\tilde{\xi}_2})  (dy_2 ) \bigg] \nonumber \\
& = & \frac{1}{N} \bE \bigg[ \int_{\bR^d}   \int_{\bR^d}  \frac{\delta^3 \Phia}{\delta m^3}( {m}^{N}_{t,t_1})(\tilde{\xi}_1,y_1,y_2) (\delta_{\tilde{\xi}_1} - \delta_{{\xi}_1}) (dy_1)   (\delta_{{\xi}_1}-\delta_{\tilde{\xi}_2})  (dy_2 ) \bigg] \nonumber \\
&& + \frac{N-1}{N} \bE \bigg[ \int_{\bR^d}   \int_{\bR^d}  \frac{\delta^3 \Phia}{\delta m^3}( {m}^{N}_{t,t_1})(\tilde{\xi}_1,y_1,y_2) (\delta_{\tilde{\xi}_1} - \delta_{{\xi}_1}) (dy_1)   (\delta_{{\xi}_2}-\delta_{\tilde{\xi}_2})  (dy_2 ) \bigg]. \label{eq: exp third step exchangeability}  
\end{eqnarray}
As before, we introduce measures
\begin{align*}
\tilde{m}^{N}_{t,t_1} :=  {m}^{N}_{t_1} + \frac{t}{N} (\delta_{\tilde{\xi}_2} - \delta_{{\xi}_2}) \quad \text{and} \quad {m}^{N}_{t,t_1,t_2}:= (\tilde{m}^{N}_{t,t_1}-m^{N}_{t,t_1})t_2 + m^{N}_{t,t_1}, 
 \quad t,t_1,t_2 \in [0,1].
\end{align*}
Then 
\begin{eqnarray}
 &&  \bE \bigg[ \int_{\bR^d}   \int_{\bR^d}  \frac{\delta^3 \Phia}{\delta m^3}( {m}^{N}_{t,t_1})(\tilde{\xi}_1,y_1,y_2) (\delta_{\tilde{\xi}_1} - \delta_{{\xi}_1}) (dy_1)   (\delta_{{\xi}_2}-\delta_{\tilde{\xi}_2})  (dy_2 ) \bigg] \nonumber \\
 & = & \bE \bigg[ \int_{\bR^d}     \frac{\delta^3 \Phia}{\delta m^3}( {m}^{N}_{t,t_1})(\tilde{\xi}_1,y_1,{\xi}_2) (\delta_{\tilde{\xi}_1} - \delta_{{\xi}_1}) (dy_1) \bigg] - \bE \bigg[ \int_{\bR^d}     \frac{\delta^3 \Phia}{\delta m^3}( {m}^{N}_{t,t_1})(\tilde{\xi}_1,y_1,\tilde{\xi}_2) (\delta_{\tilde{\xi}_1} - \delta_{{\xi}_1}) (dy_1) \bigg] \nonumber \\
 & = & \bE \bigg[ \int_{\bR^d}     \frac{\delta^3 \Phia}{\delta m^3}( \tilde{m}^{N}_{t,t_1})(\tilde{\xi}_1,y_1,\tilde{\xi}_2) (\delta_{\tilde{\xi}_1} - \delta_{{\xi}_1}) (dy_1) \bigg] - \bE \bigg[ \int_{\bR^d}     \frac{\delta^3 \Phia}{\delta m^3}( {m}^{N}_{t,t_1})(\tilde{\xi}_1,y_1,\tilde{\xi}_2) (\delta_{\tilde{\xi}_1} - \delta_{{\xi}_1}) (dy_1) \bigg] \nonumber \\
 & = & \frac{t}{N} \bE \bigg[\int_0^1 \int_{\bR^d} \int_{\bR^d}     \frac{\delta^4 \Phia}{\delta m^4}( {m}^{N}_{t,t_1,t_2})(\tilde{\xi}_1,y_1,\tilde{\xi}_2,y_2) (\delta_{\tilde{\xi}_1} - \delta_{{\xi}_1}) (dy_1) (\delta_{\tilde{\xi}_2} - \delta_{{\xi}_2}) (dy_2) \, dt_2 \bigg].  \label{eq: exp final step}  
\end{eqnarray}
Combining \eqref{exp first step}, \eqref{eq:exp second step}, \eqref{eq: exp third step exchangeability} and \eqref{eq: exp final step} gives
\begin{eqnarray}
 &&  \esp{\Phia(\mu_N)} - \Phia(\mu) \nonumber \\
 &= & \frac{1}{2N} \bE \bigg[   \int_{\bR^d} \frac{\delta^2 \Phia}{\delta m^2}( \mu)(\tilde{\xi}_1,y_1) (\delta_{\tilde{\xi}_1} - \delta_{{\xi}_1}) (dy_1)  \bigg] \nonumber  \\
&& + \frac{1}{2N^2} \bE \bigg[ \int_0^1 \int_0^1  \int_{\bR^d}   \int_{\bR^d} ( 1-t^2 ){t_1}  \frac{\delta^3 \Phia}{\delta m^3}( {m}^{N}_{t,t_1})(\tilde{\xi}_1,y_1,y_2) (\delta_{\tilde{\xi}_1} - \delta_{{\xi}_1}) (dy_1) (\delta_{\tilde{\xi}_1} - \delta_{{\xi}_1})   (dy_2 ) \, dt_1 \,dt \bigg] \nonumber \\
&& + \frac{1}{2N^2} \bE \bigg[\int_0^1 \int_0^1  \int_{\bR^d}   \int_{\bR^d} (1-t^2)  \frac{\delta^3 \Phia}{\delta m^3}( {m}^{N}_{t,t_1})(\tilde{\xi}_1,y_1,y_2) (\delta_{\tilde{\xi}_1} - \delta_{{\xi}_1}) (dy_1)   (\delta_{{\xi}_1}-\delta_{\tilde{\xi}_2})  (dy_2 ) \, dt_1 \,dt \bigg] \nonumber \\
&& + \frac{N-1}{2N^3}  \bE \bigg[\int_0^1 \int_0^1 \int_0^1 \int_{\bR^d} \int_{\bR^d}   t(1-t^2)   \frac{\delta^4 \Phia}{\delta m^4}( {m}^{N}_{t,t_1,t_2})(\tilde{\xi}_1,y_1,\tilde{\xi}_2,y_2) \nonumber \\
&& (\delta_{\tilde{\xi}_1} - \delta_{{\xi}_1}) (dy_1) (\delta_{\tilde{\xi}_2} - \delta_{{\xi}_2}) (dy_2) \, dt_2 \, dt_1 \,dt \bigg]. \nonumber 
\end{eqnarray}
{
Using the fact that $(U,\mu)$ satisfies assumption \ref{as linear derivative} with $p\in\set{3,4}$, the statement for part (ii) is established.
}
\end{enumerate}
\end{proof}
In principle, we can continue the above expansion to higher orders. However, in the next theorem  we present a simplified argument that allows for complete weak error expansion. The simplification is at the cost of requiring one extra order of regularity in the assumption. However, we believe the argument is of independent interest.  
\begin{theorem}[Weak error expansion: static case] \label{ theorem second order expansion} 
Let $q$ be a positive integer and $\mu \in \cP_{2q-1}(\R^d)$. Suppose that assumption \ref{as linear derivative} holds for $U:\cP_2(\R^d)\rightarrow \R$, for each $p \in \{1, \ldots, 2q-1\}$. 
Then, for i.i.d. random variables $\{\xi_i\}_{i \in \bN}$ with law $\mu$,
\begin{align*}
\esp{\Phia \bigg( \frac1N \sum_{i=1}^N \delta_{\xi_i} \bigg)} - \Phia(\mu) &= \sum_{p=2}^{q-1}\frac{C_p}{N^{p-1}} + O(\frac1{N^{q-1}})\;,
\end{align*}
where 
\begin{align*}
C_p = \esp{\int_{\R^{pd}} \frac{\delta^{p} \Phia}{\delta m^{p}}(\mu)(\mathbf{y})\bigotimes_{k=1}^{p}(\delta_{\xi}- \delta_{\hat{\xi}^k})(dy_k) },
\end{align*}
for some i.i.d. random variables $(\hat{\xi}^k)_{1 \leq k \leq q}$  with law $\mu$ that are also independent of $(\xi_i)_{i \in \bN}$.
\end{theorem}
\begin{proof} 
Let $\mu_N = \frac1N \sum_{i=1}^N {\delta_{\xi_i}}$. By Lemma \ref{le easy expansion}, we have
\begin{align}
\esp{\Phia(\mu_N)} - \Phia(\mu) = 
\sum_{p=1}^{q-1}\frac1{p!} \esp{\int_{\R^{pd}}\frac{\delta^p \Phia}{\delta m^p}(\mu)(\mathbf{y})\,(\mu_N-\mu)^{\otimes p} (d\mathbf{y})}
+ \frac1{(q-1)!} \int_0^1(1-t)^{(q-1)} R(q,N,t) \ud t \label{eq remainder}
\end{align}
with
\begin{align*}
R(q,N,t) := \esp{\int_{\R^{qd}}\frac{\delta^{q} \Phia}{\delta m^{q}}(m^N_t)(\mathbf{y})\,(\mu_N-\mu)^{\otimes q} (d \mathbf{y})},
\end{align*}
where $m^N_t:= (1-t) \mu + t \mu_N$. Observe that by assumption \ref{as linear derivative} all the terms in the expansion are well defined. We study them now. For $p=1$, we have
$$ \esp{\int_{\R^{d}}\frac{\delta \Phia}{\delta m}(\mu)(y)\,(\mu_N-\mu)(dy)} =0.  $$ 
Now let $p \in \set{2,\dots,q-1}$ and observe that
\begin{align*}
\esp{\int_{\R^{pd}}\frac{\delta^p \Phia}{\delta m^p}(\mu)(\mathbf{y})\,(\mu_N-\mu)^{ \otimes p} (d\mathbf{y})}
= \frac1{N^p}\sum_{1 \le i_1,\dots,i_p \le N} \esp{\int_{\R^{pd}}\frac{\delta^{p} \Phia}{\delta m^{p}}(\mu)(\mathbf{y})\bigotimes_{k=1}^p(\delta_{\xi_{i_k}}- \delta_{\hat{\xi}^k})(dy_k)}.
\end{align*}
Suppose that at least one of the $i_k$ is different from the other $i_j$, $j \neq k$. Without loss of generality, we assume that this is the case for $k=p$. We then observe that
\begin{eqnarray*}
&& \esp{\int_{\R^{pd}}\frac{\delta^{p} \Phia}{\delta m^{p}}(\mu)(\mathbf{y})\bigotimes_{k=1}^p(\delta_{\xi_{i_k}}- \delta_{\hat{\xi}^k})(dy_k)} \\
& = & \esp{\int_{\R^{(p-1)d}}\frac{\delta^{p} \Phia}{\delta m^{p}}(\mu)(y_1, \ldots, y_{p-1},\xi_{i_p})\bigotimes_{k=1}^{p-1}(\delta_{\xi_{i_k}}- \delta_{\hat{\xi}^k})(dy_k)} \nonumber \\
&& - \esp{\int_{\R^{(p-1)d}}\frac{\delta^{p} \Phia}{\delta m^{p}}(\mu)(y_1, \ldots, y_{p-1},\hat{\xi}^p)\bigotimes_{k=1}^{p-1}(\delta_{\xi_{i_k}}- \delta_{\hat{\xi}^k})(dy_k)} =0,
\end{eqnarray*}
by conditioning on $\xi_{i_1}, \ldots, \xi_{i_{p-1}}, \hat{\xi}^1, \ldots, \hat{\xi}^{p-1}$. Therefore, when $i_1=\dots=i_p$, 
\begin{align*}
\esp{\int_{\R^{pd}}\frac{\delta^p \Phia}{\delta m^p}(\mu)(\mathbf{y})\,(\mu_N-\mu)^{ \otimes p} (d\mathbf{y})} = \frac1{N^{p-1}}\esp{\int_{\R^{pd}} \frac{\delta^{p} \Phia}{\delta m^{p}}(\mu)(\mathbf{y})\bigotimes_{k=1}^{p}(\delta_{\xi}- \delta_{\hat{\xi}^k})(dy_k) }\;.
\end{align*}
It remains to study the remainder term $R$ above. We rewrite
\begin{align*}
R(q,N,t) = \frac1{N^q}\sum_{1\le i_1, \dots, i_q \le N}\esp{\int_{\R^{qd}}\frac{\delta^{q} \Phia}{\delta m^{q}}(m^N_t)(\mathbf{y})\bigotimes_{p=1}^q(\delta_{\xi_{i_p}}- \delta_{\hat{\xi}^p})(dy_p)}\;.
\end{align*}
Let $\cL$  be a subset of  $\omega=\set{1, \dots , q}$. We denote $\cL^c := \set{1,\dots,q}\setminus \cL$ and  introduce
\begin{eqnarray*}
\cI^\cL := \set{\mathbf{i} =(i_1, \ldots, i_q) \in \{1, \ldots, N \}^{q}   & | &  \forall \ell,\ell' \in \cL, \,i_{\ell}=i_{\ell'}, \, \,  \forall k,k' \in  \cL^c \, \text{ s.t. } \, k \neq k'\,, i_k\neq i_{k'} \\
&&  \text{ and } \forall (\ell,k) \in \cL\times\cL^c, i_\ell \neq i_k }\;.
\end{eqnarray*}
Then 
\begin{align*}
R(q,N,t) = \frac1{N^q}\sum_{j=1}^q \sum_{\cL, |\cL|=j} \sum_{\mathbf{i} \in \cI^\cL} \esp{\int_{\R^{qd}}\frac{\delta^{q} \Phia}{\delta m^{q}}(m^N_t)(\mathbf{y})\bigotimes_{p=1}^q(\delta_{\xi_{i_p}}- \delta_{\hat{\xi}^p})(dy_p)}.
\end{align*}
For $j=q$, we simply observe that
\begin{align}
\frac1{N^q}\sum_{\mathbf{i} \in \cI^\omega}\esp{\int_{\R^{qd}}\frac{\delta^{q} \Phia}{\delta m^{q}}(m^N_t)(\mathbf{y})\bigotimes_{p=1}^q(\delta_{\xi_{i_p}}- \delta_{\hat{\xi}^p})(dy_p)} = O(\frac{1}{N^{q-1}})\;.
\end{align}
For $1 \le j<q$, we consider $\cI^\cL$ defined above and work with the special choice $\cL=\set{1,\dots, j}$, which implies, by exchangeability, that
\begin{align}
 &\sum_{\mathbf{i} \in \cI^\cL} \esp{\int_{\R^{qd}}\frac{\delta^{q} \Phia}{\delta m^{q}}(m^N_t)(\mathbf{y})\bigotimes_{p=1}^q(\delta_{\xi_{i_p}}- \delta_{\hat{\xi}^p})(dy_p)} \nonumber
 \\
 &= N(N-1)\dots(N-(q-j)) \esp{\int_{\R^{qd}}\frac{\delta^{q} \Phia}{\delta m^{q}}(m^N_t)(\mathbf{y})\bigotimes_{p=1}^j(\delta_{\xi_{1}}- \delta_{\hat{\xi}^p})(dy_p)
\bigotimes_{p=j+1}^q(\delta_{\xi_{p}}- \delta_{\hat{\xi}^p})(dy_p) 
 }. \label{eq same law}
\end{align}
For later use, we denote
\begin{align*}
\Delta(d\mathbf{y}):=\bigotimes_{p=1}^j(\delta_{\xi_{1}}- \delta_{\hat{\xi}^p})(dy_p)
\bigotimes_{p=j+1}^q(\delta_{\xi_{p}}- \delta_{\hat{\xi}^p})(dy_p). 
\end{align*}
\\We will now work iteratively from $j+1$ to $q$. 
\\Firstly, we introduce
\begin{align*}
\tilde{m}^{N}_{t} := m^N_t + \frac{t}{N} (\delta_{\tilde{\xi}_{{j+1}}} - \delta_{{\xi}_{{j+1}}})  \quad 
\text{ and } \quad  {m}^{N}_{t,s_{j+1}} := \tilde{m}^{N}_{t} + s_{j+1}(m^N_{t}-\tilde{m}^{N}_{t}),
\end{align*}
where we define independent random variables $\{ \tilde{\xi}_{u} \}_{j+1 \leq u \leq q}$ that are also independent of  $(\xi_i)_{1\leq i \leq q}$ and $(\hat{\xi}^i)_{1\leq i \leq q}$, but with the same law. We then compute that
\begin{align}
&\esp{\int_{\R^{qd}}\frac{\delta^{q} \Phia}{\delta m^{q}}(m^N_t)(\mathbf{y})\Delta(d\mathbf{y})
 }
= 
\esp{\int_{\R^{qd}}\frac{\delta^{q} \Phia}{\delta m^{q}}(\tilde{m}^N_t)(\mathbf{y})\Delta(d\mathbf{y})
 }
 \label{eq rem first term}
\\
&+ \frac{t}N \int_0^1 \!\!
\esp{\int_{\R^{(q+1)d}}\frac{\delta^{q+1} \Phia}{\delta m^{q+1}}({m}^{N}_{t,s_{j+1}})(\mathbf{y},y_{q+1})
\Delta(d\mathbf{y})(\delta_{\tilde{\xi}_{{j+1}}}- \delta_{\hat{\xi}^{j+1}})(dy_{q+1})  } \ud s_{j+1}.\label{eq rem next term }
\end{align}
As before,
\begin{align*}
&\, \esp{\int_{\R^{(q-1)d}}\frac{\delta^{q} \Phia}{\delta m^{q}}(\tilde{m}^{N}_t)(y_1,\dots,y_{j},{\xi}_{{j+1}},y_{j+2},\dots,y_q)
\bigotimes_{p=1}^j(\delta_{\xi_{1}}- \delta_{\hat{\xi}^p})(dy_p)
\bigotimes_{p=j+2}^q(\delta_{\xi_{p}}- \delta_{\hat{\xi}^p})(dy_p)  }
\\
= & \, \esp{\int_{\R^{(q-1)d}}\frac{\delta^{q} \Phia}{\delta m^{q}}(\tilde{m}^{N}_t)(y_1,\dots,y_{j},\hat{\xi}^{{j+1}},y_{j+2},\dots,y_q)
\bigotimes_{p=1}^j(\delta_{\xi_{1}}- \delta_{\hat{\xi}^p})(dy_p)
\bigotimes_{p=j+2}^q(\delta_{\xi_{p}}- \delta_{\hat{\xi}^p})(dy_p) }, 
\end{align*}
so the term on the right hand side of  \eqref{eq rem first term} is equal to zero. 
Next, for $u \in  \{j + 1, \ldots, q-1 \}$, we define inductively
\[ \begin{cases} 
      \tilde{m}^{N}_{t,s_{j+1}, \ldots, s_u} := {m}^{N}_{t,s_{j+1}, \ldots, s_u} + \frac{t}N(\delta_{\tilde{\xi}_{u+1}}-\delta_{\xi_{u+1}}), \\
       \\
      {m}^{N}_{t,s_{j+1}, \ldots, s_{u+1}} :=  \tilde{m}^{N}_{t,s_{j+1}, \ldots, s_u}  + s_{u+1}(\tilde{m}^{N}_{t,s_{j+1}, \ldots, s_u}-{m}^{N}_{t,s_{j+1}, \ldots, s_u}).
   \end{cases}
\]
This procedure is  then iterated from  $j+2$ to $q$ on the remainder term in \eqref{eq rem next term }.  We thus have
\begin{eqnarray}
\esp{\int_{\R^{qd}}\frac{\delta^{q} \Phia}{\delta m^{q}}(m^N_t)(\mathbf{y}) \Delta(d\mathbf{y}) } & =  & \left(\frac{t}{N}\right)^{q-j}\int_0^1 \ldots \int_0^1 
 \bE \bigg[ \int_{\R^{(2q-j)d}}\frac{\delta^{2q-j} \Phia}{\delta m^{2q-j}}({m}^{N}_{t,s_{j+1},\dots,s_{q}} )(\mathbf{y},y_{q+1},\dots,y_{2q-j}) \nonumber \\
 && \Delta(d\mathbf{y})
\bigotimes_{k=j+1}^{q}(\delta_{\tilde{\xi}_{k}}- \delta_{\hat{\xi}^k})(dy_{q-j+k}) \bigg] \,ds_{j+1} \ldots ds_q. \label{eq: iteration final second expansion} 
\end{eqnarray}
Next, by \eqref{eq: poly growth second expansion}, we estimate the integral by
\begin{eqnarray}
 && \bigg| \bE \bigg[ \int_{\R^{(2q-j)d}}\frac{\delta^{2q-j} \Phia}{\delta m^{2q-j}}({m}^{N}_{t,s_{j+1},\dots,s_{q}} )(\mathbf{y},y_{q+1},\dots,y_{2q-j}) \Delta(d\mathbf{y})
\bigotimes_{k=j+1}^{q}(\delta_{\tilde{\xi}_{k}}- \delta_{\hat{\xi}^k})(dy_{q-j+k}) \bigg] \bigg| \nonumber \\
& \leq & \bE \bigg[ \sum_{y_1 \in \{ \xi_1, \hat{\xi}^1 \} } \ldots  \sum_{y_j \in \{ \xi_1, \hat{\xi}^j \} } \sum_{y_{j+1} \in \{ \xi_{j+1}, \hat{\xi}^{j+1} \} } \ldots \sum_{y_{q} \in \{ \xi_{q}, \hat{\xi}^{q} \} } \sum_{y_{q+1} \in \{ \tilde{\xi}_{j+1}, \hat{\xi}^{j+1} \} } \ldots \sum_{y_{2q-j} \in \{ \tilde{\xi}_{q}, \hat{\xi}^{q} \} } \nonumber \\
&& \bigg| \frac{\delta^{2q-j} \Phia}{\delta m^{2q-j}}({m}^{N}_{t,s_{j+1},\dots,s_{q}} )(y_1,\dots,y_{2q-j}) \bigg| \bigg] \nonumber \\
&\leq & C 2^{4q-2j} L_{(U,m)}\;,
\label{eq: bound crazy summation} 
\end{eqnarray}
where we used assumption \ref{as linear derivative}.
 Combining with \eqref{eq: iteration final second expansion} and \eqref{eq: bound crazy summation} gives
\[
\esp{\int_{\R^{qd}}\frac{\delta^{q} \Phia}{\delta m^{q}}(m^N_t)(\mathbf{y})\Delta(d\mathbf{y})} = O(\frac1{N^{q-j}})\;.
\]
Finally, combining with \eqref{eq same law} yields
\begin{align*}
 \frac1{N^q}\sum_{j=1}^q \sum_{\cL, |\cL|=j} \sum_{\mathbf{i} \in \cI^\cL} \esp{\int_{\R^{qd}}\frac{\delta^{q} \Phia}{\delta m^{q}}(m^N_t)(\mathbf{y})\bigotimes_{p=1}^q(\delta_{\xi_{i_p}}- \delta_{\hat{\xi}^p})(dy_p)} = O(\frac1{N^{q-1}}).
\end{align*}
\end{proof}

\subsection{Expansion in terms of regularity of the drift and diffusion functions}
\label{subse regular expansion}
 In this subsection, we explore a sufficient condition for the expansion of an arbitrary order purely in terms of regularity of the drift and diffusion functions. It turns out that proving regularity conditions for higher order expansions for class $\cD$ is highly non-trivial and therefore a stronger notion $\cM_k$ of regularity in differentiating measures is proposed.

\begin{definition}
A function $f: \bR^d \times \cP_2 ( \bR^d)) \to \bR$  belongs to class $\cM_k( \bR^d \times \cP_2 ( \bR^d)) $, if the derivatives  $D^{(n,\ell, \bm{\beta})} f (x,\mu,v_1, \ldots, v_n)$  exist
for every multi-index $(n,\ell, \bm{\beta})$ such that $|(n,\ell, \bm{\beta})| \leq k$ and satisfy
\begin{equation}   
(a) \quad \big| D^{(n,\ell, \bm{\beta})} f (x,\mu, v_1, \ldots, v_n) \big|  
 \leq  C, 
 \label{eq:boundf}
 \end{equation}
\begin{align} 
(b) \quad \Big| D^{(n,\ell, \bm{\beta})} f (x,\mu, v_1, \ldots, v_n) -D^{(n,\ell, \bm{\beta})} f  (x', \mu' ,v'_1, \ldots, v'_n)  \Big| \nonumber \\
& \leq & C \bigg( |x-x'| + \sum_{i=1}^n|v_i-v'_i| + W_2 (\mu,\mu') \bigg),  
\label{eq:Lipf}
\end{align}
for any $x,x',v_1,v'_1,\ldots, v_n, v'_n \in \bR^d$ and $\mu, \mu' \in \mathcal{P}_2(\bR^d)$, for some constant $C>0$.
\end{definition}

By convention, a function $f$ defined only on $\cP_2(\R^d)$ will be extended to $\R^d \times \cP_2(\R^d)$ naturally by $(x,\mu)\mapsto f(\mu)$, for all $x \in \R^d$.

\vspace{2pt}
For the time-dependent case (possibly with multi-index in time), we extend the previous definition as follows.
\begin{definition} \label{def Mk time space measure}
 A function $\cV:\overline{\Delta^m_{T}} \times \R^d\times \mathcal{P}_2 ( \bR^d) \to \bR$ is said  to be in $\cM_k ( \Delta^m_{T} \times \R^d \times \mathcal{P}_2 ( \bR^d))$ \footnote{This definition is modified accordingly to define $\cM_k ( (0,t) \times \mathcal{P}_2 ( \bR^d))$, where $t \in (0,T)$.}, if
\begin{enumerate}
\item \begin{itemize}
\item $m=1$: $s \mapsto \cV(s,x,\mu)$ is continuously differentiable on $(0,T)$.
\item $m>1$: for all $(\tau_1,\dots,\tau_{m-1}) \in \Delta^{m-1}_{T}$ and all $\mu \in \cP_2(\R^d)$, the function
$$ (0,\tau_{m-1}) \ni s \mapsto \cV((\tau_1,\dots,\tau_{m-1},s),x,\mu) \in \R$$
is continuously differentiable on $(0,\tau_{m-1})$.
\end{itemize}
\item $\cV(\mathbf{t}, \cdot) \in \cM_k (\R^d \times \mathcal{P}_2 ( \bR^d))$, for each $\mathbf{t} \in \Delta^m_{T}$, where the constant $C$ in \eqref{eq:boundf} and \eqref{eq:Lipf} is uniform in $\mathbf{t}$.
\item All derivatives in measure (including the zeroth order derivative) of $\cV(\cdot, \cdot)$ up to the $k$th order are jointly continuous in time, measure and space.
\end{enumerate}
\end{definition}

When it is clear from context, we will just use the notation $\cM_k$ for the two definitions above.

Note that the condition $\cM_1(\bR^d \times \mathcal{P}_2 (\bR^d) ) $ automatically implies \ref{eq:Lip}.
The following is a generalisation of Theorem 7.2 in \cite{buckdahn2017mean} from $\cM_2$ to $\cM_k$, for any $k \geq 2$.
\begin{theorem} \label{eq:generalisationmainresult}
 Suppose that $b$ and $\sigma$ are in $\cM_k(\bR^d \times \mathcal{P}_2 (\bR^d) )\,$, where $k \geq 2$.
     We consider a function $\cV : [0,t] \times \mathcal{P}_2 (\bR^d) \to \bR $ defined by 
\begin{equation}  \cV( s, \mu) = \Phi \big( \law[{X^{s, \mu}_t}] \big),  \label{eq: defofflow}
\end{equation}
  for some function $\Phi:  \mathcal{P}_2 (\bR^d)  \to \bR$ that is also in $\cM_k( \mathcal{P}_2 (\bR^d) )$. Then  $\cV \in \cM_k ( (0,t) \times \mathcal{P}_2 (\bR^d)) $ and satisfies the PDE
\[ \begin{cases} 
       \partial_s \cV(s, \mu) + \int_{\bR^d}  \big[ \partial_{\mu} \cV (s, \mu) (x) b( x, \mu)  + \frac{1}{2} \text{Tr} \big( \partial_v \partial_{\mu} \cV  ( s,\mu) ( x) a( x, \mu) \big) \big] \, \mu (dx) =0, & s \in (0,t), \\
      &  \\
      \cV(t,\mu) = \Phi ( \mu). &    \end{cases}
\]
\end{theorem}
This proof of Theorem \ref{eq:generalisationmainresult} is postponed to the next section.  We now state the key result for this part which will certify that the expansion along the dynamics is licit.
\begin{theorem}
\label{eq:main result with regularity}
Assume  \ref{eq:UB}.  Suppose that $b$ and $\sigma$ belong to the class $\cM_{2k}(\bR^d \times \mathcal{P}_2 (\bR^d) )$.  Moreover, suppose that $\Phi: \cP_2(\bR^d) \to \bR$ also belongs to the class $\cM_{2k}( \mathcal{P}_2 (\bR^d) )$. Then Definition \ref{de expansion function} is well-posed for $m \in \{1, \ldots, k\}$. 
\end{theorem}
\begin{proof}
We prove by induction on $m \in \{1, \ldots, k \}$ and prove that for each $m \in \{1, \ldots, k \}$, $\cV^{(m)} \in \cM_{2k-2m+2}(\Delta^m_T \times \mathcal{P}_2 ( \bR^d)) \subseteq  \cM_{2}(\Delta^m_T \times \mathcal{P}_2 ( \bR^d)) $ and therefore $\cV^{(m)} \in \cD(\Delta^m_T)$, which establishes the claim. 

For simplicity of notations, we present this proof in the case of dimension one. We commence the proof by noting that  $\Phi \in \cM_{2k} $ and $b,\sigma \in \cM_{2k}$, therefore it follows from Theorem \ref{eq:generalisationmainresult} that $\cV^{(1)} \in \cM_{2k} $.

Suppose that for $m \in \{1, \ldots, k-1 \}$, $\cV^{(m)} \in \cM_{2k-2m+2}.$ 
We recall the definition of 
$\Phi^{(m)}: \Delta^m_T \times \mathcal{P}_2( \bR) \to \bR$ as
$$  \Phi^{(m)}(\mathbf{t}, \mu) = \int_{\bR}  \partial^2_{\mu} \cV^{(m)} (\mathbf{t}, \mu) (x,x) \big( \sigma(x, \mu) \big)^2   \, \mu(dx).$$
Fix $\mathbf{t} \in \Delta^m_T $. We shall first establish the smoothness of $\Phi^{(m)}(\mathbf{t}, \cdot)$. Let $p:\bR \times \mathcal{P}_2(\bR) \to \bR$ be a continuous function defined by 
$$ p(x, \mu) := \partial^2_{\mu} \cV^{(m)} (\mathbf{t}, \mu) (x,x) \big( \sigma(x, \mu) \big)^2 .$$ Since $\cV^{(m)} \in  \cM_{2k-2m+2}$, for each $x \in \bR$,   $p(x, \cdot)$ is also differentiable in measure with its derivative given by
\begin{equation} \pmu p(x, \mu)(y) = \partial^3_{\mu} \cV^{(m)} (\mathbf{t}, \mu) (x,x,y) \big( \sigma(x, \mu) \big)^2 + 2 \partial^2_{\mu} \cV^{(m)} (\mathbf{t}, \mu) (x,x)  \big( \sigma(x, \mu) \big) \pmu \sigma(x, \mu)(y). \label{eq:pmup} \end{equation} We observe that $\pmu p(x, \mu)(y)$ and $\partial_x p(x, \mu)$ are both continuous and uniformly bounded in space and measure. Therefore, by Example 3 in Section 5.2.2 of \cite{carmona2017probabilistic}, $\Phi^{(m)}(\mathbf{t}, \cdot) $ is differentiable in measure with its derivative given by
\[
\pmu \Phi^{(m)}(\mathbf{t},\mu) (y) =  \partial_x p(y, \mu) + \int_{\bR} \pmu p(x, \mu)(y) \, \mu(dx), 
\]
where $ \partial_x p(y, \mu)$ is given by
\begin{eqnarray}
\partial_x p(y, \mu) & = & \bigg[ \partial_{v_1} \partial^2_{\mu} \cV^{(m)}  (\mathbf{t}, \mu) (y,y) + \partial_{v_2} \partial^2_{\mu} \cV^{(m)}  (\mathbf{t}, \mu) (y,y) \bigg] \Big( \sigma(y, \mu) \Big)^2 \nonumber \\
&& +  2 \partial^2_{\mu} \cV^{(m)}  (\mathbf{t}, \mu) (y,y) \sigma(y, \mu) \partial_y \sigma(y, \mu). \label{eq:partialxp}
\end{eqnarray}
Formulae \eqref{eq:pmup} and \eqref{eq:partialxp} tell us that  $\pmu \Phi^{(m)}(\mathbf{t},\mu) (y)$ is uniformly bounded in measure and space. Furthermore, each of $\partial_x p(y, \mu)$ and $\pmu p(x, \mu)(y) $ is a finite sum of products of uniformly bounded Lipschitz functions in measure and space, and is hence Lipschitz continuous as well. Finally, by the duality formula for the Kantorovich-Rubinstein
distance (see Remark 6.5 in \cite{villani2008optimal}), we note that there exist constants $C_1, C_2, C_3>0$ such that for every $\mu_1, \mu_2 \in \mathcal{P}_2(\bR)$ and $y_1, y_2 \in \bR$, 
\begin{eqnarray}
&& \Big| \pmu \Phi^{(m)}(\mathbf{t},\mu_1) (y_1 ) - \pmu \Phi^{(m)}(\mathbf{t},\mu_2) (y_2 ) \Big| \nonumber \\
& \leq & C_1 \bigg( |y_1 - y_2| + W_2(\mu_1, \mu_2) +  \bigg| \int_{\bR} \pmu p (x, \mu)(y_1) \mu_1 (dx) - \int_{\bR} \pmu p (x, \mu)(y_2) \mu_1 (dx) \bigg|  \nonumber \\
&& + \bigg| \int_{\bR} \pmu p (x, \mu)(y_2) \mu_1 (dx) - \int_{\bR} \pmu p (x, \mu)(y_2) \mu_2 (dx) \bigg| \bigg) \nonumber \\
& \leq & C_2 \bigg( |y_1 - y_2| + W_2(\mu_1, \mu_2) +  \bigg| \int_{\bR} \pmu p (x, \mu)(y_2) \mu_1 (dx) - \int_{\bR} \pmu p (x, \mu)(y_2) \mu_2 (dx) \bigg| \bigg)  \nonumber  \\
& \leq & C_2 \bigg( |y_1 - y_2| + W_2(\mu_1, \mu_2) + \| \pmu p \|_{\text{Lip}} W_1(\mu_1, \mu_2) \bigg) \nonumber \\
& \leq & C_3 \Big(  |y_1 - y_2| + W_2(\mu_1, \mu_2) \Big) , \nonumber 
\end{eqnarray}
where $W_1$ denotes the $1$-Wasserstein metric.

Subsequently, we can repeat the same procedure to prove existence and regularity properties of higher order derivatives of $\Phi^{(m)}(\mathbf{t}, \cdot) $. In particular, we can show that $\ptwomu \Phi^{(m)}(\mathbf{t}, \mu, v_1, v_2)$ and \\ $\partial_{v_1} \pmu \Phi^{(m)}(\mathbf{t},\mu, v_1)$ exist, by expressing them in terms of derivatives of $\cV^{(m)} $ up to the fourth order, and derivatives of $\sigma$ up to the second order, which also allows us to show that 
they are uniformly bounded and Lipschitz continuous. In general, for any multi-index $(n, \bm{\beta})$ such that $|(n, \bm{\beta})| \leq 2k-2m$, we can show that $D^{(n, \bm{\beta})} \Phi^{(m)}(\mathbf{t}, \mu, v_1, \ldots, v_n)$  exists, by expressing it in terms of derivatives of $\cV^{(m)} $ up to the $(2k-2m+2)$th order, and derivatives of $\sigma$ up to the $(2k-2m)$th order, which again allows us to show that 
it is uniformly bounded and Lipschitz continuous. Thus, $\Phi^{(m)}(\mathbf{t}, \cdot)  \in \cM_{2k-2m}$.

Next, we note that since $\cV^{(m)} \in \cM_{2k-2m+2}( \Delta^m_T \times \mathcal{P}_2( \bR) )$, $\cV^{(m)}$ is continuously differentiable in the last component of ${\bf{t}} \in \Delta^m_T$ and so is $\Phi^{(m)}$. Moreover, as mentioned above, each derivative $D^{(n, \bm{\beta})} \Phi^{(m)}(\mathbf{t}, \mu, v_1, \ldots, v_n)$ up to the $(2k-2m)$th order  can be expressed in terms of derivatives of $\cV^{(m)} $ up to the $(2k-2m+2)$th order and derivatives of $\sigma$ up to the $(2k-2m)$th order, which implies that each derivative  $D^{(n, \bm{\beta})} \Phi^{(m)}(\mathbf{t}, \mu, v_1, \ldots, v_n)$ is jointly continuous in time, measure and space, since \\ $\cV^{(m)} \in \cM_{2k-2m+2}( \Delta^m_T \times \mathcal{P}_2( \bR) )$. Therefore, by Definition \ref{def Mk time space measure}, $\Phi^{(m)} \in \cM_{2k-2m} (\Delta^m_T \times \cP_2(\bR))$.

We now recall the definition of $\cV^{(m+1)}:\overline{\Delta^{m+1}_T }\times \cP_2(\R) \rightarrow \R$, given by 
\begin{align*}
\cV^{(m+1)}((\tau,t),\mu) = \Phi^{(m)}(\tau,\law[X^{t,\mu}_{\tau_m}]) , \quad \quad \tau \in {\Delta^{m}_T }.
\end{align*}

For fixed $\tau \in {\Delta^{m}_T },$ it follows from Theorem \ref{eq:generalisationmainresult} that $\cV^{(m+1)}((\tau,\cdot),\cdot)$ is continuously differentiable in time and that  $\cV^{(m+1)}((\tau,t),\cdot) \in \cM_{2k-2m} $, for each $t \in (0,\tau_m).$ Finally, all derivatives in measure of $\cV^{(m+1)}$ up to the $(2k-2m)$th order are jointly continuous in time, measure and space, since $\Phi^{(m)} \in \cM_{2k-2m}$. This implies that $\cV^{(m+1)} \in \cM_{2k-2m}$, which concludes the proof by the principle of induction.

\end{proof}
The following theorem is the main result of this paper and is a direct consequence of Theorem  \ref{eq:mainresultwithoutregularity}, Theorem \ref{ theorem second order expansion}, Theorem \ref{eq:main result with regularity} and Remark \ref{re poly growth}(i).

\begin{theorem}[\textbf{Main result on regularity: Full expansion}] \label{eq:main result with regularity: final} 
Assume \ref{eq:UB}.  Suppose that $b$ and $\sigma$ belong to the class $\cM_{2k+1}(\bR^d \times \mathcal{P}_2 (\bR^d) )$.  Moreover, suppose that $\Phi: \cP_2(\bR^d) \to \bR$ also belongs to the class $\cM_{2k+1}(\mathcal{P}_2 (\bR^d) )$. Finally, suppose that the initial condition satisfies $\bE[| \xi_1|^{2k+1}]< + \infty.$ Then 
\begin{equation*}
\esp{\Phi(\nlaw[T])} - \Phi (\law[X^{0, \xi}_T])
= \sum_{j=1}^{k-1} \frac{C_j}{N^j} + O(\frac1{N^k}), \label{eq: first expansion eq} 
\end{equation*}
where $C_1, \ldots, C_{k-1}$ are constants that do not depend on $N$.
\end{theorem}
\begin{proof}
We commence the proof by noting that  $\Phi$, $b$ and $\sigma$ all belong to $\cM_{2k+1}$, therefore it follows from Theorem \ref{eq:generalisationmainresult} that $\cV^{(1)} \in \cM_{2k+1}$. As in the proof of Theorem \ref{eq:main result with regularity}, we prove by induction on $m \in \{1, \ldots, k \}$ in order to establish that for each $m \in \{1, \ldots, k \}$, $\cV^{(m)} \in \cM_{2k-2m+3}$.
By Theorem \ref{eq:main result with regularity}, Definition \ref{de expansion function} is well-posed for $m \in \{1, \ldots, k\}$. Therefore, by Theorem  \ref{eq:mainresultwithoutregularity}, we have 
\begin{equation}
\esp{\Phi(\nlaw[T])} - \Phi (\law[X^{0, \xi}_T])
= \sum_{j=0}^{k-1} \frac1{N^j}\left(C_j +  \cI^N_{j+1} \right) + O(\frac1{N^k}), \label{eq: full expansion part i}
\end{equation}
for some constants $C_0 = 0, C_1, \ldots, C_{k-1}>0$, where
\[ \begin{cases} 
      \cI^N_1 :=  \esp{ {\cV}(0,\nlaw[0]) - {\cV}(0,\law[\xi])},  \\
       \\
      \cI^N_{j+1} := \int_{\Delta^{j}_T}\left(\esp{ \cV^{(j+1)}((\tau,0),\nlaw[0])} - \cV^{(j+1)}((\tau,0),\law[\xi]) \right) \ud \tau, \quad  \text{ for } j \in \{ 1, \ldots, k-1 \}. 
   \end{cases}
\]
Recall that $\cV^{(1)} \in \cM_{2(k+1)-1}$. By Remark \ref{re poly growth}(i) and Theorem \ref{ theorem second order expansion}, 
\begin{equation}
  \cI^N_1 = \esp{ {\cV}(0,\nlaw[0]) - {\cV}(0,\law[\xi])}
= \sum_{\ell=1}^{k-1} \frac{C^{(1)}_\ell}{N^\ell} + O(\frac1{N^k}),  
\label{eq: full expansion part ii}
\end{equation}
for some constants $C^{(1)}_1, \ldots, C^{(1)}_{k-1} >0$. Similarly, for every $j \in \{1, \ldots, k-1 \}$, since  $\cV^{(j+1)} \in \cM_{2(k-j+1)-1}$, it also follows  by Remark \ref{re poly growth}(i) and Theorem \ref{ theorem second order expansion} \footnote{Note that for $U \in \cM_{2(k-j+1)-1}(\Delta^{j+1}_T \times \mathcal{P}_2 ( \bR^d))$, the constant $C$ in  Lemma \ref{le polynomial growth} is uniform in $\mathbf{t} \in \Delta^{j+1}_T$.  Therefore, the constant $C$ in the same inequality \eqref{eq: poly growth second expansion} in Theorem \ref{ theorem second order expansion} is also uniform in $\mathbf{t} \in \Delta^{j+1}_T$. The fact that the constants $C^{(j)}_1, \ldots, C^{(j)}_{k-j-1}$ are well-defined follows from a similar argument as the first part of the proof of Theorem \ref{eq:mainresultwithoutregularity}.} that 
\begin{equation}
\cI^N_{j+1} 
= \sum_{\ell=1}^{k-j-1} \frac{C^{(j)}_\ell}{N^\ell} + O(\frac1{N^{k-j}}), \label{eq: full expansion part iii}
\end{equation}
for some constants $C^{(j)}_1, \ldots, C^{(j)}_{k-j-1} >0$. The result follows by combining \eqref{eq: full expansion part i}, \eqref{eq: full expansion part ii} and \eqref{eq: full expansion part iii}.
\end{proof}

\section{Proof of Theorem \ref{eq:generalisationmainresult}}
Let $\Phi: \cP_2( \bR^d) \to \bR$ be a Borel-measurable function. In this section, we study the smoothness of the function $\cV: [0,t] \times \cP_2( \bR^d) \to \bR$ defined by 
\[
\cV(s, \mu) = \Phi( \law[{X^{s, \mu}_{t}}] ).
\]
There are various methods of establishing smoothness of functions of this form in the literature. One way involves considering PDE \eqref{eq pde measure} and proving regularity properties of the solution to this PDE (\cite{cardaliaguet2015master}).

The method of Malliavin calculus is adopted in \cite{crisan2017smoothing}. This paper proves smoothness of $\cV$, for $\Phi$ being in the form
$$ \Phi (\mu) = \int_{\bR^d} \zeta(y) \,  \mu (dy), $$ 
where $\zeta: \bR^d \to \bR$ is infinitely differentiable with bounded partial derivatives.

Article \cite{de2015strong} considers the method of parametrix. We represent $\cV$  in terms of the transition density $p(s,\mu;t',y';t,y)$ of $X^{s, x, \mu}_t$ (defined below in \eqref{eq:McKeanflow1'}). This method is applied to the case in which $b$ and $\sigma$ are of the form
    $$ b(x, \mu) = \int_{\bR^d} B(x,y) \mu (dy), \quad \quad \sigma(x, \mu) = \int_{\bR^d} \Sigma(x,y) \mu (dy),$$ 
    for some functions  $B: \bR^d \times \bR^d \to \bR^d$ and $\Sigma: \bR^d \times \bR^d  \to \bR^d \otimes \bR^d$.
Nonetheless, it is not clear whether this method can be applied to $b$ and $\sigma$ with more general forms.

\vspace{2mm}
We follow here a different route.

\vspace{2mm}
  
\noindent \emph{\textbf{Framework of analysis.}}
We adopt the `variational' approach employed in \cite{buckdahn2017mean}. The core idea is to prove smoothness of $\cV$ by viewing the lift of $\cV$ as a composition of the map $\xi \mapsto X^{s, \xi}_{t}$ and the lift of $\Phi$. As \cite{buckdahn2017mean} already proves smoothness of derivatives in measure up to the second order, we generalise that result to an arbitrary order. 

The analysis of variational derivatives of solutions to classical SDEs is rather well-understood in the literature (\cite{friedman2012stochastic}, \cite{krylov1999kolmogorov}). As differentiation in the direction of measure leads to rather complicated expressions, we restrict ourselves to the following special case in this section. This captures the key difficulty of this approach. The general case can be handled in an analogous way.

We consider the forward system $\big(\{X^{s, \xi}_t \}_{t \in [s,T]}, \{X^{s,x,\mu}_t \}_{t \in [s,T]} \big) $, $\xi \sim \mu$, which takes the form 
 \begin{numcases}{} 
      X^{s, \xi}_t & = $\xi + \int_s^t \sigma(\law[X^{s, \xi}_r]) \,dW_r, \quad t \in [s,T]$,  \label{eq:Mckean flow special}  \\
      & \nonumber \\ 
      X^{s, x,\mu}_t & = $x +  \int_s^t  \sigma ( \law[{X^{s,  \mu}_r}])  \,dW_r, \quad t \in [s,T], \quad x \in \bR,$  \label{eq:McKeanflow1'}
   \end{numcases}
for some Borel-measurable function $\sigma: \cP_2(\bR) \to \bR$ and one-dimensional Brownian motion $W$. $\{X^{s, x,\xi}_t \}_{t \in [s,T]}$ is also called the
 decoupled process, as it no longer depends on the law of itself. 

For any sub-$\sigma$-algebra $\cG$, let $L^2(\cG)$ denote the set of all random variables in $L^2(\Omega, \cG, \bP; \bR^d)$. 
Let $\{ \cF_t \}_{t \in [0,T]}$ (resp. $\{ \cF^{(n)}_t \}_{t \in [0,T]}$)  denote the filtration generated by Brownian motion $W=\{ W_t \}_{t \in [0,T]}$ (resp. $\{ W^{(n)}_t \}_{t \in [0,T]}$). Let $\xi$ be a random variable in $L^2 ( \mathcal{F}_s)$. For simplicity of notations, in the following calculations, we shall denote the law $\law[\xi]$ by $[\xi]$. \\
\emph{\textbf{First order derivative of }$[\xi] \mapsto X^{s,x,[\xi]}$. }
We start our analysis by analysing the smoothness of the map $[\xi] \mapsto X^{s,x,[\xi]}_t$.  
Suppose that the lift of $[\xi] \mapsto X^{s,x,[\xi]}_t$ with values in $L^2$
$$ L^2 ( \mathcal{F}_s) \to L^2 ( \mathcal{F}_t) ; \quad \xi \mapsto X^{s,x,[\xi]}_t $$
is Fr\'{e}chet differentiable with its Fr\'{e}chet derivative given by
$$ L^2 ( \mathcal{F}_s)  \to L(L^2 ( \mathcal{F}_s),L^2 ( \mathcal{F}_t)) ; \quad \xi \mapsto \bigg( \eta \mapsto  {\hat{\bE}} \big[  U^{s,x,[\xi]}_t ({\hat{\xi}})    {\hat{\eta}} \big] 
 \bigg) ,$$  
 for some real-valued process 
 $\{U^{s,x,[\xi]}_t (y) \}_{t \in [s,T]}$ that is adapted to $\{ \cF_t \}_{t \in [s,T]}$. Then we define the derivative of $ X^{s,x,[\xi]}_t$ with respect to the measure component by
\begin{equation} \pmu X^{s,x,[\xi]}_t (y) := U^{s,x,[\xi]}_t (y), \quad \quad t \in [s,T], \quad  x,y \in \bR. \label{eq:deriinmeasureFrechetdef} \end{equation}
 The next theorem computes $\pmu X^{s,x,[\xi]}_t (y)$ explicitly.
 \begin{theorem} \label{eq: first order regularity X} 
  Suppose that $\sigma \in \cM_1 (\cP_2(\bR))$. Then $\pmu X^{s,x,[\xi]}_t (y)$ exists and is the unique solution of the SDE \[
\pmu X^{s,x,[\xi]}_t (y)  =   \int_s^t {\bE}^{(1)} \bigg[ (\pmu \sigma)\Big(  [{X^{s,\xi}_r}], \big({X^{(1)}} \big)^{s, y, [{\xi}]}_r \Big)  +  (\pmu \sigma)\Big(  [{X^{s,\xi}_r}], \big({X^{(1)}} \big)^{s,{\xi^{(1)}}}_r \Big) \pmu \big({X^{(1)}} \big)^{s, x, [{\xi}]}_r (y) \bigg] \,dW_r.
\]
 \end{theorem} 
 \begin{proof}
 The proof is done in \cite{buckdahn2017mean}, but is included for completeness. 
 We first define the $L^2$-directional derivative $D_{\xi} \big( X^{s,x,[\xi]}_t \big)(\eta)$  of $X^{s,x,[\xi]}_t$ in direction $\eta \in L^2( \cF_s)$, given by
\begin{equation} D_{\xi} \big( X^{s,x,[\xi]}_t \big)(\eta):= \lim_{h \to 0} \frac{1}{h} \Big(X^{s,x,[\xi+h \eta]}_t - X^{s,x,[\xi]}_t \Big), \label{eq:directionderidecouple} \end{equation}
where the limit is interpreted in the $L^2$ sense, i.e.
$$ \lim_{h \to 0} \bE \bigg[ \bigg( \frac{1}{h} \Big(X^{s,x,[\xi+h \eta]}_t - X^{s,x,[\xi]}_t  \Big) - D_{\xi} \big( X^{s,x,[\xi]}_t \big)(\eta) \bigg)^2 \bigg] =0. $$ 
Similarly, the $L^2$-directional derivative of $X^{s,\xi}_t$ in direction $\eta \in L^2( \cF_s)$ is given by
\begin{equation}\lim_{h \to 0} \frac{1}{h} \Big(X^{s,\xi+h \eta}_t - X^{s,\xi}_t  \Big) = \partial_h X^{s, \xi + h \eta}_t \bigg|_{h=0} , \label{eq:directionalderiMK} \end{equation}
where both the limit and the derivative are interpreted in the $L^2$ sense.
We proceed by formal differentiation and obtain that
\begin{eqnarray}
\partial_h X^{s, \xi + h \eta}_t & = & \partial_h \bigg( X^{s, x, [{\xi + h\eta}]}_t \bigg|_{x= \xi + h\eta}  \bigg) \nonumber \\
& = & \bigg( \partial_x X^{s, x, [{\xi + h\eta}]}_t \bigg|_{x= \xi + h\eta} \bigg) \eta + \bigg( \lim_{\nu \to 0} \frac{1}{\nu}  \Big(X^{s,x,[{\xi+(h + \nu)\eta }]}_t - X^{s,x,[{\xi+h \eta}]}_t  \Big) \bigg) \bigg|_{x= \xi + h\eta}. \nonumber 
\end{eqnarray}
Hence,
\begin{equation}
    D_{\xi} \big( X^{s,\xi}_t \big)(\eta)  = \partial_h X^{s, \xi + h \eta}_t \bigg|_{h=0} =  \eta +  D_{\xi} \big( X^{s,x,[\xi]}_t \big)(\eta) \bigg|_{x= \xi}  .
    \label{eq:formaldifXs,xi}
    \end{equation}
Recall that the lift of $\sigma$, i.e. $\widetilde{\sigma}: L^2( \mathcal{F}) \to \bR$, is  defined by
$ \widetilde{\sigma}( \theta) := \sigma([{\theta}]).$ 
By \eqref{eq:directionderidecouple}, \eqref{eq:directionalderiMK},  and \eqref{eq:formaldifXs,xi}, formal differentiation of \eqref{eq:McKeanflow1'} with respect to $\xi$ in the direction $\eta$ gives
\begin{equation}
D_{\xi} \big( X^{s,x,[\xi]}_t \big)(\eta)  =  \int_s^t (D_{\theta} \widetilde{\sigma}) (X^{s,\xi}_r) \Big( \eta +D_{\xi} \big( X^{s,x,[\xi]}_r \big)(\eta) \bigg|_{x= \xi} \Big) \,dW_r. \label{eq:couplingYsxxi1}
\end{equation}
By the definition of derivative in measure of $\sigma$, we can further rewrite \eqref{eq:couplingYsxxi1} as 
\begin{equation}
D_{\xi} \big( X^{s,x,[\xi]}_t \big)(\eta) =  \int_s^t {\bE}^{(1)} \bigg[ (\pmu \sigma)\Big( [{X^{s,\xi}_r}], \big( {X^{(1)}} \big)^{s, {\xi^{(1)}}}_r \Big)\Big( {\eta^{(1)}} + D_{\xi} \big( \big( {{X}^{(1)}} \big)^{s,x,[\xi]}_r \big)(\eta^{(1)}) \bigg|_{x= {\xi}^{(1)}} \Big) \bigg] \,dW_r. \label{eq:couplingYsxxi1'}
\end{equation}
It is then verified rigorously in Lemma 4.2 of \cite{buckdahn2017mean} that $D_{\xi} \big( X^{s,x,[\xi]}_t \big)(\eta) $ is indeed the directional derivative of $X^{s,x,[\xi]}_t$ in direction $\eta \in L^2(\cF_s)$, by using the fact that $\sigma$ is in $\cM_1$.

The next step involves the consideration of a process $\{U^{s,x,[{\xi}]}_t \}_{t \in [s,T]}$ satisfying the SDE 
\begin{equation}
U^{s,x,[\xi]}_t (y)  =   \int_s^t {\bE}^{(1)} \Big[ (\pmu \sigma)\Big(  [{X^{s,\xi}_r}], \big({X^{(1)}} \big)^{s, y, [{\xi}]}_r \Big)  +  (\pmu \sigma)\Big(  [{X^{s,\xi}_r}], \big({X^{(1)}} \big)^{s,{\xi^{(1)}}}_r \Big)  \big({U^{(1)}} \big)^{s, x, [{\xi}]}_r (y) \Big|_{x=\xi^{(1)}} \Big] \,dW_r.
\label{eq:couplingUsxxi1} 
\end{equation}
We write 
\begin{eqnarray}
 &&  \hat{\bE} \bigg[U^{s,x,[{\xi}]}_t ( \hat{\xi}) \hat{\eta}\Big] \nonumber \\
 & = & \int_s^t \hat{\bE} \bigg[  {\bE}^{(1)} \bigg[ (\pmu \sigma)\Big(  [{X^{s,\xi}_r}], \big({X^{(1)}} \big)^{s, y, [{\xi}]}_r \Big) \bigg] \bigg|_{y= \hat{\xi}} \hat{\eta} \bigg] \,dW_r  \nonumber \\
 && + \int_s^t \hat{\bE} \bigg[  {\bE}^{(1)} \bigg[ (\pmu \sigma)\Big(  [{X^{s,\xi}_r}], \big({X^{(1)}} \big)^{s,{\xi^{(1)}}}_r \Big)  \big({U^{(1)}} \big)^{s, x, [{\xi}]}_r (\hat{\xi})  \bigg] \bigg|_{x=\xi^{(1)}} \hat{\eta} \bigg] \, dW_r \nonumber 
 \end{eqnarray}
 and notice that
 \begin{eqnarray}
&&  \hat{\bE} \bigg[  {\bE}^{(1)} \bigg[ (\pmu \sigma)\Big(  [{X^{s,\xi}_r}], \big({X^{(1)}} \big)^{s, y, [{\xi}]}_r \Big)  \bigg] \bigg|_{y= \hat{\xi}} \hat{\eta} \bigg] \nonumber \\ 
 &= &   {\bE}^{(1)}  \bigg[  {\bE}^{(1)} \bigg[ (\pmu \sigma)\Big(  [{X^{s,\xi}_r}], \big({X^{(1)}} \big)^{s, y, [{\xi}]}_r \Big)   \bigg] \bigg|_{y= {\xi}^{(1)}} {\eta^{(1)}} \bigg] \nonumber \\
 & = & {\bE}^{(1)}  \bigg[  {\bE}^{(1)} \bigg[  (\pmu \sigma)\Big(  [{X^{s,\xi}_r}], \big({X^{(1)}} \big)^{s, y, [{\xi}]}_r \Big)  \bigg|_{y= {\xi}^{(1)}} {\eta^{(1)}}  \bigg| \cF^{(1)}_s \bigg] \bigg] \nonumber \\
 & = & {\bE}^{(1)}  \bigg[  (\pmu \sigma)\Big(  [{X^{s,\xi}_r}], \big({X^{(1)}} \big)^{s, {\xi}^{(1)} }_r \Big) {\eta^{(1)}} \bigg], \label{ U directional derivative equality independence argument} 
  \end{eqnarray}
  where the second equality uses the fact that $\big({X^{(1)}} \big)^{s, y, [{\xi}]}$ is $\sigma \{ W^{(1)}_r -W^{(1)}_s \, | \,  r \in [s,t] \} $-adapted and is therefore independent of $\cF^{(1)}_s$, whereas ${\xi}^{(1)}$ and  ${\eta^{(1)}}$ are both $\cF^{(1)}_s$-measurable. The final equality uses the fact that $\big({X^{(1)}} \big)^{s, y, [{\xi}]}_r   \big|_{y= {\xi}^{(1)}} = \big({X^{(1)}} \big)^{s, {\xi}^{(1)} }_r$. We also notice by the Fubini's theorem that
  \begin{eqnarray}
   && \hat{\bE} \bigg[  {\bE}^{(1)} \bigg[ (\pmu \sigma)\Big(  [{X^{s,\xi}_r}], \big({X^{(1)}} \big)^{s,{\xi^{(1)}}}_r \Big)  \big({U^{(1)}} \big)^{s, x, [{\xi}]}_r ( \hat{\xi})  \bigg] \bigg|_{{x=\xi^{(1)}}} \hat{\eta} \bigg] \nonumber \\
     & = &   {\bE}^{(1)} \Big[ (\pmu \sigma)\Big(  [{X^{s,\xi}_r}], \big({X^{(1)}} \big)^{s,{\xi^{(1)}}}_r \Big)  \hat{\bE} \Big[ \big({U^{(1)}} \big)^{s, x, [{\xi}]}_r (\hat{\xi})  \hat{\eta} \Big] \Big|_{{x=\xi^{(1)}}}  \Big]. \label{ U directional derivative equality fubini} 
   \end{eqnarray}
   Therefore, by \eqref{ U directional derivative equality independence argument}  and \eqref{ U directional derivative equality fubini}, we observe that  $D_{\xi} \big( X^{s,x,[\xi]} \big)(\eta)$ and $\hat{\bE} \big[U^{s,x,[{\xi}]} ( \hat{\xi}) \hat{\eta}\big]$ satisfy the same SDE and hence 
\begin{equation}
 D_{\xi} \big( X^{s,x,[\xi]}_t \big)(\eta) = \hat{\bE} \Big[U^{s,x,[{\xi}]}_t ( \hat{\xi}) \hat{\eta}\Big], \quad t \in [s,T], \quad \eta \in L^2(\cF_s). \label{eq:relationYandU}
\end{equation}
We then observe that $U^{s,x,[{\xi}]}_t (y)$ satisfies the same SDE for any $x \in \bR$. Therefore, there is no dependence on $x$ and hence \eqref{eq:couplingUsxxi1} can be rewritten as
\begin{equation}
U^{s,x,[\xi]}_t (y)  =   \int_s^t {\bE}^{(1)} \bigg[ (\pmu \sigma)\Big(  [{X^{s,\xi}_r}], \big({X^{(1)}} \big)^{s, y, [{\xi}]}_r \Big)  +  (\pmu \sigma)\Big(  [{X^{s,\xi}_r}], \big({X^{(1)}} \big)^{s,{\xi^{(1)}}}_r \Big)  \big({U^{(1)}} \big)^{s, x, [{\xi}]}_r (y)  \bigg] \,dW_r. \label{eq: U final formula} 
\end{equation}
Moreover, by the fact that  $\sigma$ is in $\cM_1$, we establish that 
\begin{enumerate}[(i)]
\item \begin{equation} \bE \bigg[ \sup_{t \in [s,T]} \big| U^{s,x,[{\xi}]}_t (y) \big|^2 \bigg] \leq C,  \label{eq: bound regularity} \end{equation}
\item \begin{equation}  \bE \bigg[ \sup_{t \in [s,T]} \big| U^{s,x,[{\xi}]}_t (y) -U^{s,x,[{\xi'}]}_t (y')  \big|^2 \bigg] \leq C \bigg(  |y-y'|^2 + W_2 ([{\xi}],[{\xi'}])^2 \bigg), \label{eq: Lipschitz regularity} \end{equation} 
\end{enumerate}
for any $s \in [0,T]$, $x,y,y' \in \bR$ and $\xi, \xi' \in L^2 (\cF_s)$, for some constant $C>0$. Indeed, \eqref{eq: bound regularity} follows from the boundedness of $\pmu \sigma$ and Gronwall's inequality. \eqref{eq: Lipschitz regularity} follows from the Lipschitz property of $\pmu \sigma$ and Gronwall's inequality, along with the bounds
\[ \begin{cases} 
      \bE \big[ \sup_{t \in [s,T]} \big|X^{s, \xi}_t -  X^{s, \xi'}_t \big|^2 \big] \leq C \bE |\xi - \xi'|^2, & \\
       \sup_{t \in [s,T]} W_2( [X^{s,\xi}_t], [X^{s,\xi'}_t] )^2 \leq C W_2 ( [\xi],[\xi'])^2,  &\\
      \bE \big[ \sup_{t \in [s,T]} \big|X^{s, x, [\xi]}_t -  X^{s, x',[\xi']}_t \big|^2 \big] \leq C \big( |x-x'|^2 + W_2 ( [\xi],[\xi'])^2 \big),   &    \end{cases}
\]
for some constant $C>0$. 
Finally, the bounds \eqref{eq: bound regularity}, \eqref{eq: Lipschitz regularity}  and connection \eqref{eq:relationYandU} allow us to establish that the G\^{a}teaux  derivative
\begin{equation} L^2 ( \mathcal{F}_s)  \to L(L^2 ( \mathcal{F}_s),L^2 ( \mathcal{F}_t)) ; \quad \xi \mapsto \bigg(  \eta \mapsto  D_{\xi} \big( X^{s,x,[\xi]}_t \big)(\eta) \bigg)
 \label{eq:Gateauxbasestep} \end{equation}
is continuous (where the space $L(L^2 ( \mathcal{F}_s),L^2 ( \mathcal{F}_t))$ is equipped with the corresponding operator norm), which proves that \eqref{eq:Gateauxbasestep} is indeed the Fr\'{e}chet derivative of $X^{s,x,[{\xi}]}_t$ with respect to $\xi$. By \eqref{eq:relationYandU}, it follows from the definition of $\pmu X^{s,x,[{\xi}]}_t (y)$ that
\begin{equation} \pmu X^{s,x,[{\xi}]}_t (y) = U^{s,x,[{\xi}]}_t (y), \quad \quad t \in [s,T].  \label{U equals pmu} \end{equation}
 \end{proof}
$ {}$ \\ \emph{\textbf{Higher order derivatives of }$[\xi] \mapsto X^{s,x,[\xi]}$. }  We recall that $\pmu X^{s,x,[{\xi}]}_t (y)$ does not depend on $x$ and hence we define
\begin{equation} \pmu X^{s,[\xi]}_t (y) := \pmu X^{s,x,[{\xi}]}_t (y). \label{no dependence on x} \end{equation}
Subsequently, we define  inductively as in \eqref{eq:generalformulaintro} and \eqref{eq:generalformulamixed},  the $n$th order derivative in measure of $X^{s,x,[{\xi}]}_t$ by
$$ \pnmu X^{s,[{\xi}]}_t (v_1, \ldots, v_n):= \partial^{n-1}_{\mu} \bigg( \pmu X^{s,[{\xi}]}_t (v_1) \bigg) (v_2, \ldots, v_n), \quad \quad t \in [s,T], \, v_1, \ldots, v_n \in \bR^d, $$
and its corresponding mixed derivatives by
$$ \partial^{\beta_n}_{v_n} \ldots \partial^{\beta_1}_{v_1}  \pnmu X^{s,[\xi]}_t (v_1, \ldots, v_n), \quad \quad \ell, \beta_1, \ldots, \beta_n \in \bN \cup \{0 \},$$ 
provided that these derivatives actually exist, where each derivative in $v_i$ is interpreted in the $L^2$ sense. (See Lemma 4.1 in \cite{buckdahn2017mean} for its precise meaning.)

Next, we generalise the multi-index notation and the class $\cM_k$ to include derivatives of $X^{s,x, [{\xi}]}_t$.

\begin{definition}[Multi-index notation for derivatives of $X^{s,x, [{\xi}]}_t$]
Let $(n,\bm{\beta})$ be a {multi-index}. Then $D^{(n, \bm{\beta})} X^{s,[\xi]}_t (v_1, \ldots, v_n)$ is defined by 
$$ D^{(n,\bm{\beta})} X^{s,[\xi]}_t (v_1, \ldots, v_n):=  \partial^{{\beta}_n}_{v_n} \ldots \partial^{{\beta}_1}_{v_1}  \pnmu X^{s,[\xi]}_t  (v_1, \ldots, v_n), $$
if this derivative is well-defined.
\end{definition}

\begin{definition}[Class $\cM_k$  of $k$th order differentiable functions of $X^{s,x,[{\xi}]}$] ${}$ \label{eq:classmkX} \\
The process $X^{s,x,[{\xi}]}= \{ X^{s,x,[{\xi}]}_t  \}_{t \in [s, T]}$ belongs to class $\cM_k(X^{s,x,[{\xi}]}) $, if $D^{(n, \bm{\beta})} X^{s,[\xi]}_t (v_1, \ldots, v_n)$ exists for every multi-index $(n,\bm{\beta})$ such that $|(n,\bm{\beta})| \leq k$ and 
\begin{enumerate}[(a)]
\item \begin{equation} \bE \bigg[ \sup_{t \in [s,T]} \big| D^{(n,\bm{\beta})} X^{s,[\xi]}_t (v_1, \ldots, v_n) \big|^2 \bigg] \leq C, \label{eq:boundFrechet} \end{equation}
\item \begin{eqnarray} 
&& \bE \bigg[ \sup_{t \in [s,T]} \big| D^{(n,\bm{\beta})} X^{s,[\xi]}_t (v_1, \ldots, v_n) -D^{(n,\bm{\beta})} X^{s,[\xi']}_t (v'_1, \ldots, v'_n)  \big|^2 \bigg] \nonumber \\
& \leq & C \bigg(  \sum_{i=1}^n|v_i-v'_i|^2 + W_2 ([{\xi}],[{\xi'}])^2 \bigg), \label{eq:LipFrechet}
\end{eqnarray}
\end{enumerate}
for any $s \in [0,T]$, $v_1,v'_1,\ldots, v_n, v'_n \in \bR^d$ and $\xi, \xi' \in L^2 (\cF_s)$, for some constant $C>0$.
\end{definition}
The following theorem extends Theorem \ref{eq: first order regularity X} to higher order derivatives. It uses the notations
$$ \Lambda_{i,k} := \Big\{ \theta: \{1, \ldots , i \} \to \{1, \ldots, k \} \Big| \, \, \theta \text{ is a strictly increasing function} \Big\}, \quad \quad i \in \{1, \ldots, k \},$$
and
$$R_k:= \Big\{ y= \big( {{y}}_{(j, \ell)} \big)_{ 1 \leq j, \ell \leq k}  \, \Big| \, {{y}}_{(j, \ell)} \in \bR \Big\}, \quad T_k := \Big\{  z= \big( {{z}}_{(j,i,\theta)} \big)_{\substack{1 \leq j,i \leq {k}\\ \theta \in \Lambda_{i,{k}}}}  \, \Big| \, {{z}}_{(j,i,\theta)}  \in \bR  \Big\}.$$ 
For any function $F_k: \cP_2(\bR) \times \bR^k \times R_k \times T_k \to \bR$,
$\partial_{x_{j}} {F}_{k}$ denotes the corresponding partial derivative with respect to the second component of ${F}_{k}$. $\partial_{y_{(j,\ell)}} {F}_{k}$ denotes the corresponding partial derivative with respect to the third component of ${F}_{k}$.  $\partial_{z_{(j,i,\theta)}} {F}_{k} $ denotes the corresponding partial derivative with respect to the fourth component of ${F}_{k}$. 
\begin{theorem}  \label{eq:purederimeasure}  Suppose that $\sigma$ is $\cM_K( \cP_2(\bR))$. Then,
   for any $k \in \{1, \ldots, K \},$  $t \in [s,T],$ the $k$th order derivative in measure $\partial^{k}_{\mu} X^{s,[\xi]}_t (v_1, \ldots, v_k) $ exists and satisfies \eqref{eq:boundFrechet} and \eqref{eq:LipFrechet}. In particular, it is the unique solution of an SDE given by
   \begin{eqnarray}
\partial^{k}_{\mu} X^{s, [\xi]}_t (v_1, \ldots, v_{k}) & = & \int_s^t \bE^{(1)} \bE^{(2)} \ldots \bE^{({k})} \bigg[ {F}_{k} \bigg( [X^{s, \xi}_r], \Big(  \big( X^{(j)} \big)^{s, {\xi}^{(j)}}_r  \Big)_{ 1 \leq j \leq {k}} ,  \Big(  \big( X^{(j)} \big)^{s, v_{\ell}, [\xi]}_r  \Big)_{ 1 \leq j,\ell \leq {k}},  \nonumber \\
&& \Big( \partial^i_{\mu} \big( X^{(j)} \big)^{s, [\xi] }_r (v_{\theta(1)}, \ldots, v_{\theta(i)}) \Big)_{\substack{1 \leq j,i \leq {k}\\ \theta \in \Lambda_{i,{k}}}} \bigg) \bigg] \,dW_r,  \label{eq: general formula kth derivative X}   
\end{eqnarray}
where $F_k: \cP_2(\bR) \times \bR^k \times R_k \times T_k \to \bR$ is defined by the recurrence relation 
\begin{eqnarray}
  && F_{k+1} \Big( \mu, (x_j)_{1 \leq j \leq k+1}, (y_{(j,\ell)} )_{1 \leq j, \ell \leq k+1}, \big( {{z}}_{(j,i,\theta)} \big)_{\substack{1 \leq j,i \leq {k+1}\\ \theta \in \Lambda_{i,{k+1}}}} \Big)  \nonumber \\
  & = & \pmu F_{k} \Big( \mu, (x_j)_{1 \leq j \leq k}, (y_{(j,\ell)} )_{1 \leq j, \ell \leq k}, \big( {{z}}_{(j,i,\theta)} \big)_{\substack{1 \leq j,i \leq {k}\\ \theta \in \Lambda_{i,{k}}}} , y_{(k+1,k+1)}\Big) \nonumber \\
  && + \pmu F_{k} \Big( \mu, (x_j)_{1 \leq j \leq k}, (y_{(j,\ell)} )_{1 \leq j, \ell \leq k}, \big( {{z}}_{(j,i,\theta)} \big)_{\substack{1 \leq j,i \leq {k}\\ \theta \in \Lambda_{i,{k}}}} , x_{k+1}\Big)  {{z}}_{(k+1,1,P_{k+1})} \nonumber \\
  &&  + \sum_{j=1}^k \partial_{x_j} F_{k} \Big( \mu, \big( x_1, \ldots, x_{j-1}, y_{(j,k+1)}, x_{j+1}, \ldots, x_k \big) , (y_{(j,\ell)} )_{1 \leq j, \ell \leq k}, \big( {{z}}_{(j,i,\theta)} \big)_{\substack{1 \leq j,i \leq {k}\\ \theta \in \Lambda_{i,{k}}}} \Big)  \nonumber  \\
  && + \sum_{j=1}^k \partial_{x_j} F_{k} \Big( \mu, (x_{j})_{1 \leq j \leq k}, (y_{(j,\ell)} )_{1 \leq j, \ell \leq k}, \big( {{z}}_{(j,i,\theta)} \big)_{\substack{1 \leq j,i \leq {k}\\ \theta \in \Lambda_{i,{k}}}} \Big)  z_{(j,1,P_{k+1})} \nonumber \\
  && + \sum_{j,\ell=1}^k \partial_{y_{(j,\ell)}} F_{k} \Big( \mu, (x_{j})_{1 \leq j \leq k}, (y_{(j,\ell)} )_{1 \leq j, \ell \leq k}, \big( {{z}}_{(j,i,\theta)} \big)_{\substack{1 \leq j,i \leq {k}\\ \theta \in \Lambda_{i,{k}}}} \Big)  z_{(j,1,P_{k+1})} \nonumber \\
  && + \sum_{j,i=1}^k \sum_{\theta \in \Lambda_{i,k}} \partial_{z_{(j,i,\theta)}}  F_{k} \Big( \mu, (x_{j})_{1 \leq j \leq k}, (y_{(j,\ell)} )_{1 \leq j, \ell \leq k}, \big( {{z}}_{(j,i,\theta)} \big)_{\substack{1 \leq j,i \leq {k}\\ \theta \in \Lambda_{i,{k}}}} \Big) z_{(j,i+1, \theta_{k+1})}, \quad k \in \{1, \ldots, K-1 \}, \nonumber \\ 
  && \label{recurrence Fk}
\end{eqnarray}
where $P_{k+1} \in \Lambda_{1, k+1}$ is defined by $P_{k+1}(1) = k+1$ and for each $\theta \in \Lambda_{i, k}$, the function $\theta_{k+1}  \in \Lambda_{i+1, k+1}$ is defined such that $ \theta_{k+1} \big|_{ \{1, \ldots, i \}} = \theta $ and $\theta_{k+1} (i+1) = k+1$. Moreover, $F_1$ is given by
\begin{equation}  F_1 (\mu, x,y,z) =  \pmu \sigma ( \mu, y) + \pmu \sigma ( \mu, x) z. \label{eq: F1 expression} \end{equation}
 \end{theorem}

\begin{proof}
We remark that the functions $F_k$, $k \in \{1, \ldots, K \}$, are well-defined, since $\sigma \in \cM_K$. We proceed by strong induction on $k \in \{1, \ldots, K\}$. The base step $k=1$ is done in Theorem \ref{eq: first order regularity X}. In particular, \eqref{eq: U final formula} verifies \eqref{eq: F1 expression}. The main arguments in the induction step are the same as the base step. Suppose that the statement holds for all $k \in \{1, \ldots, k^{*} \}$, where $k^{*} \in \{1, \ldots, K-1 \}$. Then, in particular, $\partial^{k^{*}}_{\mu} X^{s, [\xi]}_t (v_1, \ldots, v_{k^{*}})$ satisfies the SDE
\begin{eqnarray}
\partial^{k^{*}}_{\mu} X^{s, [\xi]}_t (v_1, \ldots, v_{k^{*}}) & = & \int_s^t \bE^{(1)} \bE^{(2)} \ldots \bE^{({k^{*}})} \bigg[ {F}_{k^{*}} \bigg( [X^{s, \xi}_r], \Big(  \big( X^{(j)} \big)^{s, {\xi}^{(j)}}_r  \Big)_{ 1 \leq j \leq {k^{*}}} ,  \Big(  \big( X^{(j)} \big)^{s, v_{\ell}, [\xi]}_r  \Big)_{ 1 \leq j,\ell \leq {k^{*}}},  \nonumber \\
&& \Big( \partial^i_{\mu} \big( X^{(j)} \big)^{s, [\xi] }_r (v_{\theta(1)}, \ldots, v_{\theta(i)}) \Big)_{\substack{1 \leq j,i \leq {k^{*}}\\ \theta \in \Lambda_{i,{k^{*}}}}} \bigg) \bigg] \,dW_r. \label{eq: kth order derivative long formula}  
\end{eqnarray}
Let $\tilde{F}_{k^{*}}$ be the lift of $F_{k^{*}}$. In the following expression, $\partial_{x} \tilde{F}_{k^{*}} $ denotes the partial derivative with respect to the lifted component of $\tilde{F_k}$.   As in \eqref{eq:formaldifXs,xi} and \eqref{eq:couplingYsxxi1}, we formally differentiate \eqref{eq: kth order derivative long formula} with respect to $\xi$ in the direction $\eta$ to obtain the directional derivative  
  \begin{eqnarray}
&& D_{\xi} \big( \partial^{{k^{*}}}_{\mu} X^{s, [\xi]}_t (v_1, \ldots, v_{k^{*}}) \big)(\eta) \nonumber \\
& = & \int_s^t \bE^{(1)} \bE^{(2)} \ldots \bE^{({k^{*}})} \bigg[ \partial_{x} \tilde{F}_{k^{*}} \bigg( [X^{s, \xi}_r], \Big(  \big( X^{(j)} \big)^{s, {\xi}^{(j)}}_r  \Big)_{ j} ,  \Big(  \big( X^{(j)} \big)^{s, v_{\ell}, [\xi]}_r  \Big)_{  j,\ell},  \nonumber \\
&& \Big( \partial^i_{\mu} \big( X^{(j)} \big)^{s, [\xi] }_r (v_{\theta(1)}, \ldots, v_{\theta(i)}) \Big)_{j,i,\theta} \bigg)  \bigg(\eta + D_{\xi} (X^{s,x,[\xi]}_r)(\eta) \Big|_{x= \xi} \bigg) \bigg] \,dW_r \nonumber \\
&& + \sum_{j=1}^{{k^{*}}} \int_s^t \bE^{(1)} \bE^{(2)} \ldots \bE^{({k^{*}})} \bigg[ \partial_{x_j} \tilde{F}_{k^{*}} \bigg( [X^{s, \xi}_r], \Big(  \big( X^{(j)} \big)^{s, {\xi}^{(j)}}_r  \Big)_{ j} ,  \Big(  \big( X^{(j)} \big)^{s, v_{\ell}, [\xi]}_r  \Big)_{  j,\ell},  \nonumber \\
&& \Big( \partial^i_{\mu} \big( X^{(j)} \big)^{s, [\xi] }_r (v_{\theta(1)}, \ldots, v_{\theta(i)}) \Big)_{j,i,\theta} \bigg)  \bigg( \eta^{(j)} +  D_{\xi} \big( \big( X^{(j)} \big)^{s, x, [\xi] }_r  \big)(\eta^{(j)}) \Big|_{x= \xi^{(j)}} \bigg) \bigg]  \,dW_r \nonumber \\
&& + \sum_{j,\ell=1}^{k^{*}} \int_s^t \bE^{(1)} \bE^{(2)} \ldots \bE^{({k^{*}})} \bigg[ \partial_{y_{(j,\ell)}} \tilde{F}_{k^{*}} \bigg( [X^{s, \xi}_r], \Big(  \big( X^{(j)} \big)^{s, {\xi}^{(j)}}_r  \Big)_{ j} ,  \Big(  \big( X^{(j)} \big)^{s, v_{\ell}, [\xi]}_r  \Big)_{  j,\ell},  \nonumber \\
&& \Big( \partial^i_{\mu} \big( X^{(j)} \big)^{s, [\xi] }_r (v_{\theta(1)}, \ldots, v_{\theta(i)}) \Big)_{j,i,\theta} \bigg)  D_{\xi} \Big( \big(X^{(j)} \big)^{s, v_{\ell}, [\xi]}_r \Big)(\eta^{(j)}) \bigg] \,dW_r \nonumber \\
&& + \sum_{j,i=1}^{k^{*}} \sum_{\theta \in \Lambda_{i,{k^{*}}}}  \int_s^t \bE^{(1)} \bE^{(2)} \ldots \bE^{({k^{*}})} \bigg[ \partial_{z_{(j,i,\theta)}} \tilde{F}_{k^{*}} \bigg( [X^{s, \xi}_r], \Big(  \big( X^{(j)} \big)^{s, {\xi}^{(j)}}_r  \Big)_{ j} ,  \Big(  \big( X^{(j)} \big)^{s, v_{\ell}, [\xi]}_r  \Big)_{  j,\ell},  \nonumber \\
&& \Big( \partial^i_{\mu} \big( X^{(j)} \big)^{s, [\xi] }_r (v_{\theta(1)}, \ldots, v_{\theta(i)}) \Big)_{j,i,\theta} \bigg)  D_{\xi} \big(\partial^i_{\mu} \big( X^{(j)} \big)^{s, [\xi] }_r (v_{\theta(1)}, \ldots, v_{\theta(i)}) \big)(\eta^{(j)})  \bigg] \,dW_r. \nonumber \label{eq: directional derivatives expansion in terms of directional derivatives} 
  \end{eqnarray}

    We then recall that the following directional derivatives can be represented as 
 $$D_{\xi} (X^{s,x,[\xi]}_r)(\eta) \Big|_{x= \xi}=  \hat{\bE} \Big[ \pmu X^{s, [\xi]}_r (\hat{\xi}) \hat{\eta} \Big], \quad \quad D_{\xi} \Big( \big(X^{(j)} \big)^{s, v_{\ell}, [\xi]}_r \Big)(\eta^{(j)}) = \hat{\bE} \Big[ \pmu \big(X^{(j)} \big)^{s, [\xi]}_r (\hat{\xi}) \hat{\eta} \Big],$$ 
 \[
 D_{\xi} \big( \big( X^{(j)} \big)^{s, x, [\xi] }_r \big)(\eta^{(j)}) \Big|_{x= \xi^{(j)}} =  \hat{\bE} \Big[ \partial_{\mu} \big( X^{(j)} \big)^{s, [\xi] }_r ( \hat{\xi}) \hat{\eta} \Big]
\]
 and
 \[
 D_{\xi} \big( \partial^i_{\mu} \big( X^{(j)} \big)^{s,  [\xi] }_r (v_{\theta(1)}, \ldots, v_{\theta(i)})\big)(\eta^{(j)})  =  \hat{\bE} \Big[ \partial^{i+1}_{\mu} \big( X^{(j)} \big)^{s, [\xi] }_r (v_{\theta(1)}, \ldots, v_{\theta(i)}, \hat{\xi}) \hat{\eta} \Big], \quad i \in \{1, \ldots ,{k^{*}}-1 \}.  
\]
We can therefore rewrite \eqref{eq: directional derivatives expansion in terms of directional derivatives} as
\begin{eqnarray}
&& D_{\xi} \big( \partial^{{k^{*}}}_{\mu} X^{s, [\xi]}_t (v_1, \ldots, v_{k^{*}}) \big)(\eta) \nonumber \\
& = & \int_s^t \bE^{(1)} \bE^{(2)} \ldots \bE^{({k^{*}})} \bE^{({k^{*}}+1)}\bigg[ \partial_{\mu} F_{k^{*}} \bigg( [X^{s, \xi}_r], \Big(  \big( X^{(j)} \big)^{s, {\xi}^{(j)}}_r  \Big)_{ j} ,  \Big(  \big( X^{(j)} \big)^{s, v_{\ell}, [\xi]}_r  \Big)_{  j,\ell},  \nonumber \\
&& \Big( \partial^i_{\mu} \big( X^{(j)} \big)^{s, [\xi] }_r (v_{\theta(1)}, \ldots, v_{\theta(i)}) \Big)_{j,i,\theta}, \big(X^{({k^{*}}+1)}  \big)^{s, \xi^{({k^{*}}+1)}}_r \bigg) \bigg(\eta^{({k^{*}}+1)} + \hat{\bE} \Big[ \pmu \big(X^{({k^{*}}+1)}  \big)^{s, [\xi]}_r (\hat{\xi}) \hat{\eta} \Big]  \bigg) \bigg] \,dW_r \nonumber \\
&& + \sum_{j=1}^{{k^{*}}} \int_s^t \bE^{(1)} \bE^{(2)} \ldots \bE^{({k^{*}})} \bigg[ \partial_{x_j} F_{k^{*}} \bigg( [X^{s, \xi}_r],  \Big(  \big( X^{(j)} \big)^{s, {\xi}^{(j)}}_r  \Big)_{ j} ,  \Big(  \big( X^{(j)} \big)^{s, v_{\ell}, [\xi]}_r  \Big)_{  j,\ell},  \nonumber \\
&& \Big( \partial^i_{\mu} \big( X^{(j)} \big)^{s, [\xi] }_r (v_{\theta(1)}, \ldots, v_{\theta(i)}) \Big)_{j,i,\theta} \bigg) \bigg(  \eta^{(j)} + \hat{\bE} \Big[ \partial_{\mu} \big( X^{(j)} \big)^{s, [\xi] }_r ( \hat{\xi}) \hat{\eta} \Big] \bigg) \bigg]  \,dW_r \nonumber \\
&& + \sum_{j,\ell=1}^{k^{*}} \int_s^t \bE^{(1)} \bE^{(2)} \ldots \bE^{({k^{*}})} \bigg[ \partial_{y_{(j,\ell)}} F_{k^{*}} \bigg( [X^{s, \xi}_r], \Big(  \big( X^{(j)} \big)^{s, {\xi}^{(j)}}_r  \Big)_{ j} ,  \Big(  \big( X^{(j)} \big)^{s, v_{\ell}, [\xi]}_r  \Big)_{  j,\ell},  \nonumber \\
&& \Big( \partial^i_{\mu} \big( X^{(j)} \big)^{s, [\xi] }_r (v_{\theta(1)}, \ldots, v_{\theta(i)}) \Big)_{j,i,\theta}  \bigg) \hat{\bE} \Big[ \pmu \big(X^{(j)} \big)^{s, [\xi]}_r (\hat{\xi}) \hat{\eta} \Big] \bigg] \,dW_r  \nonumber \\
&& + \sum_{j=1}^{k^{*}} \sum_{i=1}^{{k^{*}}-1}  \sum_{\theta \in \Lambda_{i,{k^{*}}}} \int_s^t \bE^{(1)} \bE^{(2)} \ldots \bE^{({k^{*}})} \bigg[ \partial_{z_{(j,i,\theta)}} F_{k^{*}} \bigg( [X^{s, \xi}_r], \Big(  \big( X^{(j)} \big)^{s, {\xi}^{(j)}}_r  \Big)_{ j} ,  \Big(  \big( X^{(j)} \big)^{s, v_{\ell}, [\xi]}_r  \Big)_{  j,\ell},  \nonumber \\
&& \Big( \partial^i_{\mu} \big( X^{(j)} \big)^{s, [\xi] }_r (v_{\theta(1)}, \ldots, v_{\theta(i)}) \Big)_{j,i,\theta}  \bigg)  \hat{\bE} \Big[ \partial^{i+1}_{\mu} \big( X^{(j)} \big)^{s, [\xi] }_r (v_{\theta(1)}, \ldots, v_{\theta(i)}, \hat{\xi}) \hat{\eta} \Big] \bigg] \,dW_r \nonumber \\
&& + \sum_{j=1}^{k^{*}}  \int_s^t \bE^{(1)} \bE^{(2)} \ldots \bE^{({k^{*}})} \bigg[ \partial_{z_{(j,{k^{*}},\mathbf{I_{k^{*}}})}} F_{k^{*}} \bigg( [X^{s, \xi}_r], \Big(  \big( X^{(j)} \big)^{s, {\xi}^{(j)}}_r  \Big)_{ j} ,  \Big(  \big( X^{(j)} \big)^{s, v_{\ell}, [\xi]}_r  \Big)_{  j,\ell},  \nonumber \\
&& \Big( \partial^i_{\mu} \big( X^{(j)} \big)^{s, [\xi] }_r (v_{\theta(1)}, \ldots, v_{\theta(i)}) \Big)_{j,i,\theta}  \bigg)  D_{\xi} \Big( \partial^{{k^{*}}}_{\mu} \big( X^{(j)} \big)^{s, [\xi]}_t (v_1, \ldots, v_{k^{*}}) \Big)(\eta) \bigg] \,dW_r, \nonumber \\
\end{eqnarray}
where, on the second last line, $\mathbf{I_{k^{*}}}$ denotes the identity function from $\{ 1, \ldots, {k^{*}} \}$ to itself. 
We now define a process $ \big\{ \big( U_{{k^{*}}+1} \big)^{s,[\xi]}_t (v_1, \ldots, v_{{k^{*}}+1}) \big\}_{t \in [s,T]}$ that satisfies the SDE
\begin{eqnarray}
&& \big( U_{{k^{*}}+1}  \big) ^{s,[\xi]}_t (v_1, \ldots, v_{{k^{*}}+1}) \nonumber \\
& = & \int_s^t \bE^{(1)} \bE^{(2)} \ldots \bE^{({k^{*}})} \bE^{({k^{*}}+1)} \bigg[ \pmu F_{k^{*}} \bigg( [X^{s, \xi}_r], \Big(  \big( X^{(j)} \big)^{s, {\xi}^{(j)}}_r  \Big)_{ j} ,  \Big(  \big( X^{(j)} \big)^{s, v_{\ell}, [\xi]}_r  \Big)_{  j,\ell},  \nonumber \\
&& \Big( \partial^i_{\mu} \big( X^{(j)} \big)^{s, [\xi] }_r (v_{\theta(1)}, \ldots, v_{\theta(i)}) \Big)_{j,i,\theta}, \big( X^{({k^{*}}+1)} \big)^{s, v_{{k^{*}}+1}, [\xi]}_r  \bigg)     \bigg] \,dW_r \nonumber \\
&& + \int_s^t \bE^{(1)} \bE^{(2)} \ldots \bE^{({k^{*}})} \bE^{({k^{*}}+1)} \bigg[ \pmu F_{k^{*}} \bigg( [X^{s, \xi}_r], \Big(  \big( X^{(j)} \big)^{s, {\xi}^{(j)}}_r  \Big)_{ j} ,  \Big(  \big( X^{(j)} \big)^{s, v_{\ell}, [\xi]}_r  \Big)_{  j,\ell},  \nonumber \\
&& \Big( \partial^i_{\mu} \big( X^{(j)} \big)^{s, [\xi] }_r (v_{\theta(1)}, \ldots, v_{\theta(i)}) \Big)_{j,i,\theta}, \big( X^{({k^{*}}+1)} \big)^{s,  \xi^{({k^{*}}+1)}}_r  \bigg) \partial_{\mu} \big( X^{({k^{*}}+1)} \big)^{s, [\xi]}_r   (v_{{k^{*}}+1}) \bigg] \,dW_r \nonumber \\
&& + \sum_{j=1}^{{k^{*}}} \int_s^t  \bE^{(1)} \bE^{(2)} \ldots \bE^{({k^{*}})} \bigg[ \partial_{x_{j}} F_{k^{*}}  \bigg( [X^{t, \xi}_r], \Big(  \big( X^{(1)} \big)^{s, {\xi}^{(1)}}_r, \ldots, \big( X^{(j-1)} \big)^{s, {\xi}^{(j-1)}}_r, \big( X^{(j)} \big)^{s, v_{k^{*}+1}, [\xi]}_r ,  \nonumber \\
&& \big( X^{(j+1)} \big)^{s, {\xi}^{(j+1)}}_r ,  \ldots,  \big( X^{(k^{*})} \big)^{s, {\xi}^{(k^{*})}}_r \Big) ,  \Big(  \big( X^{(j)} \big)^{s, v_{\ell}, [\xi]}_r  \Big)_{  j,\ell},   \Big( \partial^i_{\mu} \big( X^{(j)} \big)^{s, [\xi] }_r (v_{\theta(1)}, \ldots, v_{\theta(i)}) \Big)_{j,i,\theta} \bigg)  \bigg] \,dW_r \nonumber \\
&& + \sum_{j=1}^{k^{*}}  \int_s^t \bE^{(1)} \bE^{(2)} \ldots \bE^{({k^{*}})} \bigg[ \partial_{x_{j}} F_{k^{*}}  \bigg( [X^{s, \xi}_r], \Big(  \big( X^{(j)} \big)^{s, {\xi}^{(j)}}_r  \Big)_{ j}  ,  \Big(  \big( X^{(j)} \big)^{s, v_{\ell}, [\xi]}_r  \Big)_{  j,\ell},  \nonumber \\
&& \Big( \partial^i_{\mu} \big( X^{(j)} \big)^{s, [\xi] }_r (v_{\theta(1)}, \ldots, v_{\theta(i)}) \Big)_{j,i,\theta}  \bigg)  \partial_{\mu} \big( X^{(j)} \big)^{s, [\xi]}_r (v_{{k^{*}}+1}) \bigg] \,dW_r   \nonumber \\
&& + \sum_{j,\ell=1}^{k^{*}}  \int_s^t \bE^{(1)} \bE^{(2)} \ldots \bE^{({k^{*}})} \bigg[ \partial_{y_{(j,\ell)}} F_{k^{*}}  \bigg( [X^{s, \xi}_r], \Big(  \big( X^{(j)} \big)^{s, {\xi}^{(j)}}_r  \Big)_{ j}  ,  \Big(  \big( X^{(j)} \big)^{s, v_{\ell}, [\xi]}_r  \Big)_{  j,\ell},  \nonumber \\
&& \Big( \partial^i_{\mu} \big( X^{(j)} \big)^{s, [\xi] }_r (v_{\theta(1)}, \ldots, v_{\theta(i)}) \Big)_{j,i,\theta} \bigg)  \partial_{\mu} \big( X^{(j)} \big)^{s, [ \xi]}_r (v_{{k^{*}}+1})  \bigg] \,dW_r \nonumber \\
&& + \sum_{j=1}^{k^{*}} \sum_{i=1}^{{k^{*}}-1}  \sum_{\theta \in \Lambda_{i,{k^{*}}}}   \int_s^t \bE^{(1)} \bE^{(2)} \ldots \bE^{({k^{*}})} \bigg[ \partial_{z_{(j,i,\theta)}} F_{k^{*}}  \bigg( [X^{s, \xi}_r], \Big(  \big( X^{(j)} \big)^{s, {\xi}^{(j)}}_r  \Big)_{ j}  ,  \Big(  \big( X^{(j)} \big)^{s, v_{\ell}, [\xi]}_r  \Big)_{  j,\ell},  \nonumber \\
&& \Big( \partial^i_{\mu} \big( X^{(j)} \big)^{s, [\xi] }_r (v_{\theta(1)}, \ldots, v_{\theta(i)}) \Big)_{j,i,\theta} \bigg)  \partial^{i+1}_{\mu} \big( X^{(j)} \big)^{s, [\xi]}_r (v_{\theta(1)}, \ldots, v_{\theta(i)}, v_{{k^{*}}+1}) \bigg] \,dW_r   \nonumber \\
&& + \sum_{j=1}^{k^{*}} \int_s^t \bE^{(1)} \bE^{(2)} \ldots \bE^{({k^{*}})} \bigg[ \partial_{z_{(j,{k^{*}},\mathbf{I_{k^{*}}})}} F_{k^{*}}  \bigg( [X^{s, \xi}_r], \Big(  \big( X^{(j)} \big)^{s, {\xi}^{(j)}}_r  \Big)_{ j}  ,  \Big(  \big( X^{(j)} \big)^{s, v_{\ell}, [\xi]}_r  \Big)_{  j,\ell},  \nonumber \\
&& \Big( \partial^i_{\mu} \big( X^{(j)} \big)^{s, [\xi] }_r (v_{\theta(1)}, \ldots, v_{\theta(i)}) \Big)_{j,i,\theta} \bigg) \big( U_{{k^{*}}+1}^{(j)} \big) ^{s,[\xi]}_t (v_1, \ldots, v_{{k^{*}}+1}) \bigg] \,dW_r.     \label{eq: Uk+1 expression} 
\end{eqnarray}
Then we write
\begin{eqnarray}
&& \hat{\bE} \Big[ \big( U_{{k^{*}}+1}  \big) ^{s,[\xi]}_t (v_1, \ldots, v_{k^{*}}, \hat{\xi}) \hat{\eta} \Big] \nonumber \\
& = &  \int_s^t \hat{\bE} \bigg[ \bE^{(1)} \bE^{(2)} \ldots \bE^{({k^{*}})} \bE^{({k^{*}}+1)} \bigg[ \pmu F_{k^{*}} \bigg( [X^{s, \xi}_r], \Big(  \big( X^{(j)} \big)^{s, {\xi}^{(j)}}_r  \Big)_{ j} ,  \Big(  \big( X^{(j)} \big)^{s, v_{\ell}, [\xi]}_r  \Big)_{  j,\ell},  \nonumber \\
&& \Big( \partial^i_{\mu} \big( X^{(j)} \big)^{s, [\xi] }_r (v_{\theta(1)}, \ldots, v_{\theta(i)}) \Big)_{j,i,\theta}, \big( X^{({k^{*}}+1)} \big)^{s, v_{{k^{*}}+1}, [\xi]}_r  \bigg)     \bigg] \bigg|_{v_{k^{*}+1}= \hat{\xi}} \hat{\eta} \bigg] \,dW_r \nonumber \\
&& + \int_s^t \hat{\bE} \bigg[ \bE^{(1)} \bE^{(2)} \ldots \bE^{({k^{*}})} \bE^{({k^{*}}+1)} \bigg[ \pmu F_{k^{*}} \bigg( [X^{s, \xi}_r], \Big(  \big( X^{(j)} \big)^{s, {\xi}^{(j)}}_r  \Big)_{ j} ,  \Big(  \big( X^{(j)} \big)^{s, v_{\ell}, [\xi]}_r  \Big)_{  j,\ell},  \nonumber \\
&& \Big( \partial^i_{\mu} \big( X^{(j)} \big)^{s, [\xi] }_r (v_{\theta(1)}, \ldots, v_{\theta(i)}) \Big)_{j,i,\theta}, \big( X^{({k^{*}}+1)} \big)^{s,  \xi^{({k^{*}}+1)}}_r  \bigg) \partial_{\mu} \big( X^{({k^{*}}+1)} \big)^{s, [\xi]}_r   (v_{{k^{*}}+1}) \bigg] \bigg|_{v_{k^{*}+1}= \hat{\xi}} \hat{\eta}  \bigg] \,dW_r \nonumber \\
&& + \sum_{j=1}^{k^{*}} \int_s^t \hat{\bE} \bigg[  \bE^{(1)} \bE^{(2)} \ldots \bE^{({k^{*}})} \bigg[ \partial_{x_{j}} F_{k^{*}}  \bigg( [X^{t, \xi}_r], \nonumber \\
&& \Big(  \big( X^{(1)} \big)^{s, {\xi}^{(1)}}_r, \ldots, \big( X^{(j-1)} \big)^{s, {\xi}^{(j-1)}}_r, \big( X^{(j)} \big)^{s, v_{k^{*}+1}, [\xi]}_r , \big( X^{(j+1)} \big)^{s, {\xi}^{(j+1)}}_r ,  \ldots,  \big( X^{(k^{*})} \big)^{s, {\xi}^{(k^{*})}}_r \Big) ,  \nonumber \\
&& \Big(  \big( X^{(j)} \big)^{s, v_{\ell}, [\xi]}_r  \Big)_{  j,\ell},   \Big( \partial^i_{\mu} \big( X^{(j)} \big)^{s, [\xi] }_r (v_{\theta(1)}, \ldots, v_{\theta(i)}) \Big)_{j,i,\theta} \bigg)  \bigg] \bigg|_{v_{k^{*}+1}= \hat{\xi}} \hat{\eta}  \bigg] \,dW_r \nonumber \\
&& + \sum_{j=1}^{k^{*}} \int_s^t \hat{\bE} \bigg[  \bE^{(1)} \bE^{(2)} \ldots \bE^{({k^{*}})} \bigg[ \partial_{x_{j}} F_{k^{*}}  \bigg( [X^{s, \xi}_r], \Big(  \big( X^{(j)} \big)^{s, {\xi}^{(j)}}_r  \Big)_{ j}  ,  \Big(  \big( X^{(j)} \big)^{s, v_{\ell}, [\xi]}_r  \Big)_{  j,\ell},  \nonumber \\
&& \Big( \partial^i_{\mu} \big( X^{(j)} \big)^{s, [\xi] }_r (v_{\theta(1)}, \ldots, v_{\theta(i)}) \Big)_{j,i,\theta}  \bigg)  \partial_{\mu} \big( X^{(j)} \big)^{s, [\xi]}_r (v_{{k^{*}}+1}) \bigg] \bigg|_{v_{k^{*}+1}= \hat{\xi}} \hat{\eta}  \bigg] \,dW_r \nonumber \\
&& + \sum_{j,\ell=1}^{k^{*}}   \int_s^t \hat{\bE} \bigg[  \bE^{(1)} \bE^{(2)} \ldots \bE^{({k^{*}})} \bigg[ \partial_{y_{(j,\ell)}} F_{k^{*}}  \bigg( [X^{s, \xi}_r], \Big(  \big( X^{(j)} \big)^{s, {\xi}^{(j)}}_r  \Big)_{ j}  ,  \Big(  \big( X^{(j)} \big)^{s, v_{\ell}, [\xi]}_r  \Big)_{  j,\ell},  \nonumber \\
&& \Big( \partial^i_{\mu} \big( X^{(j)} \big)^{s, [\xi] }_r (v_{\theta(1)}, \ldots, v_{\theta(i)}) \Big)_{j,i,\theta} \bigg)  \partial_{\mu} \big( X^{(j)} \big)^{s, [ \xi]}_r (v_{{k^{*}}+1})  \bigg] \bigg|_{v_{k^{*}+1}= \hat{\xi}} \hat{\eta}  \bigg] \,dW_r   \nonumber \\
&& + \sum_{j=1}^{k^{*}} \sum_{i=1}^{{k^{*}}-1}  \sum_{\theta \in \Lambda_{i,{k^{*}}}}   \int_s^t \hat{\bE} \bigg[  \bE^{(1)} \bE^{(2)} \ldots \bE^{({k^{*}})} \bigg[ \partial_{z_{(j,i,\theta)}} F_{k^{*}}  \bigg( [X^{s, \xi}_r], \Big(  \big( X^{(j)} \big)^{s, {\xi}^{(j)}}_r  \Big)_{ j}  ,  \Big(  \big( X^{(j)} \big)^{s, v_{\ell}, [\xi]}_r  \Big)_{  j,\ell},  \nonumber \\
&& \Big( \partial^i_{\mu} \big( X^{(j)} \big)^{s, [\xi] }_r (v_{\theta(1)}, \ldots, v_{\theta(i)}) \Big)_{j,i,\theta} \bigg)  \partial^{i+1}_{\mu} \big( X^{(j)} \big)^{s, [\xi]}_r (v_{\theta(1)}, \ldots, v_{\theta(i)}, v_{{k^{*}}+1}) \bigg] \bigg|_{v_{k^{*}+1}= \hat{\xi}} \hat{\eta}  \bigg] \,dW_r   \nonumber \\
&& + \sum_{j=1}^{k^{*}} \int_s^t \hat{\bE} \bigg[ \bE^{(1)} \bE^{(2)} \ldots \bE^{({k^{*}})} \bigg[ \partial_{z_{(j,{k^{*}},\mathbf{I_{k^{*}}})}} F_{k^{*}}  \bigg( [X^{s, \xi}_r], \Big(  \big( X^{(j)} \big)^{s, {\xi}^{(j)}}_r  \Big)_{ j}  ,  \Big(  \big( X^{(j)} \big)^{s, v_{\ell}, [\xi]}_r  \Big)_{  j,\ell},  \nonumber \\
&& \Big( \partial^i_{\mu} \big( X^{(j)} \big)^{s, [\xi] }_r (v_{\theta(1)}, \ldots, v_{\theta(i)}) \Big)_{j,i,\theta} \bigg) \big( U_{{k^{*}}+1}^{(j)} \big) ^{s,[\xi]}_t (v_1, \ldots, v_{{k^{*}}+1}) \bigg] \bigg|_{v_{k^{*}+1}= \hat{\xi}} \hat{\eta}  \bigg] \,dW_r.   \nonumber
\end{eqnarray}
As in the proof of Theorem \ref{eq: first order regularity X}, we deduce that $ D_{\xi} \big( \partial^{{k^{*}} }_{\mu} X^{s, [\xi]} (v_1, \ldots, v_{k^{*}} ) \big)(\eta) $ satisfies the same SDE as  $\hat{\bE} \Big[ \big( U_{{k^{*}} +1}  \big) ^{s,[\xi]} (v_1, \ldots, v_{k^{*}} , \hat{\xi}) \hat{\eta} \Big]$. (Note that equality of the first and third terms follows from the same argument as  \eqref{ U directional derivative equality independence argument}   and equality of the other terms follows from the same argument as \eqref{ U directional derivative equality fubini}.) Consequently,  
\begin{equation} D_{\xi} \big( \partial^{{k^{*}} }_{\mu} X^{s, [\xi]}_t (v_1, \ldots, v_{k^{*}} ) \big)(\eta) = \hat{\bE} \Big[ \big( U_{{k^{*}} +1}  \big) ^{s,[\xi]}_t (v_1, \ldots, v_{k^{*}} , \hat{\xi}) \hat{\eta} \Big]. \label{induction identification}  \end{equation} 
By the induction hypothesis, we can again establish that  (as in the proof of Theorem \ref{eq: first order regularity X})
\begin{enumerate}[(i)]
\item $$ \bE \bigg[ \sup_{t \in [s,T]} \big| \big( U_{{k^{*}} +1}  \big)^{s,[\xi]}_t (v_1, \ldots,  v_{{k^{*}} +1}) \big|^2 \bigg] \leq C, $$
\item $$ \bE \bigg[ \sup_{t \in [s,T]} \big| \big( U_{{k^{*}} +1}  \big)^{s,[\xi]}_t  (v_1, \ldots,  v_{{k^{*}} +1})-\big( U_{{k^{*}} +1}  \big)^{s,[\xi']}_t (v'_1, \ldots, v'_{{k^{*}} +1})  \big|^2 \bigg] \leq C \bigg(  \sum_{i=1}^{{k^{*}} +1} |v_i-v'_i|^2 + W_2 ([{\xi}],[{\xi'}])^2 \bigg), $$ 
\end{enumerate}
for any $s \in [0,T]$, $v_1, \ldots, v_{{k^{*}} +1}, v'_1, \ldots, v'_{{k^{*}} +1} \in \bR$ and $\xi, \xi' \in L^2 (\cF_s)$, for some constant $C>0$. Subsequently, it follows from the same reasoning as in the proof of Theorem \ref{eq: first order regularity X} and \eqref{induction identification}  that 
$$ \partial^{{k^{*}} +1}_{\mu} X^{s, [\xi]}_t (v_1, \ldots, v_{{k^{*}} +1})  =  \big( U_{{k^{*}} +1}  \big) ^{s,[\xi]}_t (v_1, \ldots, v_{{k^{*}} +1}) . $$ Finally, by the recurrence relation \eqref{recurrence Fk} and the expression of $\big( U_{{k^{*}} +1}  \big) ^{s,[\xi]}_t $ in \eqref{eq: Uk+1 expression}, it is clear that $\partial^{{k^{*}} +1}_{\mu} X^{s, [\xi]}_t (v_1, \ldots, v_{{k^{*}} +1})$ satisfies the SDE 
\begin{eqnarray}
&& \partial^{k^{*}+1}_{\mu} X^{s, [\xi]}_t (v_1, \ldots, v_{k^{*}+1}) \nonumber \\
& = & \int_s^t \bE^{(1)} \bE^{(2)} \ldots \bE^{({k^{*}+1})} \bigg[ {F}_{k^{*}+1} \bigg( [X^{s, \xi}_r], \Big(  \big( X^{(j)} \big)^{s, {\xi}^{(j)}}_r  \Big)_{ 1 \leq j \leq {k^{*}+1}} ,  \Big(  \big( X^{(j)} \big)^{s, v_{\ell}, [\xi]}_r  \Big)_{ 1 \leq j,\ell \leq {k^{*}+1}},  \nonumber \\
&& \Big( \partial^i_{\mu} \big( X^{(j)} \big)^{s, [\xi] }_r (v_{\theta(1)}, \ldots, v_{\theta(i)}) \Big)_{\substack{1 \leq j,i \leq {k^{*}+1}\\ \theta \in \Lambda_{i,{k^{*}+1}}}} \bigg) \bigg] \,dW_r. \nonumber  
\end{eqnarray}
    \end{proof}
    \begin{corollary}
     Suppose that $\sigma$ is in $\cM_k(\cP_2(\bR))$. Then $X^{s,x,[{\xi}]} \in \cM_k(X^{s,x,[{\xi}]}). $
    \end{corollary}
    \begin{proof}
 For any multi-index $( n, \bm{\beta} )$ such that $\big| ( n, \bm{\beta} ) \big|  \leq k,$ we have an SDE representation of \\ $\partial^{n}_{\mu} X^{s,{[\xi]}}_t (v_1, \ldots, v_{n})$, by \eqref{eq: general formula kth derivative X}  in Theorem \ref{eq:purederimeasure}. By \eqref{recurrence Fk} and \eqref{eq: F1 expression}, we know that the function $F_n$ in \eqref{eq: general formula kth derivative X} is differentiable in  the spatial components for at most $k-n$ times. This is exactly what we need, since $|\bm{\beta}| = \beta_1 + \ldots + \beta_n \leq k-n$. Hence, we formally differentiate $\beta_{i}$ times with respect to each variable $v_i$, $1 \leq i \leq n$, and then use a standard Gronwall argument to establish bounds \eqref{eq:boundFrechet} and \eqref{eq:LipFrechet}. (See Theorem 5.5.3 in \cite{friedman2012stochastic} or  Proposition 4.10 in \cite{krylov1999kolmogorov} for details.) 
\end{proof}
We are now in a position to prove Theorem \ref{eq:generalisationmainresult}, via the smoothness of $\sigma $ and  $X^{s,x,[{\xi}]}$.
\begin{proof}[Proof of Theorem \ref{eq:generalisationmainresult}] By combining \eqref{eq:formaldifXs,xi}, \eqref{eq:relationYandU}, \eqref{U equals pmu} and \eqref{no dependence on x}, we deduce that 
$$ \chi: L^2 ( \cF_s) \to L^2( \cF_t); \quad \quad \xi \mapsto X^{s,\xi}_t  $$ 
is Fr\'{e}chet differentiable with Fr\'{e}chet derivative given by
$$ D\chi( \xi)(\eta) =   \eta + \hat{\bE} \big[  \pmu X^{s,[{\xi}]}_t  (\hat{\xi})   \hat{\eta} \big].$$ 
Next, for any fixed $s \in [0,t]$, we define the lifts  $\widetilde{\Phi}: L^2 ( \cF_t) \to \bR$ and $\widetilde{\cV}(s, \cdot): L^2 ( \cF_s) \to \bR$ for functions $\Phi $ and $\cV(s, \cdot)$ respectively, given by
$$\widetilde{\Phi} ( \theta_1) = f( [{\theta_1}]), \quad \widetilde{\cV}(s, \theta_2) = h(s, [{\theta_2}]), \quad \text{for} \quad  \theta_1 \in L^2 ( \cF_t), \quad \theta_2 \in L^2 ( \cF_s).$$
Then, we notice from equation \eqref{eq: defofflow} that 
$$ \widetilde{\cV}(s, \cdot) = \widetilde{\Phi} \circ \chi.$$ 
By the chain rule of Fr\'{e}chet differentiation, we obtain that
$$ D\widetilde{\cV}(s,  \xi) = D\widetilde{\Phi}( \chi(\xi)) \circ D{\chi}(\xi),$$ 
which implies that
\begin{eqnarray}
  D\widetilde{\cV}(s,  \xi)(\eta) & = &  D\widetilde{\Phi}( \chi(\xi))  \big( D{\chi}(\xi) (\eta) \big) \nonumber \\
  & = & \bE  \big[ \pmu \Phi \big([{X^{s, \xi}_t}], X^{s, \xi}_t \big) D{\chi}(\xi) (\eta) \big] \nonumber \\
  &= & \bE  \big[ \pmu \Phi \big([{X^{s, \xi}_t}], X^{s, \xi}_t \big) \big( \eta + \hat{\bE} \big[  \pmu X^{s,[{\xi}]}_t  (\hat{\xi})   \hat{\eta} \big] \big) \big], \label{eq:Frechetfirstorderh}
\end{eqnarray}
for any $\xi, \, \eta \in \cF_s$.
Note that the first term can be rewritten as
\begin{equation}
\bE  \big[ \pmu \Phi \big([{X^{s, \xi}_t}], X^{s, \xi}_t \big) \eta  \big]  =  \bE \big[ \bE ( \pmu \Phi  \big([{X^{s, \xi}_t}], X^{s, x, [{\xi}]}_t \big) \big) \big|_{x=\xi} \eta \big]  \label{eq:Frechetfirstorderh1}
\end{equation}
and the second term can be rewritten by the Fubini's theorem as
\begin{equation}
    \bE  \big[ \pmu \Phi  \big([{X^{s, \xi}_t}], X^{s, \xi}_t \big)  \hat{\bE} \big[ \big( \pmu X^{s, [{\xi}]}_t  (\hat{\xi}) \big)\hat{\eta} \big]  \big] = \hat{\bE} \big[ \bE \big[ \pmu \Phi  \big([{X^{s, \xi}_t}], X^{s, \xi}_t \big) \pmu X^{s,[{\xi}]}_t  (\hat{\xi})  \big] \hat{\eta} \big]. \label{eq:Frechetfirstorderh2} 
\end{equation}
Consequently, by combining \eqref{eq:Frechetfirstorderh1} and \eqref{eq:Frechetfirstorderh2}, equation \eqref{eq:Frechetfirstorderh} becomes 
$$ D\widetilde{\cV}(s,  \xi)(\eta) = \bE \big[ \bE ( \pmu \Phi  \big([{X^{s, \xi}_t}], X^{s, x, [{\xi}]}_t \big) \big) \big|_{x=\xi} \eta \big] +  \hat{\bE} \big[ \bE \big[ \pmu \Phi  \big([{X^{s, \xi}_t}], X^{s, \xi}_t \big) \pmu X^{s,[{\xi}]}_t  (\hat{\xi})  \big] \hat{\eta} \big], $$ which implies that
$$ \pmu \cV(s, [{\xi}])(y)=  \bE \Big[\pmu \Phi  \big([{X^{s, \xi}_t}], X^{s, y, [{\xi}]}_t  \big) + \pmu \Phi  \big([{X^{s, \xi}_t}], X^{s, \xi}_t \big) \pmu X^{s,[{\xi}]}_t  (y)    \Big], \quad \quad y \in \bR. $$ 
By our assumption, we know that $\pmu \Phi$ satisfies \eqref{eq:boundf} and \eqref{eq:Lipf}, and  the process $ \pmu X^{s,[{\xi}]}_t (v_1)$ satisfies \eqref{eq:boundFrechet} and \eqref{eq:LipFrechet}. It follows that $\pmu \cV$ also satisfies \eqref{eq:boundf} and \eqref{eq:Lipf}, with the constant bound $C$ uniform in time. 

By iterating this procedure, we can show that for any multi-index $(n, \bm{\beta})$ such that $|(n, \bm{\beta})| \leq k$, $D^{(n, \bm{\beta})} \cV(s, \mu)(v_1, \ldots, v_{n})$ can be computed explicitly as above and can be represented in terms of derivatives in the form $D^{(n',\bm{\beta'})} X^{s,[\xi]}_t (v'_1, \ldots, v'_{n'})$ and $D^{(n'', \bm{\beta''})} \Phi(\mu) (v''_1, \ldots, v''_{n''}) $, for some $n', n''  \in \bN \cup \{0 \}$, $\bm{\beta'} \in \big( \bN \cup \{0 \} \big)^{n'}$ and   $\bm{\beta''} \in \big( \bN \cup \{0 \} \big)^{n''}$, such that $|(n', \bm{\beta'})| \leq k$ and $|(n'', \bm{\beta''})| \leq k$. The facts that $X^{s,x,[{\xi}]} \in \cM_k(X^{s,x,[{\xi}]}) $ and $\Phi \in \cM_k$ also allow us to deduce that $D^{(n,\bm{\beta})} \cV(s, \mu)(v_1, \ldots, v_{n})$ satisfies estimates \eqref{eq:boundf} and \eqref{eq:Lipf}, with the constant bound $C$ uniform in time. Finally, we know from Theorem 7.2 in \cite{buckdahn2017mean} (which corresponds to Theorem \ref{eq:generalisationmainresult} with $k=2$) that $\cV( \cdot, \mu) \in C^1((0,t)), $  for every $\mu \in \mathcal{P}_2(\bR)$. Therefore, we conclude that $\cV \in \cM_k$.
\end{proof}
\bibliographystyle{plain}
\bibliography{Particles(expansion)}

\end{document}